\documentclass[a4paper, reqno, twoside]{amsart}
\usepackage{etex}
\usepackage{amsfonts,amssymb,graphics,amsthm,amsmath}
\usepackage{latexsym}
\usepackage[2emode]{psfrag}
\usepackage{mathrsfs}
\usepackage[all]{xy}
\usepackage{enumerate}
\usepackage[utf8]{inputenc}
\usepackage{lscape}
\usepackage{a4wide}
\usepackage[bookmarks=false, backref=page]{hyperref}

\usepackage{txfonts}

\usepackage[all]{pstricks}
\usepackage{pst-text,pst-node,pst-coil}
\usepackage{tikz}
\usepackage{mathtools}

\usepackage{cancel}
\usepackage[normalem]{ulem}
\usetikzlibrary{arrows,shapes,backgrounds}

\usepgflibrary{arrows}
\usetikzlibrary{topaths}
\usetikzlibrary{er}

\usepackage{multicol}
\usepackage{tikz-cd}

\newtheorem{theorem}{\sc Theorem}[section]
\newtheorem{proposition}[theorem]{\sc Proposition}
\newtheorem{propositiondef}[theorem]{\sc Proposition and Definition}

\newtheorem{lemma}[theorem]{\sc Lemma}
\newtheorem{corollary}[theorem]{\sc Corollary}
\theoremstyle{definition}
\newtheorem{definition}[theorem]{\sc Definition}
\newtheorem{definitions}[theorem]{\sc Definitions}
\newtheorem{example}[theorem]{\sc Example}
\newtheorem{examples}[theorem]{\sc Examples}

\theoremstyle{remark}
\newtheorem{remark}[theorem]{\sc Remark}

\newcommand{\tensor}[1]{\otimes_{\scriptscriptstyle{#1}}}

\newcommand{\Sf}[1]{\mathsf{#1}}
\newcommand{\fk}[1]{\mathfrak{#1}}

\renewcommand{\hom}[3]{\mathrm{Hom}_{\Sscript{#1}}\left(#2,\,#3\right)}

\newcommand{\bd}[1]{\boldsymbol{#1}}

\newcommand{\lar}[1]{\langle #1 \rangle}

\newcommand{\LR}[1]{\left\{\underset{}{} #1 \right\}}

\newcommand{\Go}{\cG_{\Sscript{0}}}
\newcommand{\Ga}{\cG_{\Sscript{1}}}
\newcommand{\Ho}{\cH_{\Sscript{0}}}
\newcommand{\Ha}{\cH_{\Sscript{1}}}
\newcommand{\Ko}{\cK_{\Sscript{0}}}
\newcommand{\Ka}{\cK_{\Sscript{1}}}

\newcommand{\Io}{\cI_{\Sscript{0}}}
\newcommand{\Ia}{\cI_{\Sscript{1}}}

\newcommand{\Ad}[1]{{\bd{ad}}_{\Sscript{g}}}
\newcommand{\Gop}{\cG^{\Sscript{op}}}

\newcommand{\simc}{\sim_{\Sscript{\rm{C}}}}
\newcommand{\SubG}{\text{\large $\mathcal{S}$}_{\Sscript{\cG}}}

\newcommand{\SubGx}[1]{\text{\large $\mathcal{S}$}_{\Sscript{\cG^{\langle #1\rangle}}}}

\newcommand{\tauk}{\tau_{\Sscript{\cK}}}
\newcommand{\tauh}{\tau_{\Sscript{\cH}}}

\newcommand{\bB}{\mathscr{B}}

\newcommand{\gG}{\mathscr{G}}

\newcommand{\lL}{\mathscr{L}}

\newcommand{\cA}{{\mathcal A}}
\newcommand{\cB}{{\mathcal B}}

\newcommand{\cF}{{\mathcal F}}
\newcommand{\cG}{{\mathcal G}}
\newcommand{\cH}{{\mathcal H}}
\newcommand{\cI}{{\mathcal I}}

\newcommand{\cK}{{\mathcal K}}

\newcommand{\cR}{{\mathcal R}}

\newcommand{\cX}{{\mathcal X}}

\newcommand{\Sets}{\mathsf{Sets}}

\newcommand{\rGsets}{{\emph{Sets}}\text{-}\cG}
\newcommand{\frGsets}{{\emph{sets}}\text{-}\cG}

\newcommand{\Sscript}[1]{\scriptscriptstyle{#1}}
\newcommand{\due}[3]{{}_{{#2 }} {#1}_{{ #3}}\,}

\newcommand{\Stab}[2]{\Stabi_{ \scriptscriptstyle{\mathcal{#2}}}\left({#1}\right)}
\newcommand{\Orbit}[2]{\Orb_{ \scriptscriptstyle{\mathcal{#2}}}\left({#1}\right)}
\newcommand{\rhaction}{\scriptsize{\leftharpoonup}}
\newcommand{\lgaction}{\scriptsize{\rightharpoonup}}
\newcommand{\rgaction}{\scriptsize{\leftharpoondown}}
\newcommand{\lhaction}{\scriptsize{\rightharpoondown}}

\newcommand{\Hlcoset}[1]{\cH[#1]}
\newcommand{\Glcoset}[1]{\cG[#1]}
\newcommand{\Hrcoset}[1]{[#1]\cH}

\newcommand{\LG}[1]{\mathscr{L}(\mathcal{#1})}
\newcommand{\GHom}[2]{{\rm Hom}_{\Sscript{\rGsets}}(#1,#2)}

\usepackage{braket, upgreek}
\usepackage{todonotes}
\usepackage{tcolorbox}

\DeclareMathOperator{\Id}{Id}
\DeclareMathOperator{\pr}{pr}
\DeclareMathOperator{\Stabi}{\emph{Stab}}
\DeclareMathOperator{\Orb}{\emph{Orb}}

\DeclareMathOperator{\rep}{\emph{rep}}

\DeclareMathOperator{\Grou}{\mathbf{Grpd}}

\DeclareRobustCommand{\Cat}[1]{\mathcal{#1}}

\DeclareRobustCommand{\Rset}[1]{\emph{Sets}\text{-}{#1}}
\DeclareRobustCommand{\rset}[1]{\emph{sets}\text{-}{#1}}

\DeclareRobustCommand{\Uplus}{\mathbin{\scalebox{1.50}{\ensuremath{\uplus}}}}

\DeclareRobustCommand{\fpro}[1]{\underset{#1}{\times}}

\DeclareRobustCommand{\trg}[2]{\mathcal{G}_{{#1},\, {#2}}}

\DeclareMathOperator{\gRo}{\gG}
\DeclareRobustCommand{\gro}{\gRo}

\begin{document}
\allowdisplaybreaks

\title[On Burnside Theory for groupoids]{On Burnside Theory for groupoids}

\author{Laiachi El Kaoutit}
\address{Universidad de Granada, Departamento de \'{A}lgebra. Facultad de Ciencias, Fuente Nueva s/n. E-18071, Granada. Spain}
\email{kaoutit@ugr.es}
\urladdr{\url{http://www.ugr.es/~kaoutit/}}

\author{Leonardo Spinosa}
\address{University of Ferrara, Department of Mathematics and Computer Science.
Via Machiavelli 30, Ferrara, I-44121, Italy}
\email{leonardo.spinosa@unife.it}
\urladdr{\url{https://orcid.org/0000-0003-3220-6479}}

\date{\today}
%\thanks{????}

\begin{abstract}
We explore the concept of \emph{conjugation} between  subgroupoids, providing several  characterizations of the conjugacy relation (Theorem \ref{thm:A} in \S \ref{ssec:DR}). We show that two finite groupoid-sets, over a locally strongly finite groupoid,  are isomorphic, if and only if, they have the same number of fixed points with respect to any subgroupoid with a single object (Theorem \ref{thm:B} in \S \ref{ssec:DR}).
Lastly, we examine the ghost map of a finite groupoid and the idempotents elements of its Burnside algebra. The exposition includes  an Appendix where we gather  the main general technical notions  that are needed along the paper. 
\end{abstract}

\thanks{Research supported by the Spanish Ministerio de Econom\'{\i}a y Competitividad  and FEDER, grant MTM2016-77033-P. The work of Leonardo Spinosa was partially supported by   the  ``National Group for Algebraic and Geometric Structures, and their Applications'' (GNSAGA– INdAM)}
\maketitle
\vspace{-0.8cm}
\begin{small}
\tableofcontents
\end{small}

\pagestyle{headings}

\vspace{-1.2cm}

\section{Introduction}\label{sec:Intro}

We will describe the motivations behind our work, how the classical Burnside Theory   fits into the contemporary mathematical framework and we will delineate some  paradigms where this research fits into. Thereafter, we will reproduce our main results in sufficient details, aiming to make this  summary self-contained.

\subsection{Motivations and overview}\label{ssec:MO}
The Burnside theory is a classical part of the representation theory of finite groups and its first introduction has been done by Burnside in \cite{Burnside:1911}. Subsequently, other work has been realized in this direction: see, for example, \cite{DressSolvGrps} and \cite{Solomon:1967} among others. 
Apparently, there are two interrelated aspects of this theory. One of them is the well known  Burnside Theorem that codifies some basic  combinatorial properties of the lattice of subgroups of a given finite group, providing for instance its table of marks\footnote{As we will see here, this lattice can be viewed as a category whose arrows are equivariant maps between cosets. An entry in the table of marks is nothing but the number of arrows between two objects in this category. On the other hand, it is noteworthy to see Remark \ref{rem:contraejemplo} below, for some new observations on the classical Burnside Theorem for groups.}. The other aspect is the construction of the Burnside ring over the integers and its extension algebra over the rational numbers. Known is the fact that  two conjugated subgroups lead, via the isomorphism between their cosets,  to the same element in the Burnside ring.

It seems that, years after its discovery, Burnside ring has became a very powerful tool in different branches of pure mathematics. For instance,  in certain equivariant stable homotopies (e.g., that of the sphere in dimension zero \cite{Segal:1970}), the influence of the Burnside ring is conspicuously present so that, in particular, stable equivariant homotopy groups are modules over the Burnside ring (see \cite{TomDieck:1979} for further details). There are in the literature other more sophisticated versions of the Burnside ring of a finite group, like the ones  introduced and studied in \cite{Hartman/Yalcin:2007, Diaz/Libman:2009, GunnellsEtall}. 

From a categorical point of view, the Burnside ring can be constructed, with the help of the Grothendieck functor, from any skeletally small category with initial  object and finite co-products, which possesses a symmetric monoidal structure compatible with this co-product (known as \emph{Laplaza categories}, see the Appendices). Undoubtedly,  two such categories, that are connected by a convenient symmetric monoidal equivalence, they have,  up to  a possibly non canonical isomorphism, the same Burnside ring.  
A special case is when  the starting monoidal category is a certain category of representations over a specific object: a group, a groupoid, a 2-group, a 2-groupoid, etc. In this case one is  tempted to use this ring in order to decipher part of the structure of the handled object (for instance, it was proved in \cite{DressSolvGrps} that a group is solvable if and only if its Burnside ring is connected). Namely, this was perhaps the origin and the motivation behind  the classical Burnside theory for (abstract) groups described above.

Groupoids are natural generalization of groups and prove to be useful in different branches of mathematics, see  \cite{Brown:1987}, \cite{Cartier:2008} and \cite{WeiGrpdUnInExtSym} (and the references therein) for a brief survey. Specifically, a groupoid is defined as a small category whose every morphism is an isomorphism and can be thought as a ``group with many objects''.
In the same way, a group can be seen as a groupoid with only one object.
As explained in \cite{Brown:1987}, while a groupoid in its very simple facet can be seen as a disjoint union of groups, this forces unnatural choices of base points and obscures the overall structure of the situation.
Besides, even under this simplicity,  structured groupoids, like topological or differential groupoids, cannot even be thought like a disjoint union of topological or differential groups, respectively. Different specialists realized (see  for instance \cite{Brown:1987} and \cite[page 6-7]{Connes:1994}) in fact that the passage from groups to groupoids is not  a trivial research and have its own difficulties and  challenges.
Thus extending a certain well known result in groups context to the framework of groupoids is not an easy task and there are often difficulties to overcome.  

The research of this paper fits in this paradigm: our main aim here is to reproduce some classical results of the Burnside theory of groups, like the so called Burnside theorem and the existence of a ghost map,  in the groupoid context.
In order to do so, we will explore in depth, with several example and counterexamples, the concept of \emph{conjugation} of two subgroupoids of a given groupoid, illustrating instances of new phenomenons that are not present in the group context.
For example, unlike the classical case of groups or that of disjoint union of groups, there can be two subgroupoids which are conjugated without being isomorphic (see Example \ref{exam:Nconj} below).  Concerning the Burnside ring of a groupoid, it is observed that this ring  can be decomposed, although in  non canonical way,  as a direct product of the Burnside rings of its isotropy groups type: exactly one for each connected component (see \cite{Kaoutit/Spinosa:2019} for another approach to this ring).

In the appendix we will briefly recall some useful notions about monoidal structures and the concepts of ``rig''	 and Grothendieck functor.
We define a rig as a ring without negatives that is, without the inverses of the addition.
The Grothendieck functor enables us to ``add'' the additive inverses to a rig to obtain a ring.
It's exactly in this way that the ring of integers \(\mathbb{Z}\) is constructed from the natural numbers \(\mathbb{N}\), the quintessential example of rig.

\subsection{Description of the main results}\label{ssec:DR}
We proceed in describing with full details our main results. For a given groupoid $\cG=(\Ga,\Go)$ ($\Go$ is the set of objects and $\Ga$ the set of arrows), we denote by  $\cG(a,a')$  the set of arrows from $a$ to $a'$.  A subgroupoid $\cH$ of $\cG$  is a subcategory of the underlying category $\cG$, which is non empty and stable under inverses.  The notation $\big(\cG/\cH\big)^{\Sscript{\mathsf{R}}}$ (resp. $\big(\cG/\cH\big)^{\Sscript{\mathsf{L}}}$) stands for the set of right (resp. left) coset \cite{Kaoutit/Spinosa:2018} (see subsection \ref{ssec:conj} for the  precise definition).

The criteria of conjugacy between subgroupoids is given by the following first result, stated below as Theorem \ref{thm:LR}, and therein we refer to Definition \ref{def:Gset} for the precise notion of (right or left) $\cG$-set. 

{\renewcommand{\thetheorem}{{\bf A}}
\begin{theorem}\label{thm:A}
Let $\cH$ and $\cK$ be two subgroupoids of a given groupoid $\cG$ with canonical morphisms $\tauh: \cH \hookrightarrow \cG \hookleftarrow \cK: \tauk$. 
Then the following conditions are equivalent:
\begin{enumerate}[(i)]
\item $\big(\cG/\cH\big)^{\Sscript{\mathsf{R}}} \, \cong\,  \big(\cG/\cK\big)^{\Sscript{\mathsf{R}}}$  as right $\cG$-sets; 
\item There are morphisms of groupoids  $F: \cK \rightarrow \cH$ and $G: \cH  \to \cK$ together with two  natural transformations $\fk{g}: \tauh F \to \tauk$ and $\fk{f}: \tauk G \to \tauh$.
\item The subgroupoids $\cH$ and $\cK$ are conjugally equivalent (see Definition \ref{def:conj0}).
\item There are families \(\left(u_{\Sscript{b}}\right)_{b \in \Cat{K}_{ \scriptscriptstyle{0} }}\) and \(\left(g_{\Sscript{b}}\right)_{b \in \Cat{K}_{ \scriptscriptstyle{0} }}\) with \(u_{\Sscript{b}} \in \Cat{H}_{ \scriptscriptstyle{0} }\) and \(g_{\Sscript{b}} \in \Cat{G}\left({u_{\Sscript{b}}, b}\right)\) for every \(b \in \Cat{K}_{ \scriptscriptstyle{0} }\), such that:
\begin{enumerate}[(a)]
\item for each \(b_1, b_2 \in \Cat{K}_{ \scriptscriptstyle{0} }\) we have \(g_{\Sscript{b_2}}^{-1} \Cat{K}\left({b_1, b_2}\right) g_{\Sscript{b_1}} = \Cat{H}\left({u_{\Sscript{b_1}}, u_{\Sscript{b_2}}}\right)\);
\item for each \(u \in \Cat{H}_{ \scriptscriptstyle{0} }\) there is \(z \in \Cat{K}_{ \scriptscriptstyle{0} }\) such that \(\Cat{H}\left({u_z, u}\right) \neq \emptyset\).
\end{enumerate}
\item  $\big(\cG/\cH\big)^{\Sscript{\mathsf{L}}} \, \cong\,  \big(\cG/\cK\big)^{\Sscript{\mathsf{L}}} $ as left $\cG$-sets.
\end{enumerate}
\end{theorem}
}
As we mentioned above, in Example \ref{exam:Nconj} we show that, in contrast with the case of disjoint union of group, it could happen that two subgroupoids are conjugated without being isomorphic. In Example \ref{exam:IsotropyNotConj} we show that there is a groupoid with two conjugated subgroupoids such that not each isotropy group of the first subgroupoid is conjugated to each isotropy group of the second subgroupoid. Both Examples make manifest  the complexity of the study of the ``poset'' of all subgroupoids using the conjugacy relation.   

The subsequent result is our version of the classical Burnside Theorem. Given a right $\cG$-set $(X,\varsigma)$ and a subgroupoid $\cH$ of $\cG$ with a single object, the symbol $X^{\Sscript{\cH}}$ denotes the subset of the $\cH$-invariant elements of $X$, which can be identified with the set of all $\cG$-equivariant maps from the set of right  cosets $\cG/\cH$ to $X$.   

{\renewcommand{\thetheorem}{{\bf B}}
\begin{theorem}[Burnside Theorem]\label{thm:B}
Let $\cG$ be a locally strongly finite groupoid (Definition \ref{def:finiteG}). Consider  two  finite right \(\mathcal{G}\)-sets \(\left({X, \varsigma}\right)\) and \(\left({Y, \vartheta}\right)\). Then the following statements are equivalent.
\begin{enumerate}
\item The right \(\Cat{G}\)-sets \(\left({X, \varsigma}\right)\) and \(\left({Y, \vartheta}\right)\) are isomorphic.
\item For each subgroupoid \(\mathcal{H}\) of \(\mathcal{G}\) with a single object, we have that 	
\[
\left\lvert{X^{\Sscript{\cH} }}\right\rvert= \left\lvert{Y^{\Sscript{\cH}}}\right\rvert.
\]
\end{enumerate}
In particular, this applies to any strongly finite groupoid. 
\end{theorem}
}
The proof of Theorem \ref{thm:B} is based on the fact that the table of marks, or the matrix of marks, of $\cG$ can be shown to be a diagonal block matrix, where each block is a lower triangular matrix, which corresponds to  the matrix of marks of an isotropy type group (see Proposition \ref{prop:tableofmarks}).

Now, we turn our attention to the Burnside ring. Given a groupoid $\cG$ with finitely many objects, we fix a set $\rep(\Go)$ of representative objects modulo the regular action of $\cG$ over itself, by using either the source or the target. The \emph{Burnside ring of $\cG$} is, by definition, $\bB(\cG)=\gG\lL(\cG)$ (this  is a $\mathbb{Z}$-algebra), where $\lL(\cG)$ is the Burnside rig of $\cG$ constructed from the category of right $\cG$-sets  with underlying finite sets, and $\gG$ is the Grothendieck functor (see Appendices  \ref{ssec:Rig}  and \ref{ssec:GrothF}).

Given a groupoid \(\mathcal{G}\) with a finite set of objects, we have the following isomorphism of rings:
\[ \bB\left({\mathcal{G}}\right) \cong \prod_{a \, \in \, \rep(\Go)} \bB\left({\mathcal{G}^{ \scriptscriptstyle{a} }  }\right),
\]
where the right hand side term is the product of commutative rings and  each of the $\bB\left({\mathcal{G}^{ \scriptscriptstyle{a} }}\right)$'s is the Burnside ring of the isotropy group $\cG^{\Sscript{a}}$ of the representative object $a$. 
The stated isomorphism depends on a given choice of a set of representative, that is, the decomposition is not unique. This is in fact due to the non canonical equivalence of categories between the underlying categories of the groupoids $\cG$ and $\biguplus_{\Sscript{a \, \in \, \rep(\Go)}} \cG^{\Sscript{a}}$.

%\vspace{2cm}
%
%The following is a corollary of  Theorem \ref{tBurnTransCompProd} and Proposition \ref{pTransGrpdBurnRing}, stated as Corollary \ref{coro:main} below, gives the decomposition of $\bB(\cG)$ into a finite product of Burnside rings of groups, although, in a non canonical way: 
%
%
%{\renewcommand{\thetheorem}{{\bf C}}
%\begin{corollary}\label{coro:C}
%Given a groupoid \(\mathcal{G}\) with a finite set of objects, we have the following isomorphism of rings:
%\[ \bB\left({\mathcal{G}}\right) \cong \prod_{a \, \in \, \rep(\Go)} \bB\left({\mathcal{G}^{ \scriptscriptstyle{a} }  }\right),
%\]
%where the right hand side term is the product of commutative rings and  each of the $\bB\left({\mathcal{G}^{ \scriptscriptstyle{a} }}\right)$'s is the Burnside ring of the isotropy group $\cG^{\Sscript{a}}$ of the representative object $a$. 
%\end{corollary}
%}
%
%As one can see from this Corollary, 

Lastly, in analogy with finite group theory, one can also introduce, in the context of groupoids, the ghost map and show that it is injective as in the classical case. The idempotent of the $\mathbb{Q}$-algebra $\mathbb{Q}\tensor{\mathbb{Z}}\bB(\cG)$ are then computed by means of those of each component of the previous decomposition. All this results are  explicitly expounded in Section \ref{sec:Idem} below.

\section{Abstract groupoids: definitions, basic properties and examples}\label{sec:Grpd}
The material of this section, which will be used throughout the paper,   is somehow considered folklore and most of its content  can be found in \cite{Kaoutit/Kowalzig:14, ElKaoutit:2017} and \cite{Kaoutit/Spinosa:2018}. However, for the sake of completeness and for the convenience of the reader,  we are going to illustrate the basic notions, as well as some motivating examples, of the groupoid theory.

\subsection{Notations, basic notions and examples}\label{ssec:basic}
A \emph{groupoid}  is a small category where each morphism is an isomorphism. That is, a pair of two sets $\cG:=(\cG_{\Sscript{1}}, \cG_{\Sscript{0}})$ with diagram of sets
$$
\xymatrix@C=35pt{\cG_{\Sscript{1}}\ar@<0.70ex>@{->}|-{\scriptscriptstyle{{\Sf{s}}}}[r] \ar@<-0.70ex>@{->}|-{\scriptscriptstyle{{\Sf{t}}}}[r] & \ar@{->}|-{ \scriptscriptstyle{\iota}}[l] \cG_{\Sscript{0}}},
$$
where $\Sf{s}$ and $\Sf{t}$ are resp.~ the source and the target of a given arrow, and $\iota$ assigns to each object its identity arrow; together with an associative and unital multiplication  $\cG_{\Sscript{2}}:= \cG_{\Sscript{1}}\, \due  \times {\Sscript{{\Sf{s}}}} {\, \Sscript{{\Sf{t}}}} \, \cG_{\Sscript{1}} \to \cG_{\Sscript{1}}$  as well as a map $\cG_{\Sscript{1}} \to \cG_{\Sscript{1}}$ which associated to each arrow its inverse. Notice, that $\iota$ is an injective map, and so $\Go$ is identified with a subset of $\Ga$. 
A groupoid is then a small category with more structure, namely, the map which sends  each arrow to its inverse. We implicitly identify a groupoid with its underlying category. Interchanging the source and the target will lead to \emph{the opposite groupoid} which we denote by $\Gop$.

Given a groupoid $\cG$, consider two objects $x, y \in \cG_{\Sscript{0}}$: we denote by $\cG(x,y)$ the set of all arrows with source $x$ and target $y$.  \emph{The isotropy group of $\cG$ at $x$} is then the \emph{group of loops}:
\begin{equation}\label{Eq:isotropy}
\cG^{\Sscript{x}}:=\cG(x,x)\,=\,\Big\{ g \in \cG_{\Sscript{1}}|\, \Sf{t}(g)=\Sf{s}(g)=x \Big\}.
\end{equation}
Clearly each of the sets $\cG(x,y)$ is, by the groupoid multiplication, a left $\cG^{\Sscript{y}}$-set and right $\cG^{\Sscript{x}}$-set. In fact, each of the $\cG(x,y)$ sets is a $(\cG^{\Sscript{y}}, \cG^{\Sscript{x}})$-biset,  in the sense of \cite{Bouc:2010}. 
Two objects $x, x' \in \Go$ are said to be equivalent if and only if there is an arrow connecting them.  This in fact defines an equivalence relation whose quotient set is the set of all \emph{connected components} of $\cG$, which we denote by $\pi_{\Sscript{0}}(\cG):=\Go/\cG$.

Given a set $I$ and a family of groupoids $\{\cG^{\Sscript{(i)}}\}_{i \, \in \, I}$,  \emph{the coproduct groupoid} is a groupoid denoted by $\cG= \coprod_{i \, \in \, I} \cG^{\Sscript{(i)}}$ and defined by
$$
\Go\,=\, \underset{i \, \in \, I}{\biguplus} \, \cG^{\Sscript{(i)}}_{\Sscript{0}}, \qquad \cG(x,y)\,=\, \begin{cases} \cG^{\Sscript{(i)}}(x,y), & \text{ if }\; \exists \,  i\, \in I \text{ such that } x,y \in \cG^{\Sscript{(i)}}_{\Sscript{0}}  \\  \emptyset, & \text{otherwise}. \end{cases}
$$

\begin{definitions}\label{def:IsotropyConj}
Let $\cG$ be a groupoid. 
\begin{enumerate}
\item We say that $\cH$ is a \emph{subgroupoid of $\cG$}, provided that $\cH$ is a subcategory of the underlying category of $\cG$, which is also stable under the inverse map, that is, $h^{-1} \in \Ha$, for every $h \in \Ha$. For instance, any connected component of $\cG$ is a subgroupoid. On the other hand, a subgroup $H$ of an isotropy group $\cG^{\Sscript{x}}$, for an object $x \in \Go$, can be considered as a \emph{subgroupoid with only one object} of $\cG$. Conversely, any subgroupoid of $\cG$ with one object is of this form. 
\item Two isotropy groups $\cG^{\Sscript{x}}$ and $\cG^{\Sscript{x'}}$   are said to be \emph{conjugated} when there exists $g \in \cG(x,x')$ such that $\cG^{\Sscript{x}}\,=\, g^{-1}\cG^{\Sscript{x'}}g$ (equality as subsets of $\Ga$).    
\item $\cG$ is said to be \emph{transitive} (or \emph{connected}) if for any pair of objects $(x,x')$ we have $\cG(x,x') \neq \emptyset$; equivalently, the map $(\Sf{s}, \Sf{t}): \Ga \to \Go \times \Go$ is surjective.  In general, any groupoid can be seen as a coproduct of transitive groupoids: namely, its connected components.   
\end{enumerate}
\end{definitions}

A \emph{morphism of groupoids} $\phiup: \cH \to \cG$ is a functor between the underlying categories.  
Thus $\phi=(\phi_{\Sscript{0}}, \phi_{\Sscript{1}}): (\Ho, \Ha) \to (\Go, \Ga)$ is a pair of maps  compatible with multiplication and identity maps, and interchange the sources and the targets. In other words, the following diagram is commutative
$$
\xymatrix@C=35pt@R=15pt{\cH_{\Sscript{1}}\ar@<0.70ex>@{->}|-{\scriptscriptstyle{{\Sf{s}}}}[r] \ar@<-0.70ex>@{->}|-{\scriptscriptstyle{{\Sf{t}}}}[r] \ar@{->}_-{\phi_{\Sscript{1}}}[d] & \ar@{->}|-{ \scriptscriptstyle{\iota}}[l] \cH_{\Sscript{0}} \ar@{->}^-{\phi_{\Sscript{0}}}[d]   \\ \cG_{\Sscript{1}}\ar@<0.70ex>@{->}|-{\scriptscriptstyle{{\Sf{s}}}}[r] \ar@<-0.70ex>@{->}|-{\scriptscriptstyle{{\Sf{t}}}}[r] & \ar@{->}|-{ \scriptscriptstyle{\iota}}[l] \cG_{\Sscript{0}}. }
$$
Clearly any subgroupoid $\cH$ of $\cG$ induces a morphism $\tauup: \cH \hookrightarrow \cG$ of groupoids whose both maps $\tauup_{\Sscript{0}}$ and $\tauup_{\Sscript{1}}$ are injectives. On the other hand, it is obvious that any morphism $\phi$  induces  homomorphisms of groups between the isotropy groups: $\phiup^{\Sscript{y}}: \cH^{\Sscript{y}}  \to \cG^{\Sscript{\phiup_0(y)}}$, for every $y \in \cH_{\Sscript{0}}$. In order to illustrate the foregoing notions,  we quote here  some standard examples of groupoids and their morphisms.

\begin{example}[Trivial groupoid and product of groupoids]\label{exam:trivial}
Let $X$ be a set. Then the pair $(X, X)$ is obviously a groupoid (in fact a small discrete category, i.e., with only identities as arrows) with trivial structure. This is known as \emph{the trivial groupoid}. Note that, with this definition, the empty groupoid is the trivial groupoid $(\emptyset, \emptyset)$ which, by convention, is also considered as a  transitive groupoid.

The \emph{product $\cG \times \cH$ of two groupoids} $\cG$ and $\cH$ is the groupoid whose sets of objects and arrows, are respectively, the  Cartesian products  $\Go \times \Ho$ and  $\Ga \times \Ha$. The multiplication, inverse and unit arrow are canonically given as follows:
$$
(g,h) \, (g', h') \,=\, (gg', hh'), \quad 	\quad (g,h)^{-1}\,=\, (g^{-1}, h^{-1}), \quad \iota_{\Sscript{(x,\, u)}}\,=\, (\iota_{\Sscript{x}}, \iota_{\Sscript{u}}).
$$
\end{example}

\begin{example}[Action groupoid]\label{exam:action}
Any group $G$ can be considered as a groupoid by taking $G_{\Sscript{1}}=G$ and $G_{\Sscript{0}}=\{*\}$ (a set with one element). Now if $X$ is a right $G$-set with action $\rho: X\times G \to X$,  it is possible to define the  \emph{action groupoid $\gG$}, whose set of objects is $G_{\Sscript{0}}=X$ and whose set of arrows is $G_{\Sscript{1}}=X \times G$; the source and the target maps are, respectively, $\Sf{s}=\rho$ and $\Sf{t}=pr_{\Sscript{1}}$ and, lastly, the identity map sends $x$ to $(x,e)=\iota_{\Sscript{x}}$, where $e$ is the identity element of $G$. The multiplication is given by  $(x,g) (x',g')=(x,gg')$, whenever $xg=x'$, and the inverse is defined by $(x,g)^{-1}=(xg,g^{-1})$. Clearly the pair of maps $(pr_{\Sscript{2}}, *): \cG= (G_{\Sscript{1}}, G_{\Sscript{0}}) \to (G,\{*\})$  defines a morphism of groupoids. For a given $ x \in X$, the isotropy group $\cG^{\Sscript{x}}$ is clearly identified with  $\emph{Stab}_{\Sscript{G}}(x)=\{g \in G |\, gx=x\}$, the stabilizer subgroup of \(x\) in $G$.  
\end{example}

\begin{example}[Equivalence relation groupoid]\label{exam:X} 
Here is a standard class of examples of groupoids. Notice that in all these examples each of the isotropy groups is the trivial group. 
\begin{enumerate}[(1)] 
\item One can associated to  a given set $X$ the so called \emph{the groupoid  of pairs} (called \emph{fine groupoid} in \cite{Brown:1987} and \emph{simplicial groupoid} in \cite{Higgins:1971}): its set of arrows is defined by $G_{\Sscript{1}}=X \times X$ and the set of objects by $G_{\Sscript{0}}=X$. The source and the target are $\Sf{s}=pr_{\Sscript{2}}$ and $\Sf{t}=pr_{\Sscript{1}}$, the second and the first projections, and the  map of identity arrows $\iota$ is the diagonal map $x \mapsto (x,x)$. The multiplication and the inverse maps are given by 
$$
(x,x') \, (x',x'')\,=\, (x,x''),\quad \text{and} \quad (x,x')^{-1}\,=\, (x',x).
$$

\item Let $\nuup: X \to Y$ be a map.  Consider the fibre product $X\, \due \times {\Sscript{\nuup}} {\; \Sscript{\nuup}} \, X$  as a set of arrows of the groupoid $\xymatrix@C=35pt{ X\, \due \times {\Sscript{\nuup}} {\; \Sscript{\nuup}} \, X \ar@<0.8ex>@{->}|-{\scriptscriptstyle{pr_2}}[r] \ar@<-0.8ex>@{->}|-{\scriptscriptstyle{pr_1}}[r] & \ar@{->}|-{ \scriptscriptstyle{\iota}}[l] X, }$  where as before  $\Sf{s}=pr_{\Sscript{2}}$ and $\Sf{t}=pr_{\Sscript{1}}$,  and the  map of identity arrows $\iota$ is the diagonal map. The multiplication and the inverse are clear.  

\item Assume that $\cR \subseteq X \times X$ is an equivalence relation on the set $X$.  One can construct a groupoid $\xymatrix@C=35pt{\cR \ar@<0.8ex>@{->}|-{\scriptscriptstyle{pr_2}}[r] \ar@<-0.8ex>@{->}|-{\scriptscriptstyle{pr_1}}[r] & \ar@{->}|-{ \scriptscriptstyle{\iota}}[l] X, }$  with structure maps as before. This is  an important class of groupoids  known as \emph{the groupoid of equivalence relation} (or \emph{equivalence relation groupoid}). Obviously $(\cR, X) \hookrightarrow  (X\times X, X)$ is a morphism of groupoid, see for instance \cite[Exemple 1.4, page 301]{DemGab:GATIGAGGC}.
\end{enumerate}
\end{example}

\begin{example}[Induced groupoid]\label{exam:induced}
Let $\cG=(\cG_{\Sscript{1}},\cG_{\Sscript{0}})$ be a groupoid and $\varsigma:X \to \cG_{\Sscript{0}}$ a map. Consider the following  pair of sets:
$$
\cG^{\Sscript{\varsigma}}{}_{\Sscript{1}}:= X \,  \due \times {\Sscript{\varsigma}} {\, \Sscript{\Sf{t}}} \, \cG_{\Sscript{1}} \;  \due \times {\Sscript{\Sf{s}}} {\, \Sscript{\varsigma}}  \, X= \Big\{ (x,g,x') \in X\times \cG_{\Sscript{1}}\times X| \;\; \varsigma(x)=\Sf{t}(g), \varsigma(x')=\Sf{s}(g)  \Big\}, \quad \cG^{\Sscript{\varsigma}}{}_{\Sscript{0}}:=X.
$$
Then $\cG^{\Sscript{\varsigma}}{}=(\cG^{\Sscript{\varsigma}}{}_{\Sscript{1}}, \cG^{\Sscript{\varsigma}}{}_{\Sscript{0}})$ is a groupoid, with structure maps: $\Sf{s}= pr_{\Sscript{3}}$, $\Sf{t}= pr_{\Sscript{1}}$, $\iota_{\Sscript{x}}=(\varsigma(x), \iota_{\Sscript{\varsigma(x)}}, \varsigma(x))$, $x \in X$. The multiplication is defined by $(x,g,y) (x',g',y')= ( x,gg',y')$, whenever $y=x'$, and the inverse is given by $(x,g,y)^{-1}=(y,g^{-1},x)$. 
The groupoid $\cG^{\Sscript{\varsigma}}$ is known as \emph{the induced groupoid of $\cG$ by the map $\varsigma$}, (or \emph{ the pull-back groupoid of $\cG$ along $\varsigma$}, see   \cite{Higgins:1971} for dual notion).  Clearly, there  is a canonical morphism $\phi^{\Sscript{\varsigma}}:=(pr_{\Sscript{2}}, \varsigma): \cG^{\Sscript{\varsigma}} \to \cG$ of groupoids. 
\end{example}

\begin{remark}\label{pIsomTrGrpdGS}\label{pCharacTrGrpd}
A particular instance of an induced groupoid is the one when $\cG=G$ is a groupoid with one object. Thus for any group $G$ one can consider the Cartesian product $X \times G \times X$ as a groupoid with set of objects $X$. This groupoid, denoted by $\trg{G}{X}$,  is clearly  transitive with $G$ as a type of isotropy groups.  It is noteworthy to mention that the class of groupoids given in Example  \ref{exam:induced} characterizes, in fact, transitive groupoids. More precisely, any transitive groupoid is isomorphic, not in a canonical way, however, to some groupoid of the form $\trg{G}{X}$ with admissible choices \(X=\Cat{G}_{ \scriptscriptstyle{0} }\) and \(G=\Cat{G}^{ \scriptscriptstyle{x} }\) for any \(x \in \Cat{G}_{ \scriptscriptstyle{0} }\). 

Furthermore, given groups \(G\) and \(H\) and sets \(S\) and \(T\), it is easily shown that the  following statements are equivalent:
\begin{enumerate}
\item The groupoids \(\trg{G}{S}\) and \(\trg{H}{T}\) are isomorphic.
\item There is a bijection \(S \simeq T\) and an isomorphism of groups \(G \cong H\).
\end{enumerate}

\end{remark}

\subsection{Groupoid actions and equivariant maps}\label{ssec:Grpd1} The following crucial definition that we reproduce here from \cite{Kaoutit/Kowalzig:14, ElKaoutit:2017, Kaoutit/Spinosa:2018},  is a natural generalization to the context of groupoids, of the usual notion of group-set, see for instance \cite{Bouc:2010}. As was mentioned in \emph{loc.~cit} it is an abstract formulation of that given in \cite[Definition 1.6.1]{Mackenzie:2005} for Lie groupoids, and essentially the same definition based on the Sets-bundles notion given in  \cite[Definition 1.11]{Renault:1980}.
\begin{definition}\label{def:Gset}
Let  $\mathcal{G}$ be a groupoid and  $\varsigma:X \to \Go$ a map. We say that  $(X,\varsigma)$ is a \emph{right} $\cG$-\emph{set} (with a \emph{structure map} $\varsigma$), if there is a map (\emph{the action}) $\rho: X\, \due \times {\Sscript{\varsigma}} {\, \Sscript{{\Sf{t}}}} \, \Ga \to X$ sending $(x,g) \mapsto xg$ and  satisfying the following conditions:
\begin{enumerate}
\item $\Sf{s}(g)=\varsigma(xg)$, for any $x \in X$ and $g \in \Ga$ with $\varsigma(x)=\Sf{t}(g)$.
\item $x \, \iota_{\varsigma(x)}= x$, for every $x \in X$.
\item $ (xg)h= x(gh)$, for every $x \in X$, $g,h \in \Ga$ with $\varsigma(x)=\Sf{t}(g)$ and $\Sf{t}(h)=\Sf{s}(g)$.
\end{enumerate}
\end{definition}
In order to simplify the notation, the action map of a given right $\cG$-set $(X,\varsigma)$ will be omitted  and, by abuse of notation, we will simply refer to a right $\cG$-set $X$ without even mentioning the structure map. 
 A \emph{left action} is analogously defined by interchanging the source with the target and similar notations will be employed. 
 In general a set $X$ with a (right or left) groupoid action is just called \emph{a groupoid-set} but we will also employ the terminology: \emph{a set $X$ with a left (or right) $\cG$-action}. 
 
Obviously, any groupoid  $\cG$ acts  over itself on both sides by using the regular action, that is, the multiplication $\Ga \, \due \times {\Sscript{{\Sf{s}}}} {\, \Sscript{\Sf{t}}} \, \Ga \to \Ga$. This means that  $(\Ga, {s})$ is a right $\cG$-set and $(\Ga, {t})$ is a left $\cG$-set with this action. It is also clear that $(\Go, id_{\Sscript{\Go}})$ is a right $\cG$-set wit action given by
\begin{equation}\label{Eq:Go}
\Go \, {{}_{ \scriptscriptstyle{id} } {\times}_{ \scriptscriptstyle{\Sf{t}} }\,} \Ga \longrightarrow \Go, \quad \Big(  (a,g) \longmapsto ag=\Sf{s}(g) \Big).
\end{equation}

A \emph{morphism of  right $\cG$-sets} (or \emph{$\cG$-equivariant map})  $F: (X,\varsigma) \to (X',\varsigma')$ is a map $F:X \to X'$ such that the diagrams 

\begin{equation}\label{Eq:Gequi}
\begin{gathered}
\xymatrix@R=12pt{ & X \ar@{->}_-{\Sscript{\varsigma}}[ld]  \ar@{->}^-{F}[dd] & \\ \mathcal{G}_{\Sscript{0}}& & \\ & X' \ar@{->}^-{\Sscript{\varsigma'}}[lu]  & } \qquad  \qquad \xymatrix@R=12pt{X\, \due \times {\Sscript{\varsigma}} {\, \Sscript{\Sf{t}}} \,  \Ga \ar@{->}^-{}[rr]  \ar@{->}_-{\Sscript{F\, \times \, id}}[dd] & & X  \ar@{->}^-{\Sscript{F}}[dd] \\  & & \\ X'\, \due \times {\Sscript{\varsigma'}} {\, \Sscript{\Sf{t}}} \,  \Ga  \ar@{->}^-{}[rr] & & X'  } 
\end{gathered}
\end{equation}
commute.  
We denote by $\rGsets$ the category of right $\cG$-sets and by $\hom{\rGsets}{X}{X'}$ the set of all $\cG$-equivariant maps from $(X,\varsigma)$ to $(X',\varsigma')$. The category of left $\cG$-sets is analogously defined and it is isomorphic to the category of right $\cG$-sets, using the inverse map by switching the source with the target. It is noteworthy to mention that the definition of the category of groupoid-sets, as it has been recalled in Definition \ref{def:Gset}, can be rephrased using the core of the category of sets. To our purposes, it is advantageous to work with Definition \ref{def:Gset}, rather than this formal definition (see  \cite[Remark 2.6]{Kaoutit/Spinosa:2018} for more explanations).

\begin{example}\label{exam: HG}
Let $\phi: \cH \to \cG$ be a morphism of groupoids. Consider the triple $(\Ho\, \due \times {\Sscript{\phi_0}} {\, \Sscript{\Sf{t}}} \,  \Ga, pr_{\Sscript{1}}, \varsigma)$, where $\varsigma: \Ho\, \due \times {\Sscript{\phi_0}} {\, \Sscript{\Sf{t}}} \,  \Ga \to \Go$ sends $(u,g) \mapsto s(g)$, and $pr_{\Sscript{1}}$ is the first projection. Then the following maps 
\begin{equation}\label{Eq:HG}
\begin{gathered}
\xymatrix@R=0pt@C=10pt{ \Big(\Ho\, \due \times {\Sscript{\phi_0}} {\, \Sscript{\Sf{t}}} \,  \Ga\Big) \, \due \times {\Sscript{\varsigma}} {\, \Sscript{\Sf{t}}} \,  \Ga \ar@{->}^-{}[r] &  \Ho\, \due \times {\Sscript{\phi_0}} {\, \Sscript{\Sf{t}}} \,  \Ga,  \\ \Big((u,g'),g \Big)\ar@{|->}^-{}[r]  &  (u,g')  \rgaction g:=(u, g'g)}  \quad \xymatrix@R=0pt@C=10pt{ \Ha \, \due \times {\Sscript{\Sf{s}}} {\, \Sscript{pr_1}} \,  \Big(\Ho\, \due \times {\Sscript{\phi_0}} {\, \Sscript{\Sf{t}}} \,  \Ga\Big) \ar@{->}^-{}[r] &  \Ho\, \due \times {\Sscript{\phi_0}} {\, \Sscript{\Sf{t}}} \,  \Ga \\ \Big(h, (u,g)\Big) \ar@{|->}^-{}[r] &  h \lhaction (u,g):=(t(h), \phi(h)g) }
\end{gathered}
\end{equation}
define, respectively, a  structure of  right $\cG$-sets and  that of  left $\cH$-set.  Analogously,  the maps 
\begin{equation}\label{Eq:GH}
\begin{gathered}
\xymatrix@R=0pt@C=10pt{ \Big(\Ga\, \due \times {\Sscript{\Sf{s}}} {\, \Sscript{\phi_0}} \,  \Ho\Big) \, \due \times {\Sscript{pr_2}} {\, \Sscript{\Sf{t}}} \,  \Ha \ar@{->}^-{}[r] & \Ga\, \due \times {\Sscript{\Sf{s}}} {\, \Sscript{\phi_0}} \,  \Ho   \\ \Big((g,u),h \Big)\ar@{|->}^-{}[r]  &   (g,u) \rhaction h:= (g\phi(h), s(h))}  \;\;  \xymatrix@R=0pt@C=10pt{ \Ga \, \due \times {\Sscript{\Sf{s}}} {\, \Sscript{\vartheta}} \,  \Big(\Ga\, \due \times {\Sscript{\Sf{s}}} {\, \Sscript{\phi_0}} \,  \Ho\Big) \ar@{->}^-{}[r] &  \Ga\, \due \times {\Sscript{\Sf{s}}} {\, \Sscript{\phi_0}} \,  \Ho  \\ \Big(g, (g',u)\Big) \ar@{|->}^-{}[r] & g \lgaction (g',u):= (gg',u)}
\end{gathered}
\end{equation}
where $\vartheta: \Ga\, \due \times {\Sscript{\Sf{s}}} {\, \Sscript{\phi_0}} \,  \Ho \to \Go$ sends $(g,u) \mapsto t(g)$, define a left $\cH$-set and right $\cG$-set structures on $\Ga\, \due \times {\Sscript{\Sf{s}}} {\, \Sscript{\phi_0}} \,  \Ho$, respectively. 
This in particular can be applied to any morphism of groupoids of the form $(X,X)\to (Y \times Y, Y)$, $(x,x') \mapsto \big((f(x),f(x)), \, f(x')\big)$, where $f: X\to Y$ is any map. On the other hand, if $f$ is a $G$-equivariant map, for  a group $G$ acting on both $X$ and $Y$, then the above construction applies, as well, to the morphism of action groupoids $(G\times X, X) \to (G\times Y, Y)$ sending $\big(  (g,x), x' \big) \mapsto \big( (g,f(x)) , f(x')\big)$. 
\end{example}

The notion of groupoids-bisets is intuitively introduced. These are left and right groupoid-sets (over different groupoids) with a certain compatibility condition. We refer to \cite[Section 3.1]{Kaoutit/Spinosa:2018} or \cite[Definition 2.7]{ElKaoutit:2017} for  further details and we limit ourselves to give some examples. 

\begin{example}\label{exam:bisets}
Let $\phiup: \cH \to \cG$ be a morphism of groupoids. Consider, as in Example \ref{exam: HG},  the associated triples $(\Ho\, \due \times {\Sscript{\phiup_0}} {\, \Sscript{\Sf{t}}} \,  \Ga, \varsigma, pr_{\Sscript{1}})$ and  
$(\Ga\, \due \times {\Sscript{\Sf{s}}} {\, \Sscript{\phiup_0}} \,  \Ho, pr_{\Sscript{2}}, \vartheta)$ with actions defined as in eqautions \eqref{Eq:HG} and \eqref{Eq:GH}.  Then these triples are an $(\cH, \cG)$-biset and a $(\cG,\cH)$-biset, respectively. 
\end{example}

\subsection{Orbit sets and stabilizers}\label{ssec:Orbits}

Next we recall the notion (see, for instance, \cite[page~11]{Jelenc:2013} and \cite{Kaoutit/Spinosa:2018}) of the orbit set attached to a right groupoid-set.
This notion is a generalization of the orbit set in the context of group-sets.  Here we use the (right) translation groupoid to introduce this set.

Given  a right $\cG$-set  $(X,\varsigma)$, the \emph{orbit set}  $X/\cG$ of $(X,\varsigma)$ is the orbit set of the (right) translation groupoid  $X \rJoin \cG$, that is, $X/\cG=\pi_{\Sscript{0}}(X \rJoin \cG)$, the set of all connected  component.  For an element \(x \in X\),  the \emph{equivalence class}  of $x$, called the \emph{the orbit of $x$}, is denoted by 
\[ \Orb_{ \scriptscriptstyle{X \,  \rtimes \, \mathcal{G}} } (x)=\Set{ y \in X | \begin{aligned}
& \exists \, (x,g) \in \left(X \rtimes \mathcal{G}\right)_{ \scriptscriptstyle{1} } \, \text{ such that } \\
& x = \mathsf{t}^{\rtimes }(x,g) \, \text{ and }\, 
 y = \mathsf{s}^{\rtimes }(x,g)=xg
\end{aligned} } = \LR{ xg \in X |\; \mathsf{t}(g)=\varsigma(x) }
:=[x] \,  \mathcal{G}.
\]

Let us denote by $\rep_{\Sscript{\cG}}(X)$ a \emph{representative set} of the orbit set $X/\cG$. For instance, if $\cG=(X\times G, X)$ is an  action groupoid as in Example \ref{exam:action}, then obviously the orbit set of this groupoid coincides with the classical set of orbits  $X/G$. Of course, the orbit set of an equivalence relation groupoid $(\cR, X)$ (see Example \ref{exam:X}) is precisely the quotient set $X/\cR$.  
The left orbits sets for  left groupoids sets are analogously  defined by using the left translation groupoids. We will use the following notations for left orbits sets: given a left $\cG$-set $(Z,\vartheta)$, its orbit set will be denoted by $\cG \backslash Z$ and the orbit of an element $z \in Z$ by $\Glcoset{z}$.

A right $\cG$-sets is said to be \emph{transitive} if it has a single orbit, that is, if $X/\cG$ is a singleton, or equivalently its associated right translation groupoid $X \rJoin \cG$ is transitive.

Let \((X, \varsigma)\) be a right \(\mathcal{G}\)-set with action \(\rho \colon X  {{}_{ \scriptscriptstyle{\varsigma} } {\times }}_{ \scriptscriptstyle{\mathsf{t}} }\, \mathcal{G}_{ \scriptscriptstyle{1} } \longrightarrow X\). 
The \emph{stabilizer} \(\Stabi_{ \scriptscriptstyle{\mathcal{G}}}\left({x}\right) \) of \(x\) in \(\mathcal{G}\) is the groupoid with arrows 
\[ 
\left( \Stabi_{ \scriptscriptstyle{\mathcal{G}}}\left({x}\right)  \right)_{ \scriptscriptstyle{1} }
= \Set{ g \in \mathcal{G}_{ \scriptscriptstyle{1} }  |  \varsigma\left({x}\right)= \mathsf{t}\left({g}\right)  \quad \text{and} \quad x g = x }
\]
and objects 
\[ \left( \Stabi_{ \scriptscriptstyle{\mathcal{G}}}\left({x}\right)  \right)_{ \scriptscriptstyle{0} } 
= \Set{ u \in \mathcal{G}_{ \scriptscriptstyle{0} }  |  \exists g \in \mathcal{G}_{ \scriptscriptstyle{1} }(\varsigma\left({x}\right), u) : xg=x   } \subseteq \mathscr{O}_{ \scriptscriptstyle{\varsigma\left({x}\right)} } .
\]

Therefore, we have that
\[ \left( \Stabi_{ \scriptscriptstyle{\mathcal{G}}}\left({x}\right)  \right)_{ \scriptscriptstyle{0} } 
= \Set{\varsigma\left({x}\right)}
\qquad and \qquad
\Stabi_{\Sscript{\cG}}(x)^{\Sscript{\varsigma(x))}}=(\Stab{x}{G})_{\Sscript{1}}
 \le \mathcal{G}^{ \scriptscriptstyle{\varsigma\left({x}\right)} }.
\]
Thus $\Stab{x}{G}$ is a subgroupoid with one object, namely, $\varsigma(x)$.  Furthermore, as a subgroup of loops, $\Stab{x}{G}$ is identified with the isotropy group  $(X\rtimes \cG)^{\Sscript{x}}$ of the rigth translation groupoid $X \rJoin \cG$ (see \cite[Lemma 2.10]{Kaoutit/Spinosa:2018}).

\section{Monoidal equivalences between groupoid-sets}\label{sec:II}
This section contains the material and machinery that we are going to employ in performing Burnside theory for groupoids. This mainly consists in deciphering the monoidal structures of the category of groupoid-sets, in understanding the fixed point subsets functors and in characterising the conjugacy relation between subgroupoid with only one object. Notions like Laplaza categories, Laplaza functors and so on are exposed, in greater generality,  in the Appendix \ref{ssec:Laplaza}.

\subsection{The monoidal structures of the category of (right) $\cG$-sets and the induction functors}\label{ssec:SMC}
Let $\cG$ be a groupoid. Recall  that the category of right $\cG$-sets is a symmetric monoidal category with respect to the disjoint union $\Uplus$. This structure is given as follows: given two right $\cG$-set $(X, \varsigma)$ and $(Y,\vartheta)$, we set  \((X,\varsigma) \Uplus (Y, \vartheta) = (Z, \mu)\) where \(Z = X \Uplus Y\) and  the map $\muup: Z \to \Go$ is defined by the conditions $\muup_{|\, X}= \varsigma$ and $\muup_{|\, Y}= \vartheta$. The action is defined by:
$$
Z \, {{}_{ \scriptscriptstyle{\muup} } {\times}_{ \scriptscriptstyle{\Sf{t}} }\,} \Ga \longrightarrow Z, \quad \Big(  (z,g) \longmapsto zg \Big),
$$
where $zg$ stands for $xg$ if $z=x \in X$ or $yg$ if $z=y \in Y$. The identity object of this monoidal structure is the right $\cG$-set with an empty underlying set whose action is, by convention, the empty one.

On the other hand, the fibered product $-\underset{\Sscript{\Go}}{\times}-$ induces another symmetric monoidal structure: see, for instance, \cite[\S 2]{Kaoutit/Kowalzig:14}. Explicitly, the tensor product of $(X, \varsigma)$ and $(Y,\vartheta)$  is defined as follows:
$$
(X, \varsigma)\, \underset{\Sscript{\Go}}{\times} \, (Y,\vartheta) \,=\,  \big( X \underset{\Sscript{\Go}}{\times} Y, \varsigma\vartheta\big),
$$
where $\varsigma\vartheta: X \underset{\Sscript{\Go}}{\times} Y \to \Go$ sends $(x,y) \mapsto \varsigma(x)=\vartheta(y)$. The action, when it is possible,  is given by $(x,y) g\,=\, (xg, yg)$. The identity object is the right $\cG$-set $(\Go, id_{\Sscript{\Go}})$ with action  given as in \eqref{Eq:Go}. Furthermore, up to isomorphisms, $(\Go, id_{\Sscript{\Go}})$ is the only dualizable object with respect to this monoidal category. 

The compatibility between the two monoidal structure is expressed by the subsequent. 
\begin{lemma}\label{lRightGSetCoprod}
Given a groupoid \(\Cat{G}\), let be \(\left(\left({X_i, \varsigma_i}\right) \right)_{i\,  \in\,  I}\) and \(\left(\left({Y_j, \vartheta_j}\right) \right)_{j \,  \in \,  J}\) two families of right \(\Cat{G}\)-sets.
Then there is an isomorphism of right \(\Cat{G}\)-sets
\[ \biguplus_{\begin{subarray}{c}
i\,  \in \,  I \\
j \,  \in \,  J \end{subarray} } \left({ \left({X_i, \varsigma_i}\right) \fpro{\Cat{G}_{ \scriptscriptstyle{0} }} \left({Y_j, \vartheta_j}\right) }\right) \,\, \cong\,\,  \left({\biguplus_{i \, \in \,  I} \left({X_i, \varsigma_i}\right) }\right) \, \fpro{\Cat{G}_{ \scriptscriptstyle{0} } }\,   \left({\biguplus_{j \, \in\,  J} \left({Y_j, \vartheta_j}\right) }\right) .
\]
\end{lemma}
\begin{proof}
It is omitted, since it is a direct verification. 
\end{proof}

Let \(\phi \colon \Cat{H} \longrightarrow \Cat{G}\) be a morphism of groupoid.
We define the induced functor, referred to as the \emph{induction functor}: 
\[ \phi^\ast  \colon \Rset{\cG} \longrightarrow \Rset{\cH} ,
\]
which sends the right \(\Cat{G}\)-set \(\left({X, \varsigma}\right)\) to the right \(\Cat{H}\)-set \(\left({X  \, {{}_{ \scriptscriptstyle{\varsigma} } {\times}_{ \scriptscriptstyle{\phi_0} }\,}  \Cat{H}_{ \scriptscriptstyle{0} } , \pr_{ \scriptscriptstyle{2} }}\right)\) with action
\[ 
\begin{aligned}
& \left({X  \, {{}_{ \scriptscriptstyle{\varsigma} } {\times}_{ \scriptscriptstyle{\phi_0} }\,}  \Cat{H}_{ \scriptscriptstyle{0} }}\right)  \, {{}_{ \scriptscriptstyle{\pr_{ \scriptscriptstyle{2} } } } {\times}_{ \scriptscriptstyle{\mathsf{t} } }\,} \Cat{H}_{ \scriptscriptstyle{1} }  \longrightarrow X  \, {{}_{ \scriptscriptstyle{\varsigma} } {\times}_{ \scriptscriptstyle{\phi_0} }\,}  \Cat{H}_{ \scriptscriptstyle{0} }, \qquad 
& \Big((x,a), h) \longrightarrow (x \phi_{\Sscript{1}}(h), \mathsf{s}(h))\Big) .
\end{aligned}
\]
Given a morphism of right \(\Cat{G}\)-set \(f \colon \left({X, \varsigma}\right) \longrightarrow \left({Y, \vartheta}\right)\), we define the morphism
\[ 
\phi^\ast(f) \colon \phi^\ast\left({X, \varsigma}\right) \longrightarrow \phi^\ast\left({Y, \vartheta}\right) 
\]
as the morphism
\[ \begin{aligned}{f \times \Cat{H}_{ \scriptscriptstyle{0} } }  \colon & {\left({X  \, {{}_{ \scriptscriptstyle{\varsigma} } {\times }_{ \scriptscriptstyle{\phi_{ \scriptscriptstyle{0} }  } }\,} \Cat{H}_{ \scriptscriptstyle{0} } , \pr_{ \scriptscriptstyle{2} } }\right) } \longrightarrow {\left({Y  \, {{}_{ \scriptscriptstyle{\theta} } {\times}_{ \scriptscriptstyle{\phi_{ \scriptscriptstyle{0} } } }\,}\Cat{H}_{ \scriptscriptstyle{0} } , \pr_{ \scriptscriptstyle{2} } }\right) }, \qquad & \Big({\left({x, a}\right) }  \longrightarrow {\left({f(x), a}\right) } \Big).\end{aligned} 
\]
For instance, we have that 
$\phi^{\ast}(\Go, id_{\Sscript{\Go}})=\left({\Go  \, {{}_{ \scriptscriptstyle{id} } {\times}_{ \scriptscriptstyle{\phi_0} }\,}  \Cat{H}_{ \scriptscriptstyle{0} } , \pr_{ \scriptscriptstyle{2} }}\right)$. The following is a well known property of the induction functor (see \cite{Kaoutit/Kowalzig:14}).
However, for the sake of completeness and for the convenience of the reader, we give here an elementary proof.

\begin{proposition}\label{pMorGrpdIndFuRSets}
The functor  \(\phi^\ast\) is monoidal with respect to both \(\Uplus\) and the fibered product $-\underset{\Go}{\times}-$.
\end{proposition}
\begin{proof}
The fact that \(\phi^\ast\) is well defined is routine computation and we leave it to the reader. 
Let us check  that \(\phi^\ast\) is monoidal with respect to \(\Uplus\).
Given right sets \(\left({X, \varsigma}\right)\) and \(\left({Y, \vartheta}\right)\) we have the natural isomorphisms
\begin{multline*}
 \quad \phi^\ast\left({ \left({X, \varsigma}\right) \Uplus \left({Y, \vartheta}\right) }\right) = \phi^\ast\left({X \Uplus Y, \varsigma \Uplus \vartheta}\right) 
= \left({ \left({X \Uplus Y}\right)  \, {{}_{ \scriptscriptstyle{\varsigma \uplus \vartheta} } {\times}_{ \scriptscriptstyle{\phi_{ \scriptscriptstyle{0} } } }\,} \Cat{H}_{ \scriptscriptstyle{0} } , \pr_{ \scriptscriptstyle{2} } }\right) \\ \cong \left({X  \, {{}_{ \scriptscriptstyle{\varsigma} } {\times}_{ \scriptscriptstyle{\phi_{ \scriptscriptstyle{0} } } }\,} \Cat{H}_{ \scriptscriptstyle{0} }  , \pr_{ \scriptscriptstyle{2} } }\right) \Uplus \left({Y  \, {{}_{ \scriptscriptstyle{\vartheta} } {\times}_{ \scriptscriptstyle{\phi_{ \scriptscriptstyle{0} } } }\,}  \Cat{H}_{ \scriptscriptstyle{0} } , \pr_{ \scriptscriptstyle{0} } }\right) 
= \phi^\ast\left({X, \varsigma}\right) \Uplus \phi^\ast\left({Y, \vartheta}\right) 
\end{multline*}
and, we also have that
\[ \phi^\ast\left({\emptyset, \emptyset}\right) = \left({ \emptyset  \, {{}_{ \scriptscriptstyle{\emptyset} } {\times}_{ \scriptscriptstyle{\phi_{ \scriptscriptstyle{0} }  } }\,} \Cat{H}_{ \scriptscriptstyle{0} }, \pr_{ \scriptscriptstyle{2} } }\right) 
= \left({\emptyset, \emptyset}\right) .
\]
Now we have to prove that \(\phi^\ast\) is monoidal with respect to the fibered product.
Given right \(\mathcal{G}\)-sets \(\left({X, \varsigma}\right)\) and \(\left({Y, \vartheta}\right)\) we have the natural isomorphisms
\begin{multline*}
 \quad \phi^\ast\left({ \left({X, \varsigma}\right) \fpro{\Cat{G}_{ \scriptscriptstyle{0} }} \left({Y, \vartheta}\right) }\right) = \phi^\ast\left({ X  \, {{}_{ \scriptscriptstyle{\varsigma} } {\times}_{ \scriptscriptstyle{\vartheta} }\,} Y, \varsigma \vartheta}\right)
= \left({\left({X  \, {{}_{ \scriptscriptstyle{\varsigma} } {\times}_{ \scriptscriptstyle{\vartheta} }\,} Y }\right)  \, {{}_{ \scriptscriptstyle{\varsigma \vartheta} } {\times}_{ \scriptscriptstyle{\phi_{ \scriptscriptstyle{0} } } }\,} \Cat{H}_{ \scriptscriptstyle{0} } , \pr_{ \scriptscriptstyle{2} } }\right)  \\
\cong \left({ \left({X  \, {{}_{ \scriptscriptstyle{\varsigma} } {\times}_{ \scriptscriptstyle{\phi_{ \scriptscriptstyle{0} } } }\,} \Cat{H}_{ \scriptscriptstyle{0} } }\right)  \, {{}_{ \scriptscriptstyle{\pr_{ \scriptscriptstyle{2} } } } {\times}_{ \scriptscriptstyle{\pr_{ \scriptscriptstyle{2 } } } }\,} \left({Y  \, {{}_{ \scriptscriptstyle{\vartheta} } {\times}_{ \scriptscriptstyle{\phi_{ \scriptscriptstyle{0} } } }\,}  \Cat{H}_{ \scriptscriptstyle{0} }  }\right), \left({\pr_{ \scriptscriptstyle{2} } }\right) \left({\pr_{ \scriptscriptstyle{2 } } }\right) }\right) 
 = \left({X  \, {{}_{ \scriptscriptstyle{ \varsigma} } {\times}_{ \scriptscriptstyle{\phi_{ \scriptscriptstyle{0} } } }\,} \Cat{H}_{ \scriptscriptstyle{0} }, \pr_{ \scriptscriptstyle{2} } }\right) \fpro{\Cat{H}_{ \scriptscriptstyle{0} } } \left({Y  \, {{}_{ \scriptscriptstyle{\vartheta} } {\times}_{ \scriptscriptstyle{\phi_{ \scriptscriptstyle{0} } } }\,}, \pr_{ \scriptscriptstyle{2} } }\right) 
= \phi^\ast\left({X, \varsigma}\right) \fpro{\Cat{H}_{ \scriptscriptstyle{0} }} \phi^\ast\left({Y, \vartheta}\right),
\end{multline*}
because an element of
\( \left({X  \, {{}_{ \scriptscriptstyle{\varsigma} } {\times}_{ \scriptscriptstyle{\phi_{ \scriptscriptstyle{0} } } }\,} \Cat{H}_{ \scriptscriptstyle{0} } }\right)  \, {{}_{ \scriptscriptstyle{\pr_{ \scriptscriptstyle{2} } } } {\times}_{ \scriptscriptstyle{\pr_{ \scriptscriptstyle{2 } } } }\,} \left({Y  \, {{}_{ \scriptscriptstyle{\vartheta} } {\times}_{ \scriptscriptstyle{\phi_{ \scriptscriptstyle{0} } } }\,}  \Cat{H}_{ \scriptscriptstyle{0} }  }\right)
\)
is of the kind \(\left({x,a, y, a}\right)\) with \(x \in X\), \(y \in Y\), \(a \in \Cat{H}_{ \scriptscriptstyle{0} }\). 
Therefore, we have a natural isomorphism $\phi^\ast\left({ \left({X, \varsigma}\right) \fpro{\Cat{G}_{ \scriptscriptstyle{0} }} \left({Y, \vartheta}\right) }\right) \cong  \phi^\ast\left({X, \varsigma}\right) \fpro{\Cat{H}_{ \scriptscriptstyle{0} }} \phi^\ast\left({Y, \vartheta}\right)$, for every pair of right $\cG$-sets $(X, \varsigma)$ and $(Y,\vartheta)$. On the other hand, we have that 
\( \phi^\ast\left({\Cat{G}_{ \scriptscriptstyle{0} }, \Id_{\Cat{G}_{ \scriptscriptstyle{0} } } }\right) = \left({\Cat{G}_{ \scriptscriptstyle{0} } \,  \, {{}_{ \scriptscriptstyle{\Id_{\Cat{G}_{ \scriptscriptstyle{0} } } } } {\times}_{ \scriptscriptstyle{\phi_{ \scriptscriptstyle{0} }} }\,} \Cat{H}_{ \scriptscriptstyle{0} } , \pr_{ \scriptscriptstyle{2} } }\right) 
\cong \left({\Cat{H}_{ \scriptscriptstyle{0} } , \Id_{\Cat{H}_{ \scriptscriptstyle{0} } } }\right),
\)
since the map $ pr_{\Sscript{2}}: \Cat{G}_{ \scriptscriptstyle{0} } \,  \, {{}_{ \scriptscriptstyle{Id } } {\times}_{ \scriptscriptstyle{\phi_{ \scriptscriptstyle{0} }} }\,} \Cat{H}_{ \scriptscriptstyle{0} }  \to \Ho$ establishes an isomorphism of right $\cH$-sets.
\end{proof}

In the terminology of the Appendix \ref{ssec:Laplaza}, the induced functor $\phi^*$ is then a Laplaza functor.

\begin{propositiondef}\label{dpDefIndTra}
Given groupoid \(\Cat{H}\) and \(\Cat{G}\), let  \(\phi, \psi \colon \Cat{H} \longrightarrow \Cat{G}\) be two morphisms of groupoids and consider a natural transformation \(\alpha \colon \phi \longrightarrow \psi\).
We define an \textit{induced natural transformation}
\[ 
\alpha^\ast \colon \phi^\ast \longrightarrow \psi^\ast
\]
between the induced functors
\[ \phi^\ast, \, \psi^\ast \colon \Rset{\cG} \longrightarrow \Rset{\cH} 
\]
as follows: for each right \(\Cat{G}\)-set \((X, \varsigma)\) and for each \((x,a) \in \phi^\ast (X, \varsigma)\) we set
\[  
\alpha^\ast_{(X, \, \varsigma)}: \phi^*(X,\varsigma) \longrightarrow  \psi^*(X,\varsigma),\quad  \Big((x,a) \longmapsto \alpha^\ast_{(X, \, \varsigma)} (x,a) = \left({ x \cdot (\alpha a)^{-1} , a}\right)\Big).
\]
Moreover, \(\alpha^\ast\) is a Laplaza transformation (see Appendix \ref{ssec:Laplaza} for the pertinent definition).
\end{propositiondef}
\begin{proof}
Given a right \(\Cat{G}\)-set \((X, \varsigma)\) and \((x,a) \in \phi^\ast(X, \varsigma)\), the situation is as follows:
\[ \xymatrix{ \varsigma(x) = \phi_{ \scriptscriptstyle{0} }(a)  \ar[rr]^{\alpha a} & & \psi_{ \scriptscriptstyle{0} }(a). }
\]
We have \(\varsigma(x) = \phi_{ \scriptscriptstyle{0} }(a) = \mathsf{s}(\alpha a) = \mathsf{t}\left({(\alpha a)^{-1}}\right)\) thus we can write \(x \cdot (\alpha a)^{-1}\) and we have
\[ \varsigma\left({ x \cdot (\alpha a)^{-1}}\right) = \mathsf{s}\left({(\alpha a)^{-1} }\right)
= \mathsf{t}(\alpha a) 
= \psi_{ \scriptscriptstyle{0} }(a).
\]
Therefore \(\alpha^\ast(X, \varsigma)\) is well defined.
We have to check that \(\phi^\ast(X, \varsigma)\) is a morphism of right \(\Cat{H}\)-sets.
The condition on the structure map is obviously satisfied.
Regarding the condition on the actions, let be \((x,a) \in \phi^\ast(X, \varsigma)\) and \(h \in \Cat{H}_{ \scriptscriptstyle{1} }\) such that \(a = \pr_{ \scriptscriptstyle{2} }(x,a) = \mathsf{t}(h)\): the arrow \(h \colon b \longrightarrow a\) is a morphism in \(\Cat{H}\) thus the following diagram is commutative:
\[ \xymatrix{
\phi_{ \scriptscriptstyle{0} }(b) \ar[rr]^{\alpha(b)} \ar[d]_{\phi_{ \scriptscriptstyle{1} }(h)} & & \psi_{ \scriptscriptstyle{0} }(b)  \ar[d]^{\psi_{ \scriptscriptstyle{1} }(h)}  \\
\phi_{ \scriptscriptstyle{0} }(a) \ar[rr]^{\alpha(a)} & & \psi_{ \scriptscriptstyle{0} }(a) .
}
\]
As a consequence we can compute
\begin{multline*}
 \quad \alpha^\ast_{(X,\,  \varsigma)}((x, a) h) = \alpha^\ast_{(X, \,\varsigma)} \left({x \phi_{ \scriptscriptstyle{1} }(h), b}\right) 
= \left({x \cdot \phi_{ \scriptscriptstyle{1} }(h) \cdot (\alpha b )^{-1}, b}\right) 
= \left({x \cdot (\alpha b)^{-1} \cdot \psi_{ \scriptscriptstyle{1} }(h), b }\right)  \\
= \left({x \cdot (\alpha b)^{-1}, a}\right)  \cdot h  
= \left({\alpha^\ast_{(X,\,  \varsigma)}(x,a)}\right) \cdot h,
\end{multline*}
which show that $\alpha^\ast_{(X,\, \varsigma)}$ is an $\cH$-equivariant map.

We have to check that \(\alpha^\ast\) is natural that is, given a morphism of right \(\Cat{G}\)-sets \(f \colon (X, \varsigma) \longrightarrow (Y,\vartheta)\), that the following diagram is commutative:
\[ \xymatrix{
\phi^\ast(X, \varsigma) = \left({X  \, {{}_{ \scriptscriptstyle{\varsigma} } {\times}_{ \scriptscriptstyle{\phi_{ \scriptscriptstyle{0} } } }\,} \Cat{H}_{ \scriptscriptstyle{0} }, \pr_{ \scriptscriptstyle{2} } }\right)  \ar[rrr]^{\alpha^\ast(X, \varsigma)} \ar[d]_{\phi^\ast(f) = f \times \Id_{\Cat{H}_{ \scriptscriptstyle{0} } } } &&& \psi^\ast(X, \varsigma) = \left({X  \, {{}_{ \scriptscriptstyle{\varsigma} } {\times}_{ \scriptscriptstyle{\psi_{ \scriptscriptstyle{0} } } }\,} \Cat{H}_{ \scriptscriptstyle{0} }, \pr_{ \scriptscriptstyle{2} } }\right)  \ar[d]^{\psi^\ast(g) = g \times \Id_{\Cat{H}_{ \scriptscriptstyle{0} } } }  \\
\phi^\ast(Y, \vartheta) = \left({Y  \, {{}_{ \scriptscriptstyle{\vartheta} } {\times}_{ \scriptscriptstyle{\phi_{ \scriptscriptstyle{0} } } }\,} \Cat{H}_{ \scriptscriptstyle{0} }, \pr_{ \scriptscriptstyle{2} } }\right) \ar[rrr]^{\alpha^\ast(Y, \vartheta)} &&& \psi^\ast(Y, \vartheta) = \left({Y  \, {{}_{ \scriptscriptstyle{\vartheta} } {\times}_{ \scriptscriptstyle{\psi_{ \scriptscriptstyle{0} } } }\,} \Cat{H}_{ \scriptscriptstyle{0} }, \pr_{ \scriptscriptstyle{2} } }\right).
}
\]
This follows from the following computation: Let be \((x, a) \in \phi^\ast(X, \varsigma)\), we have
\[ \begin{aligned}
\left( \psi^\ast f \right) \left({\alpha^\ast(X, \varsigma)}\right)(x, a) & = \left({\psi^\ast f}\right) \left({x \cdot (\alpha a)^{-1}, a}\right) 
= \left({f\left({x \cdot (\alpha a)^{-1}}\right), a}\right) 
= \left({f(x)(\alpha a)^{-1}, a}\right) \\
&= \alpha^\ast(Y, \vartheta)\left({f(x), a}\right) 
= \alpha^\ast(Y, \vartheta)\left({\phi^\ast f}\right)(x,a).
\end{aligned}
\]
The fact that \(\alpha^\ast\) is a Laplaza transformation, is directly proved and left to the reader. 
\end{proof}

\begin{proposition}\label{pIndFuIsomCompUn}
Given groupoids \(\Cat{K}\), \(\Cat{H}\) and \(\Cat{G}\), let's consider the following homomorphism of groupoids:
\[ 
\xymatrix{\Cat{K}  \ar[rr]^{\psi} && \Cat{H} \ar[rr]^{\phi} && \Cat{G} }.
\]
Then the following diagrams commute up to a natural isomorphism
\[ \xymatrix{
\Rset{G}  \ar[d]_{\phi^\ast} \ar@/^1.4pc/@{{}{ }{}}[dr]_{\cong}  \ar[rr]^{\left({\phi \psi}\right)^\ast} &  & \Rset{K}   \\
\Rset{H}   \ar[rru]_{\psi^\ast}  & & }
\qquad \text{and} \qquad
\xymatrix{\Rset{G} \ar@/_1.5pc/[rr]_{\Id_{\Rset{G} } } \ar@/^1.5pc/[rr]^{\left({\Id_{\Cat{G} } }\right)^\ast} & \cong  & \Rset{G} }
\]
That is, there are Laplaza natural isomorphisms
\[ \gamma \colon \left({\varphi \psi}\right)^\ast \longrightarrow \psi^\ast\phi^\ast
\qquad \text{and} \qquad
\beta \colon \left({\Id_{\Cat{G} } }\right)^\ast  \longrightarrow \Id_{\Rset{G} } .
\]
\end{proposition}
\begin{proof}
Given a  homomorphism \(f \colon \left({X, \varsigma}\right) \longrightarrow \left({Y, \theta}\right)\) in \(\Rset{G}\), we have
\[  \begin{gathered}
\xymatrix{ \left({X  \, {{}_{ \scriptscriptstyle{\varsigma} } {\times}_{ \scriptscriptstyle{\varphi \psi } }\,} \Cat{K}_{ \scriptscriptstyle{0} }, \pr_{ \scriptscriptstyle{2} }  }\right) =   \left({\varphi \psi}\right)^\ast \left({X, \varsigma}\right)  \ar[rr]^{\left({\varphi \psi}\right)^\ast(f) } & & \left({\varphi \psi}\right)^\ast \left({Y, \theta}\right) = \left({Y  \, {{}_{ \scriptscriptstyle{\theta} } {\times}_{ \scriptscriptstyle{\varphi \psi} }\,} \Cat{K}_{ \scriptscriptstyle{0} } , \pr_{ \scriptscriptstyle{2} } }\right)   }  \\
\xymatrix{ (x,a)  \ar[rr]^{ } & & \left({f(x),a}\right)  } 
\end{gathered}
\]
and
\[ \begin{gathered}
\xymatrix@R=1pc{
\psi^\ast\left({X  \, {{}_{ \scriptscriptstyle{\varsigma} } {\times}_{ \scriptscriptstyle{\varphi_{ \scriptscriptstyle{0} } } }\,} \Cat{H}_{ \scriptscriptstyle{0} } , \pr_{ \scriptscriptstyle{2} } }\right)  = \psi^\ast\phi^\ast\left({X, \varsigma}\right)  \ar[rrr]^{\psi^\ast\phi^\ast(f) } \ar@{=}[d]_{} &&& \psi^\ast\varphi^\ast\left({Y, \theta}\right) = \psi^\ast\left({Y  \, {{}_{ \scriptscriptstyle{\theta} } {\times}_{ \scriptscriptstyle{\varphi_{ \scriptscriptstyle{0} } } }\,} \Cat{H}_{ \scriptscriptstyle{0} }, \pr_{ \scriptscriptstyle{2} } }\right) \ar@{=}[d]^{}  \\
\left({\left({X  \, {{}_{ \scriptscriptstyle{\varsigma} } {\times}_{ \scriptscriptstyle{\varphi_{ \scriptscriptstyle{0} } } }\,}  \Cat{H}_{ \scriptscriptstyle{0} } }\right)  \, {{}_{ \scriptscriptstyle{\pr_{ \scriptscriptstyle{2} } } } {\times}_{ \scriptscriptstyle{\psi_{ \scriptscriptstyle{0} } } }\,} \Cat{K}_{ \scriptscriptstyle{0} } , \pr_{ \scriptscriptstyle{3} } }\right)   &&&  \left({\left({Y  \, {{}_{ \scriptscriptstyle{\theta} } {\times}_{ \scriptscriptstyle{\varphi_{ \scriptscriptstyle{0} } } }\,}  \Cat{H}_{ \scriptscriptstyle{0} } }\right)  \, {{}_{ \scriptscriptstyle{\pr_{ \scriptscriptstyle{2} } } } {\times}_{ \scriptscriptstyle{\psi_{ \scriptscriptstyle{0} } } }\,} \Cat{K}_{ \scriptscriptstyle{0} } , \pr_{ \scriptscriptstyle{3} } }\right)  
} \\
\xymatrix{ (x,a,b)  \ar[rrr]^{ } &&& \left({\varphi^\ast(x,a), b }\right) = \left({f(x), a, b}\right) .
}
\end{gathered}
\]
This shows the first claim. As for the second one, for any  right \(\Cat{G}\)-set \(\left({X, \varsigma}\right)\),  we consider the following $\cG$-equivariant maps:
\[ 
\begin{aligned}{\gamma_{\Sscript{\left({X, \varsigma}\right)}} }  \colon & {\psi^\ast\phi^\ast_{\Sscript{\left({X, \varsigma}\right)}} } \longrightarrow {\left({\varphi\psi}\right)^\ast_{\Sscript{\left({X, \varsigma}\right)}} } \\ & {(x,a,b)}  \longrightarrow {(x,b)} 
\end{aligned}  \qquad 
\begin{aligned}{\beta_{\Sscript{\left({X, \varsigma}\right)}} }  \colon & {\left({\Id_{\Cat{G} } }\right)^\ast_{\Sscript{\left({X, \varsigma}\right)}} } \longrightarrow { \left({X, \varsigma}\right) } \\ & {(x,a)}  \longrightarrow {x}\end{aligned}
\]
which gives us  the desired natural transformations.   
The proof of the fact that  \(\gamma\) and \(\beta\) are Laplaza isomorphism is immediate.
\end{proof}

\begin{proposition}
\label{pPropIndFun}
Given groupoids \(\Cat{H}\) and \(\Cat{G}\) and morphism of groupoids \(\phi, \psi, \mu \colon \Cat{H} \longrightarrow \Cat{G}\), let's consider natural transformations \(\alpha \colon\phi \longrightarrow \psi\) and \(\beta \colon \psi \longrightarrow \mu\).
Then the following diagrams are commutative:
\[  \begin{minipage}[c]{0.3\linewidth}
\centering
\xymatrix{
\phi^\ast \ar[d]_{\alpha^\ast}  \ar[rr]^{\left({\beta \alpha}\right)^\ast} &  & \mu^\ast  \\
\psi^\ast  \ar[rru]_{\beta^\ast}    }
\end{minipage}\hfill
\begin{minipage}[c]{0.2\linewidth}
\centering
\text{and} 
\end{minipage}\hfill
\begin{minipage}[c]{0.3\linewidth}
\centering
\xymatrix{\phi^\ast  \ar@/^1pc/[rr]^{\left({\Id_{\varphi} }\right)^\ast} \ar@/_1pc/[rr]_{ \Id_{\varphi^\ast}  }  & &\phi^\ast } 
\end{minipage}
\]
Moreover, if \(\alpha\) is a natural isomorphism, then we have \(\left({\alpha^{-1}}\right)^\ast = \left({\alpha^\ast}\right)^{-1}\).
\end{proposition}
\begin{proof}
Straightforward. 
\end{proof}

We finish this subsection with the followings useful results. 

\begin{proposition}\label{pRestrGrpdRSet}
Given a groupoid \(\Cat{G}\), let \(\Cat{A}\) be a subgroupoid of \(\Cat{G}\).
Then the functor
\[ \begin{aligned}{F}  \colon & {\Rset{\cG} } \longrightarrow {\Rset{\cA} } \\ & {\left({X, \varsigma}\right) }  \longrightarrow {\left({Y = \varsigma^{-1}\left({\mathcal{A}_{ \scriptscriptstyle{0} } }\right) , \vartheta = \left.{\varsigma}\right|_{\varsigma^{-1}\left({\mathcal{A}_{ \scriptscriptstyle{0} } }\right)} }\right), }\end{aligned} 
\]
opportunely defined on morphisms, leads to a  Laplaza functor.
\end{proposition}
\begin{proof}
The fact that \(F\) is a well defined  functor is obvious. 
We have to prove that \(F\) is monoidal with respect to the disjoint union.
So, let be \(\left({X_1, \varsigma_1}\right), \left({X_2, \varsigma_2}\right) \in \Rset{G}\): we have
\[ \begin{aligned}
 \quad F\left({X_1, \varsigma}\right) \Uplus F\left({X_2, \varsigma_2}\right) 
& = \left({\varsigma_1^{-1}\left({\mathcal{A}_{ \scriptscriptstyle{0} }}\right), \left.{\varsigma_1}\right|_{\varsigma_1^{-1}\left({\mathcal{A}_{ \scriptscriptstyle{0} }}\right)} }\right) \Uplus \left({\varsigma_2^{-1}\left({\mathcal{A}_{ \scriptscriptstyle{0} }}\right), \left.{\varsigma_2}\right|_{\varsigma_2^{-1}\left({\mathcal{A}_{ \scriptscriptstyle{0} }}\right)} }\right)  \\
& = \left({\varsigma_1^{-1}\left({\mathcal{A}_{ \scriptscriptstyle{0} }}\right) \Uplus \varsigma_2^{-1}\left({\mathcal{A}_{ \scriptscriptstyle{0} }}\right), \left.{\varsigma_1}\right|_{\varsigma_1^{-1}\left({\mathcal{A}_{ \scriptscriptstyle{0} }}\right)} \Uplus \left.{\varsigma_2}\right|_{\varsigma_2^{-1}\left({\mathcal{A}_{ \scriptscriptstyle{0} }}\right)}  }\right)  \\
&= \left({\left({\varsigma_1 \Uplus \varsigma_2}\right)^{-1}\left({\mathcal{A}_{ \scriptscriptstyle{0} }}\right), \left.{\left( \varsigma_1 \Uplus \varsigma_2 \right)}\right|_{\left({\varsigma_1 \Uplus \varsigma_2}\right)^{-1}\left({\mathcal{A}_{ \scriptscriptstyle{0} }}\right)} }\right) \\
&= F\left({X_1 \Uplus X_2, \varsigma_1 \Uplus \varsigma_2}\right) 
= F\left({\left({X_1, \varsigma_1}\right) \Uplus \left({X_2, \varsigma_2}\right) }\right) 
\end{aligned}
\]
and evidently, we have $F\left({\emptyset, \emptyset}\right)=  \left({\emptyset, \emptyset}\right)$.
Therefore, since the coherency conditions are immediate to verify, \(F\) is monoidal strict with respect to \(\Uplus\).
It remains to check  that \(F\) is monoidal with respect to the fiber product.
For each \(\left({X_1, \varsigma_1}\right), \left({X_2, \varsigma_2}\right) \in \Rset{G}\),  we have, using the notations \(\vartheta_i=\left.{\varsigma_i}\right|_{\varsigma_i^{-1}\left({\mathcal{A}_{ \scriptscriptstyle{0} } }\right)}\) for \({i}\in \Set{{1},{2}}\),
\[ \begin{aligned}
\quad  F\left({X_1, \varsigma_1}\right) \fpro{\Cat{A}_{ \scriptscriptstyle{0} }} F\left({X_2, \varsigma_2}\right) 
& = \left({\varsigma_1^{-1}\left({\mathcal{A}_{ \scriptscriptstyle{0} } }\right), \vartheta_1}\right) \fpro{\Cat{A}_{ \scriptscriptstyle{0} }} \left({\varsigma_2^{-1}\left({\mathcal{A}_{ \scriptscriptstyle{0} } }\right), \vartheta_2}\right)  \\
&= \left({\varsigma_1^{-1}\left({\mathcal{A}_{ \scriptscriptstyle{0} } }\right) \, {  \, {{}_{ \scriptscriptstyle{\vartheta_1} } {\times}_{ \scriptscriptstyle{\vartheta_2} }\,} } \varsigma_2^{-1}\left({\mathcal{A}_{ \scriptscriptstyle{0} } }\right), \vartheta_1 \vartheta_2}\right)  
\\ &=\left( \left({\varsigma_1 \varsigma_2}\right)^{-1}\left({\mathcal{A}_{ \scriptscriptstyle{0} } }\right), \left.{\varsigma_1 \varsigma_2}\right|_{\left({\varsigma_1 \varsigma_2}\right)^{-1}\left({\mathcal{A}_{ \scriptscriptstyle{0} } }\right)} \right) \\
&=F\left({X_1  \, {{}_{ \scriptscriptstyle{\varsigma_1} } {\times}_{ \scriptscriptstyle{\varsigma_2} }\,} X_2, \varsigma_1 \varsigma_2}\right) 
= F\left({\left({X_1, \varsigma_1}\right) \fpro{\Cat{G}_{ \scriptscriptstyle{0} }}, \left({X_2, \varsigma_2}\right) }\right) 
\end{aligned}
\]
and, obviously, we have that $ F\left({\mathcal{G}_{ \scriptscriptstyle{0} }, \Id_{\mathcal{G}_{ \scriptscriptstyle{0} } } }\right) = \left({\mathcal{A}_{ \scriptscriptstyle{0} }, \Id_{\mathcal{A}_{ \scriptscriptstyle{0} }} }\right) $.
Therefore, since the coherency conditions are immediate to verify, \(F\) is monoidal strict with respect to the fiber product and the proof is concluded.
\end{proof}

\subsection{Monoidal equivalences and category decompositions}\label{ssec:product}
We announce several monoidal decompositions (up to equivalence of categories)  of certain categories of groupoids-sets into a product of categories that will be used in the forthecoming sections. 

\begin{proposition}\label{pSetsCoprGrpd}
Let $\{\cG_i\}_{i \, \in\, I}$ be a family of groupoids and $\{ i_j : \cG_j \to \cG\}_{j\, \in \, I}$  be their coproduct in \(\Grou\).
Then we have a Laplaza isomorphism of categories:
\[ 
\Rset{\Big(\coprod_{j \,  \in \,  I} \Cat{G}_j\Big)} \cong \prod_{j \, \in \,  I} \Rset{\cG_j}.
\]
\end{proposition}
\begin{proof}
We define a functor
\[ 
T \colon \prod_{j \, \in \, I} \Rset{\Cat{G}_j} \longrightarrow  \Rset{\Big(\coprod_{j \, \in \,  I} \Cat{G}_j\Big)}  \,=\, \Rset{\cG}
\]
in the following way.
Let be \(\left(\left({X_j, \varsigma_j}\right) \right)_{j \,\in \, I}\) an object of $\prod_{j \,\in \, I} \Rset{\Cat{G}_j}$.
We define \(\widehat{\left({X_j, \varsigma_j}\right) } = \left({X_j, \widehat{\varsigma_j} }\right) \in \Rset{\cG}\) as follows.
The structure map \(\widehat{\varsigma_j} \colon X_j \longrightarrow \Cat{G}\) is such that for every \(x \in X_j\), \(\widehat{\varsigma_j}(x)= \varsigma_j(x)\).
The action
\[ \widehat{\rho_j} \colon X_j  \, {{}_{ \scriptscriptstyle{ \widehat{\varsigma_j} } } {\times}_{ \scriptscriptstyle{\mathsf{t} } }\,} \Cat{G}_{ \scriptscriptstyle{1} } \longrightarrow X_j
\]
is such that for every \(x \in X_j\) and \(g \in \Cat{G}_{ \scriptscriptstyle{1} }\) such that \(\widehat{\varsigma_j}(x)=\mathsf{t}(g)\) we have \(\widehat{\rho_j}(x,g)=\rho_j(x,g)\) where \(\rho_j \colon X_j  \, {{}_{ \scriptscriptstyle{\varsigma_j} } {\times}_{ \scriptscriptstyle{\mathsf{t} } }\,} \left({\Cat{G}_j}\right)_{ \scriptscriptstyle{1} } \longrightarrow X_j\) is the action of \(\left({X_j, \varsigma_j}\right)\).
Now we set
\[ T\left({\left(\left({X_j, \varsigma_j}\right) \right)_{j \,\in \,  I}}\right) = \biguplus_{j \, \in \,  I} \widehat{\left({X_j, \varsigma_j}\right) }.
\]
It is clear that \(T\) becomes a functor in the expected way.

In the other direction, we define
\[ 
S \colon  \Rset{\big(\coprod_{j \, \in \,  I} \Cat{G}_j\big)}   \longrightarrow  \prod_{j \, \in \,  I} \Rset{\Cat{G}_j}
\]
as follows.
Given \(\left({X, \varsigma}\right) \in \Rset{G}\), we set
\[ S\left({X, \varsigma}\right) = \left(\left({\varsigma^{-1}\left(\left({\Cat{G}_j}\right)_{ \scriptscriptstyle{0} } \right), \left.{\varsigma}\right|_{ \varsigma^{-1}\left(\left({\Cat{G}_j}\right)_{ \scriptscriptstyle{0} }\right) } }\right) \right)_{j \,\in \, I} .
\]
It is clear that \(S\) becomes  a functor in the expected way and that \(T\) and \(S\) are isomorphism of categories such that \(S= T^{-1}\).
To conclude, thanks to Corollary~\(\ref{cAdjEquivInvLaplaza}\), is it enough to prove that \(S\) is a Laplaza functor, but this follows from Proposition~\(\ref{pRestrGrpdRSet}\).
\end{proof}

\begin{proposition}
\label{pIsoGrpdIsoCatR}
Let \(\Cat{H}\) and \(\Cat{G}\) be isomorphic groupoids.
Then there is a Laplaza isomorphism of categories \(\Rset{\Cat{G}} \cong \Rset{\Cat{H}}\).
\end{proposition}
\begin{proof}
It is an immediate verification.
\end{proof}

Given a groupoid $\cG$ and a fixed object $x \in \Go$, we denote with $\cG^{\Sscript{(x)}}$ the one object subgroupoid with isotropy group \(\cG^{\Sscript{x}}\).

\begin{theorem}
\label{tEquivCatTransGrpd}
Given a transitive  and not empty groupoid \(\mathcal{G}\), let be \(a \in \mathcal{G}_{ \scriptscriptstyle{0} }\).
Then there is an equivalence of Laplaza categories
\[ \Rset{\Cat{G}} \simeq  \Rset{\Cat{G}^{\Sscript{(a)}}} .
\]
\end{theorem}
\begin{proof}
Set \(S=\Cat{G}_{ \scriptscriptstyle{0} }\).
Thanks to Propositions~\(\ref{pIsoGrpdIsoCatR}\) and Remark~\(\ref{pCharacTrGrpd}\), it is enough to prove the theorem when \(\Cat{G}= \trg{G}{S}\) and \(\Cat{G}^{\Sscript{(a)}}= \trg{G}{\Set{a}}\).
We set \(\Cat{A}=\Cat{G}^{\Sscript{(a)}}\) for brevity.
Let's define the functor
\[ \begin{aligned}{F}  \colon & {\Rset{G} } \longrightarrow {\Rset{\Cat{A}} } \\ & {\left({X, \varsigma}\right) }  \longrightarrow {\left({\varsigma^{-1}(a),  \left.{\varsigma}\right|_{\varsigma^{-1}(a)} }\right) }\end{aligned} 
\]
where \(\varsigma \colon X \longrightarrow \mathcal{G}_{ \scriptscriptstyle{0} }\) is the structure map of \(X\).
Thanks to Proposition~\(\ref{pRestrGrpdRSet}\), \(F\) is a well defined Laplaza functor, opportunely defined on morphisms, of course.

Now we have to construct a functor \(G \colon \Rset{\Cat{A}} \longrightarrow \Rset{\Cat{G}}\): for each \(\left({X, \varsigma}\right) \in \Rset{\Cat{A}}\) we define \(G\left({X, \varsigma}\right) =\left({Y, \vartheta}\right) \in \Rset{\Cat{G}}\) such that \(Y= X \times \mathcal{G}_{ \scriptscriptstyle{0} }\) and
\[ \begin{aligned}{\vartheta= \pr_{ \scriptscriptstyle{2} } }  \colon & {Y = X \times \mathcal{G}_{ \scriptscriptstyle{0} } } \longrightarrow {\mathcal{G}_{ \scriptscriptstyle{0} }} \\ & {\left({x,b}\right) }  \longrightarrow {b.}\end{aligned}
\]
Note that \(\varsigma(x)=a\) for every \(x \in X\) because \(\mathcal{A}_{ \scriptscriptstyle{0} }= \Set{a}\).
We also want to extend the action \(X  \, {{}_{ \scriptscriptstyle{ \varsigma} } {\times}_{ \scriptscriptstyle{\mathsf{t} } }\,} \mathcal{A}_{ \scriptscriptstyle{1} } \longrightarrow X\) to \(Y  \, {{}_{ \scriptscriptstyle{\vartheta} } {\times}_{ \scriptscriptstyle{\mathsf{t} } }\,} \mathcal{G}_{ \scriptscriptstyle{1} } \longrightarrow Y\).
Let be \((x,b) \in Y\) and \((b,g,d) \in \Cat{G}_{ \scriptscriptstyle{1} }\): we set \((x,b)(b,g,d) = (x(a,g,a), d)\).
It's easy to prove that the action axioms are satisfied.

Now let be \(f \colon \left({X_1, \varsigma_1}\right) \longrightarrow \left({X_2, \varsigma_2}\right)\) a morphism in \(\Rset{\Cat{A}}\): we define
\[  \begin{aligned}{ G\left({f}\right)= \left({ f \times \Id_{S} }\right) }  \colon & {\left({X_1 \times S, \pr_{ \scriptscriptstyle{2} }}\right)} \longrightarrow {\left({X_2 \times S, \pr_{ \scriptscriptstyle{2} }}\right) } \\ & {(x,b)}  \longrightarrow {\left({f(x), b}\right) .}\end{aligned} 
\]
It is easy to see that \(G\) is well defined and respects the properties of a functor.

For each \(\left({X, \varsigma}\right) \in \Rset{\Cat{A}}\) we calculate
\[ \begin{aligned}
FG\left({X, \varsigma}\right) &= F\left({X \times \mathcal{G}_{ \scriptscriptstyle{0} }, \vartheta \colon X \times S \longrightarrow S }\right) 
= \left({\vartheta^{-1}(a), \left.{\vartheta}\right|_{\vartheta^{-1}(a)} }\right)  \\
&= \left({X \times \Set{a}, \varsigma \times \Id_{a} }\right) \cong \left({X, \varsigma}\right).
\end{aligned}
\]
Since the behaviour on the morphisms is obvious, we obtain \(FG \cong \Id_{\Rset{\Cat{A}}}\).

For each \(\left({X, \varsigma}\right) \in \Rset{\Cat{G}}\) we have
\[ GF\left({X, \varsigma}\right) = G\left({\varsigma^{-1}(a), \left.{\varsigma}\right|_{\varsigma^{-1}(a)} }\right) 
= \left({\varsigma^{-1}(a) \times S, \pr_{ \scriptscriptstyle{2} } \colon \varsigma^{-1}(a) \times S \longrightarrow S  }\right) := (Y,\vartheta)
\]
We have to define a natural isomorphism \(\alpha \colon GF \longrightarrow \Id_{\Rset{\Cat{G}}}\). Thus, given \(\left({X, \varsigma }\right) \in \Rset{\Cat{G}}\), we have to define \(\alpha:=\alpha_{\left({X, \varsigma}\right)}: (Y,\vartheta) \to (X,\varsigma) \) and prove that it is a homomorphism of right \(\mathcal{G}\)-set.
According to the above notation, for each \((x,b) \in Y\) we set \(\alpha(x,b) = x \,  (a,1,b)\).
We left to the reader to check that this is a $\cG$-equivariant map turning commutative the following diagram of $\cG$-sets
\begin{equation*}
\xymatrix{
Y_1 : =\left({\varsigma_1^{-1}(a) \times S, \pr_{ \scriptscriptstyle{2} }}\right)  \ar[rrr]^{\alpha_{\left({X_1, \varsigma_1}\right)} } \ar[d]_{\left.{f}\right|_{\varsigma_1^{-1}(a)} \times \Id_{S } } & && \left({X_1, \varsigma_1}\right)  \ar[d]^{f}  \\
Y_2 := \left({\varsigma_2^{-1}(a) \times S, \pr_{ \scriptscriptstyle{2} }}\right) \ar[rrr]^{\alpha_{\left({X_2, \varsigma_2}\right)}} & && \left({X_2, \varsigma_2}\right)  
}
\end{equation*} 
and the proof is completed.
\end{proof}

\section{Cosets, conjugation of subgroupoids and the fixed points functors}\label{sec:II}
We give in this section our first result concerning the conjugacy criteria between subgroupoids of a given groupoid, and expound some illustrating examples. A discussion, form a categorical point of view, about fixed point subsets is provided here,  as well as the description, under certain finiteness conditions, of the table of marks of groupoids.

\subsection{Left (right) cosets by subgroupoids and the conjugation equivalence relation}\label{ssec:conj}
In this subsection we clarify the notion of conjugation between subgroupoids.
We first recall from \cite{Kaoutit/Spinosa:2018},  the notion of \emph{left and right cosets} attached to a morphism of groupoids.  
A subgroupoid $\cH$ of a given  groupoid $\cG$ is, by definition, a subcategory whose arrows are stable under the inverse map. Of course, in this case, the canonical ``injection'' leads  to a morphism of groupoids (see Definition \ref{def:IsotropyConj}).

Let us assume that a morphism of groupoids $\phi : \cH \to \cG$ is given and denote by  \({}^{\Sscript{\phi}}\cH(\cG)=\mathcal{H}_{ \scriptscriptstyle{0} }  {{}_{ \scriptscriptstyle{\phi_0} } {\times}}_{ \scriptscriptstyle{\Sf{t}} }\, \mathcal{G}_{ \scriptscriptstyle{1} }\) the underlying set of the  $(\cH, \cG)$-biset of Example \ref{exam:bisets}. Then  the left translation groupoid is given by 
\[ \mathcal{H} \ltimes  { ^{ \scriptscriptstyle{\phi} } \mathcal{H}\left(\mathcal{G}\right)} = \mathcal{H} \ltimes \left( \mathcal{H}_{ \scriptscriptstyle{0} }  {{}_{ \scriptscriptstyle{\phi_0} } {\times}}_{ \scriptscriptstyle{\mathsf{t}} }\, \mathcal{G}_1 \right) 
= \left( \mathcal{H}_{ \scriptscriptstyle{1} }  {{}_{ \scriptscriptstyle{\mathsf{s}} } {\times }}_{ \scriptscriptstyle{\pr_1} }\, \left( \mathcal{H}_{ \scriptscriptstyle{0} }  {{}_{ \scriptscriptstyle{\phi_0} } {\times}}_{ \scriptscriptstyle{\mathsf{t}} }\, \mathcal{G}_1  \right) , \mathcal{H}_{ \scriptscriptstyle{0} }  {{}_{ \scriptscriptstyle{\phi_0} } {\times}}_{ \scriptscriptstyle{\mathsf{t}} }\, \mathcal{G}_{ \scriptscriptstyle{1} }  \right)
\]
where the source \(s^{\ltimes}\) is the action $\lhaction$ described in equation \eqref{Eq:HG} and the target \(t^{\ltimes}\)  is the second projection on \(X\) (see \cite{Kaoutit/Spinosa:2018} for further details).  

Following \cite[Definition 3.5]{Kaoutit/Spinosa:2018},  given a morphism of groupoids \(\phi \colon \mathcal{H} \longrightarrow \mathcal{G}\), we define
\[ \left( \mathcal{G} / \mathcal{H} \right)^{\Sscript{\mathsf{R}}}_{ \scriptscriptstyle{\phi} } 
:= \pi_{ \scriptscriptstyle{0} } \Big( \mathcal{H} \ltimes { ^{ \scriptscriptstyle{\phi} } \mathcal{H}  } \left(\mathcal{G}\right) \Big),
\]
we consider the orbit set $\cH \backslash { ^{ \scriptscriptstyle{\phi} } \mathcal{H}  } \left(\mathcal{G}\right)$  and, for each \((a,g) \in \mathcal{H}_{ \scriptscriptstyle{0} }  {{}_{ \scriptscriptstyle{\phi_{ \scriptscriptstyle{0} }} } {\times}}_{ \scriptscriptstyle{\mathsf{t}} }\, \mathcal{G}_{ \scriptscriptstyle{1} }\), we set 
\[  {^{ \scriptscriptstyle{\phi} } \mathcal{H}[(a,g)]} 
= \Set{ \Big( h \lhaction (a,g) \Big) \in {^{ \scriptscriptstyle{\phi} } \mathcal{H} }\left(\mathcal{G}\right) | h \in \mathcal{H}_{ \scriptscriptstyle{1} }, \, \mathsf{s}(h)=a }.
\]

If \(\mathcal{H}\) is a subgroupoid of \(\mathcal{G}\), that is, if \(\upphi:=\tauup \colon \mathcal{H} \hookrightarrow \mathcal{G}\) is the inclusion functor,  we use the notations
\[ \left( \mathcal{G} /\mathcal{H} \right)^{\Sscript{\mathsf{R}}} 
:= \mathcal{H} \backslash \left({\mathcal{H}_{ \scriptscriptstyle{0} }  {{}_{ \, \scriptscriptstyle{\tauup_0} } {\times}}_{ \scriptscriptstyle{ \mathsf{t}} }\,  \mathcal{G}_{ \scriptscriptstyle{1} } }\right), \quad
 \mathcal{H}\left({\mathcal{G}}\right) := {^{ \scriptscriptstyle{\tauup} }\mathcal{H}} \left({\mathcal{G}}\right) 
\]
and, for each \((a,g) \in \mathcal{H}_{ \scriptscriptstyle{0} }  {{}_{\,  \scriptscriptstyle{\tauup_{ \scriptscriptstyle{0} }} } {\times}}_{ \scriptscriptstyle{\mathsf{t}} }\, \mathcal{G}_{ \scriptscriptstyle{1} }\), we set
\begin{equation}\label{Eq:Hag}
\begin{aligned}
\Hlcoset{(a,g)}={^{ \scriptscriptstyle{\tauup} }\mathcal{H}} [ \left({a,g}\right) ]
& = \LR{ \Big( h \lhaction (a,g) \Big) \in \mathcal{H}\left({\mathcal{G}}\right) | \; h \in \mathcal{H}_{ \scriptscriptstyle{1} }, \, \mathsf{s}(h)=a } \\
& = \LR{ \left( \mathsf{t}\left({h}\right), hg\right) \in \mathcal{H}\left({\mathcal{G}}\right)  | \;  h \in \mathcal{H}_{ \scriptscriptstyle{1} }, \, \mathsf{s}\left({h}\right)= \mathsf{t}\left({g} \right) = a  }.
\end{aligned}
\end{equation}

\begin{definition}(\cite[Definition 3.6]{Kaoutit/Spinosa:2018})\label{def:coset}
Let $\cH$ be a subgroupoid of $\cG$ via the injection $\tau:\cH \hookrightarrow \cG$. The  \emph{right cosets} of $\cG$ by $\cH$ is defined as
\begin{equation}\label{Eq:R} 
\left({\mathcal{G}/\mathcal{H}}\right)^{ \scriptscriptstyle{\mathsf{R}} } = \Set{ \Hlcoset{(a,g)}  | (a,g) \in \cH(\cG):= \mathcal{H}_{ \scriptscriptstyle{0} }  {{}_{\, \scriptscriptstyle{\tauup_{ \scriptscriptstyle{0} }} } {\times}}_{ \scriptscriptstyle{\mathsf{t}} }\, \mathcal{G}_{ \scriptscriptstyle{1} } ,  },
\end{equation}
where each class $\cH[(a,g)]$ is as in equation \eqref{Eq:Hag}. The set of \emph{left cosets} is similarly introduced, and we use the notation 
$$
[(g,u)]\cH:=\Big\{   (g,u) \rhaction h= (gh, s(h))|\,  h \in \Ha,\, g \in \Ga,\,   \Sf{s}(g)=u=\Sf{t}(h)  \Big\}
$$ 
for the equivalence classes in the set $\left({\mathcal{G}/\mathcal{H} }\right)^{\Sscript{\mathsf{L}} }$, where the action $\rhaction$ is the one defined in equation \eqref{Eq:GH}. 
\end{definition}

Applying \cite[Proposition~3.4]{Kaoutit/Spinosa:2018}, we obtain that, for every subgroupoid $\tau:\cH \hookrightarrow \cG$, the set of right cosets
\(\left({\mathcal{G}/\mathcal{H}}\right)^{\Sscript{\mathsf{R}}}\) becomes a right \(\mathcal{G}\)-set with structure map and action given as follows:
\begin{equation}\label{Eq:PalomaI}
\begin{aligned}{{\varsigma}  }  \colon & {\left({\mathcal{G}/\mathcal{H} }\right)^{\Sscript{\mathsf{R}}} } \longrightarrow {\mathcal{G}_{ \scriptscriptstyle{0} } } \\ &  \Hlcoset{(a, g)}  \longrightarrow {\mathsf{s}\left({g}\right) }\end{aligned} 
\qquad \text{and} \qquad
\begin{aligned}{{\rho} }  \colon & {\left({\mathcal{G}/\mathcal{H} }\right)^{\Sscript{\mathsf{R}} }  {{}_{ \scriptscriptstyle{{\varsigma} } } {\times}}_{ \scriptscriptstyle{\mathsf{t} } }\, \mathcal{G}_{ \scriptscriptstyle{1} } } \longrightarrow {\left({\mathcal{G}/\mathcal{H} }\right)^{\Sscript{\mathsf{R}} } } \\ & \big( \Hlcoset{(a,g_1)}, g_2\big)   \longrightarrow \Hlcoset{(a, g_1 g_2)}. 
\end{aligned} 
\end{equation}

In a similar way  the set of left cosets  $(\cG/\cH)^{\Sscript{\Sf{L}}}$ becomes a left $\cG$-set with structure and action maps given by: 
\begin{equation}\label{Eq:Paloma}
 \begin{aligned}{{\vartheta}  }  \colon & {\left({\mathcal{G}/\mathcal{H} }\right)^{\Sscript{\mathsf{L}}} } \longrightarrow {\mathcal{G}_{ \scriptscriptstyle{0} } } \\ &  \Hrcoset{(g, u)}  \longrightarrow {\mathsf{t}\left({g}\right) }\end{aligned} 
\qquad \text{and} \qquad
\begin{aligned}{{\lambda} }  \colon & {\mathcal{G}_{ \scriptscriptstyle{1} }  {{}_{  \scriptscriptstyle{\mathsf{s} } } {\times}}_{    \scriptscriptstyle{{\vartheta} }  }\, \left({\mathcal{G}/\mathcal{H} }\right)^{\Sscript{\mathsf{L}} }  } \longrightarrow {\left({\mathcal{G}/\mathcal{H} }\right)^{\Sscript{\mathsf{L}} } } \\ & \big(g_1, \Hrcoset{(g_2,u)}\big)   \longrightarrow \Hrcoset{(g_1g_2, u)}. \end{aligned} 
\end{equation}

To illustrate the concept, crucial in the sequel, of  conjugacy in the groupoid context, the following notion is needed.

\begin{definition}\label{def:conj0}
Let $\cG$ be a groupoid  and let $\cH$ and $\cK$ be two subgroupoids of $\cG$ with monomorphisms $\tauk: \cK \to \cG \leftarrow \cH: \tauh$. We say that $\cK$ and $\cH$ are \emph{conjugally equivalent} if there is an equivalence $F :\cK \to \cH$, between their underlying categories, and a natural transformation $\fk{g}: \tauh F \to \tauk$. It follows from the definition that  $\fk{g}$ is, actually, a natural isomorphism. Of course, if we consider $G: \cH \to \cK$, a natural inverse of $F$, then one has a natural isomorphism $\fk{h}: \tauk G \to \tauh$ given by the  inverse of the composition of the following natural isomorphisms: $\tauh \cong \tauh FG \to \tauk G$.  Thus the conjugacy relation is reflexive, symmetric and also transitive, that is, it is an equivalence relation on the set of all subgroupoids of $\cG$.
Using  elementary arguments, this definition  is equivalent to say that  
there is a functor \(F \colon \Cat{K} \longrightarrow \Cat{H}\), which is an equivalence of categories, such that there is a family \(\left(g_b\right)_{b \in \Cat{K}_{ \scriptscriptstyle{0} }}\) as follows.
For every \(b \in \Cat{K}_{ \scriptscriptstyle{0} }\) it has to be \(g_b \in \Cat{G}\left({F(b), b}\right)\) and for every arrow \(d \colon b_1 \longrightarrow b_2\) in \(\Cat{K}\) it has to be \(F(d) = g_{b_2}^{-1} d g_{b_1}\), which justifies  the terminology.
\end{definition}

\begin{theorem}\label{thm:LR}
Let $\cH$ and $\cK$ be two subgroupoids of a given groupoid $\cG$. 
Then the following conditions are equivalent:
\begin{enumerate}[(i)]
\item $\big(\cG/\cH\big)^{\Sscript{\mathsf{R}}} \, \cong\,  \big(\cG/\cK\big)^{\Sscript{\mathsf{R}}}$  as right $\cG$-sets; 
\item There are morphisms of groupoids  $F: \cK \rightarrow \cH$ and $G: \cH  \to \cK$ together with two  natural transformations $\fk{g}: \tauh F \to \tauk$ and $\fk{f}: \tauk G \to \tauh$.
\item The subgroupoids $\cH$ and $\cK$ are conjugally equivalent.
\item There are families \(\left(u_{\Sscript{b}}\right)_{b \in \Cat{K}_{ \scriptscriptstyle{0} }}\) and \(\left(g_{\Sscript{b}}\right)_{b \in \Cat{K}_{ \scriptscriptstyle{0} }}\) with \(u_{\Sscript{b}} \in \Cat{H}_{ \scriptscriptstyle{0} }\) and \(g_{\Sscript{b}} \in \Cat{G}\left({u_{\Sscript{b}}, b}\right)\) for every \(b \in \Cat{K}_{ \scriptscriptstyle{0} }\), such that:
\begin{enumerate}[(a)]
\item for each \(b_1, b_2 \in \Cat{K}_{ \scriptscriptstyle{0} }\) we have \(g_{\Sscript{b_2}}^{-1} \Cat{K}\left({b_1, b_2}\right) g_{\Sscript{b_1}} = \Cat{H}\left({u_{\Sscript{b_1}}, u_{\Sscript{b_2}}}\right)\);
\item for each \(u \in \Cat{H}_{ \scriptscriptstyle{0} }\) there is \(z \in \Cat{K}_{ \scriptscriptstyle{0} }\) such that \(\Cat{H}\left({u_z, u}\right) \neq \emptyset\).
\end{enumerate}
\item  $\big(\cG/\cH\big)^{\Sscript{\mathsf{L}}} \, \cong\,  \big(\cG/\cK\big)^{\Sscript{\mathsf{L}}} $ as left $\cG$-sets.
\end{enumerate}
\end{theorem}
\begin{proof}
$(i) \Rightarrow (ii)$. Let us assume that there is a $\cG$-equivariant  isomorphism $\cF: \big(\cG/\cK\big)^{\Sscript{\mathsf{R}}} \to  \big(\cG/\cH\big)^{\Sscript{\mathsf{R}}}$ and for each class of the form $\cK[(b,\iota_{\Sscript{b}})] \in \big(\cG/\cK\big)^{\Sscript{\mathsf{R}}}$,  denote by $\cH[(a_{\Sscript{b}}, g_{\Sscript{b}})]$ its image by $\cF$. Thus, for any $b \in \Ko$, there could be  many objects $a_{\Sscript{b}} \in \Ho$ such that $\cF\big( \cK[(b,\iota_{\Sscript{b}})] \big) = \cH[(a_{\Sscript{b}}, g_{\Sscript{b}})]$ and, moreover, two  such objects $a_{\Sscript{b}}$ and $a'_{\Sscript{b}}$ are  isomorphic.
Therefore, we can  make a single choice by taking a representative element, which will be denoted by $F_{\Sscript{0}}(b)$, and we will have  $\cF\big( \cK[(b,\iota_{\Sscript{b}})] \big) = \cH[(F_{\Sscript{0}}(b), g_{\Sscript{b}})]$, according to this choice.  
As a consequence, we have a map $F_{\Sscript{0}}: \Ko \to \Ho$, which will be the object function of the functor we are going to built. On the other hand, considering the definition of $\cF$,  we have $\Sf{s}(g_{\Sscript{b}})=b$ and $\Sf{t}(g_{\Sscript{b}})=F_{\Sscript{0}}(b)$, thus we obtain a family of arrows $\{\fk{g}_{\Sscript{b}}: F_{\Sscript{0}}(b) \to b\}_{\Sscript{b \, \in \, \Ko}}$.
Now, given an arrow $k : b \to b'$  in $\Ka$, we have $\cK[(b,\iota_{\Sscript{b}})]=\cK[(k \lhaction (b,\iota_{\Sscript{b}}))] = \cK[(b', k)]$, which implies
$$
\cF\big( \cK[(b,\iota_{\Sscript{b}})]  \big)\,=\, \cF \big(\cK[(b', k)]  \big)   \;  \Longrightarrow \;  \cH[(F_{\Sscript{0}}(b), g_{\Sscript{b}})] \,=\,  \cH[(F_{\Sscript{0}}(b'), g_{\Sscript{b'}}k)]. 
$$
Therefore, there is a unique arrow $h \in \cH(F_{\Sscript{0}}(b), F_{\Sscript{0}}(b'))$ such that we have the equality $h g_{\Sscript{b}}=g_{\Sscript{b'}} k$ in $\Ga$.
In this way we can construct a map, at the level of arrows, $F_{\Sscript{1}}: \Ka \to \Ha$ with the property that, for any $k \in \cK(b, b')$, we have 
$k \fk{g}_{\Sscript{b}} = \fk{g}_{\Sscript{b'}} F_{\Sscript{1}}(k) $ as an equality in $\Ga$. The properties of groupoid action show that $F: \cK \to \cH$ is actually a functor with a natural transformation $\fk{g}: \tauh F \to \tauk$, as  claimed in $(ii)$. To complete the proof of this implication, it's enough to use the inverse $\cG$-equivariant map of $\cF$ to construct, in a similar way, the functor $G$ with the required properties. 

$(ii) \Rightarrow (iii)$ We only need to check that $F$ and $G$ establish an equivalence of categories. This follows directly  from the fact that $\tauk$ and $\tauh$ are faithful functors.

$(iii) \Rightarrow (iv)$.  Let $F : \cK \to \cH$ be the given equivalence of categories and $\fk{g}$ the accompanying natural transformation. For each element $b \in\Ko$ we set $u_{\Sscript{b}}=F(b) \in \Ho$ and $g_{\Sscript{b}}= \fk{g}_{\Sscript{b}} \in \cG(u_{\Sscript{b}}, b)$. Condition $(a)$ follows now from the naturality of $\fk{g}$, while condition $(b)$ from the fact that $F$ is an essentially surjective functor. 

$(iv) \Rightarrow (v)$. We define the following map:
$$
\psi: \big(\cG/\cK\big)^{\Sscript{\mathsf{L}}} \longrightarrow  \big(\cG/\cH\big)^{\Sscript{\mathsf{L}}}, \quad \Big( [(g,b)]\cK \longmapsto [(gg_{\Sscript{b}}, u_{\Sscript{b}})]\cH   \Big).
$$
Condition $(a)$ in the statement implies that $\psi$ is a well defined and injective $\cG$-equivariant map. Let us check that it is also surjective.  Let \(\left[{\left({p, u}\right) }\right]\Cat{H} \in \left({\Cat{G}/\Cat{H} }\right)^{ \scriptscriptstyle{\Sscript{\mathsf{L}} } }\): thanks to the condition $(b)$ there is \(z \in \Cat{K}_{ \scriptscriptstyle{0} }\) such that there is \(h \in \Cat{H}\left({u_z, u}\right)\). The situation is as follows:
\[ 
\xymatrix{z \ar[r]^{g_{\Sscript{z}}^{-1}} & u_{\Sscript{z}} \ar[r]^{h} & u  \ar[r]^{p} & \Sf{t}(p).}
\]
Computing
\[ 
\psi \Big( \left[{\left({phg_{\Sscript{z}}^{-1} , z}\right) }\right]\Cat{K} \Big) = \left[{\left({ph, u_{\Sscript{z}}}\right) }\right] \Cat{H}
=  \left[{\left({p, u}\right)  \rhaction h}\right] \Cat{H} =  \left[{\left({p, u}\right) }\right] \Cat{H}
\]
we obtain that \(\psi\) is surjective.

$(v) \Rightarrow (i)$. Uses the isomorphism of categories between  right $\cG$-sets and  left $\cG$-sets.
\end{proof}

\begin{definition}\label{def:conj}
Let $\cH$ and $\cK$ be two subgroupoids of a given groupoid $\cG$.  We say that $\cH$ and $\cK$ are \emph{conjugated} if one of the equivalent conditions in Theorem \ref{thm:LR} is fulfilled.   Obviously,  conjugated subgroupoids have equivalent underlying categories, which means that  they are not necessarily isomorphic as groupoids.   Therefore, in contrast with the classical case of  groups or that of disjoint union of groups,  conjugated subgroupoids are not necessarily isomorphic (see Example  \ref{exam:Nconj} for explicit situations and further remarks). 
\end{definition}

\begin{definition}\label{def:ConIsotropy}
Given a groupoid \(\mathcal{G}\), let be \(a, b \in \mathcal{G}_{ \scriptscriptstyle{0} }\), \(H\) a subgroup of \(\mathcal{G}^{a}\) and \(K\) a subgroup of \(\mathcal{G}^{b}\).
We say that \(H\) and \(K\) are \textit{conjugated isotropy subgroups} if there is \(d \in \mathcal{G}\left({a, b}\right)\) such that \(K= dHd^{-1}\).
\end{definition}

\begin{example}\label{exam:IsotropyNotConj} 
There is a groupoid \(\Cat{G}\) with not empty subgroupoids \(\Cat{H}\) and \(\Cat{K}\) which are conjugally equivalent and satisfy the following property: there are \(x \in \Cat{H}_{ \scriptscriptstyle{0} }\) and \(w \in \Cat{K}_{ \scriptscriptstyle{0} }\) such that \(\Cat{H}^{x}\) and \(\Cat{K}^{w}\) are not conjugated.
Namely, let us consider a group \(G\), with two distinct subgroups \(A\) and \(B\) which are not conjugated, and a set \(S\) with at least four elements \(x\), \(y\), \(z\) and \(w\).
Set \(\Cat{G}= \trg{G}{S}\), as in Remark \ref{pCharacTrGrpd}, we construct two subgroupoids \(\Cat{H}\) and \(\Cat{K}\) of \(\Cat{G}\), with only loops as arrows, in the following way.
We set \(\Cat{H}_{ \scriptscriptstyle{0} } = \Set{x, z}\), \(\Cat{K}_{ \scriptscriptstyle{0} } = \Set{y, w}\),
\[ \Cat{H}^{x} = (x, A, x), \qquad \Cat{H}^{z}= (z, B, z), \qquad \Cat{K}^{y}= (y, A, y),
\qquad \text{and} \qquad \Cat{K}^{w}= (w, B, w)
\]
where we made the abuse of notation \((x, A, x) = \Set{x} \times A \times \Set{x}\).
We want to prove the condition \((iv)\) of Theorem~\(\ref{thm:LR}\): to this purpose we set \(u_y= x\) and \(u_w= z\).
We obtain \(\Cat{H}\left({u_y, x}\right) = \Cat{H}^{x} \neq \emptyset\) and \(\Cat{H}\left({u_w, z}\right) = \Cat{H}^{z} \neq \emptyset\) thus \((b)\) is proved.
Since all the arrows of \(\Cat{H}\) and \(\Cat{K}\) are loops we just have to prove that there are \(g_y  \in \Cat{G}\left({u_y, y}\right) = \Cat{G}\left({x,y}\right)\) and \(g_w \in \Cat{G}\left({u_w, w}\right) = \Cat{G}\left({z, w}\right)\) such that
\[ g_y^{-1} \Cat{K}\left({y, y}\right) g_y = \Cat{H}\left({u_y, u_y}\right)
\qquad \text{and} \qquad
g_w^{-1} \Cat{K}\left({w, w}\right) g_y = \Cat{H}\left({u_w, u_w}\right).
\]
To this end we set \(g_y = (y, 1, x)\) and \(g_w = (w, 1, z)\) and we calculate
\[ g_y^{-1} \Cat{K}\left({y, y}\right) g_y = g_y^{-1} \Cat{K}^{y} g_y 
= (x, 1, y)(y, A, y)(y, 1, x)
= (x, A, x)
= \Cat{H}^{x} 
= \Cat{H}\left({u_y, u_y}\right)
\]
and
\[ g_w^{-1} \Cat{K}\left({w, w}\right) g_w = g_w^{-1} \Cat{K}^{w} g_w 
= (z, 1, w)(w, B, w)(w, 1, z)
= (z, B, z)
= \Cat{H}^{z} 
= \Cat{H}\left({u_w, u_w}\right)
\]
proving~\((a)\) and, thus, the claim.
Now, by contradiction, let be \(d \colon w \longrightarrow x\) such that \(\Cat{K}^{w} = d^{-1} \Cat{H} d\).
Of course, there has to be \(g \in G\) such that \(d = (x, g, w)\).
Calculating
\[ (w, B, w) = \Cat{K}^{w} 
= d^{-1} \Cat{H}^{x} d
= \left({w, g^{-1},  x}\right) \left({x, A, x}\right) \left({x, g, w}\right)
= \left({w, g^{-1} A g, w}\right) 
\]
we obtain \(g^{-1} A g = B\), which  contradicts the choice of $A$ and $B$.
\end{example}

\begin{proposition}
\label{pConjSGrpdTrans}
Given a groupoid \(\Cat{G}\), let's consider two conjugated subgroupoids \(\Cat{H}\) and \(\Cat{K}\).
Then \(\Cat{H}\) is transitive if and only if \(\Cat{K}\) is transitive.
Moreover, in this case, every isotropy group of \(\Cat{H}\) is conjugated to every isotropy group of \(\Cat{K}\).
\end{proposition}
\begin{proof}
Let's assume \(\Cat{H}\) transitive and let't consider \(b_1, b_2 \in \Cat{K}_{ \scriptscriptstyle{0} }\): for \({i}\in \Set{{1},{2}}\) there are \(u_{b_i} \in \Cat{H}_{ \scriptscriptstyle{0} }\) and \(g_{b_i} \in \Cat{G}\left({u_{b_i}, b_i}\right)\) such that
\[ g_{b_2}^{-1} \Cat{K}\left({b_1, b_1}\right) g_{b_1} = \Cat{H}\left({u_{b_1}, u_{b_2}}\right) 
\]
therefore \(\Cat{K}\left({b_1, b_2}\right) \neq \emptyset\).
If we assume \(\Cat{K}\) to be transitive, we can obtain \(\Cat{H}\) to be transitive using \((iv)\)  of Theorem~\(\ref{thm:LR}\) with the two subgroupoids exchanged.

Now let's assume \(\Cat{H}\) and \(\Cat{K}\) to be transitive and let be \(u \in \Cat{H}_{ \scriptscriptstyle{0} }\) and \(v \in \Cat{K}_{ \scriptscriptstyle{0} }\).
Thanks to \((iv)\) of Theorem~\(\ref{thm:LR}\) there are \(u_v \in \Cat{H}_{ \scriptscriptstyle{0} }\) and \(g_v \in \Cat{G}\left({u_v, v}\right)\) such that \(g_v^{-1} \Cat{K}^{v} g_v = \Cat{H}^{u_v}\).
Since \(\Cat{H}\) is transitive there is \(h \in \Cat{H}\left({u_v, u}\right)\) such that \(\Cat{H}^{u_v} = h^{-1} \Cat{H}^{u} h\) therefore
\[ \left({g_v h^{-1}}\right)^{-1} \Cat{K}^{v}\left({g_v h^{-1}}\right) = h g_v^{-1} \Cat{K}^{v} g_v h^{-1} 
= \Cat{H}^{u},
\]
which shows that $\cH^{\Sscript{u}}$ and $\cK^{\Sscript{v}}$ are conjugated, and finishes the proof.
\end{proof}

\begin{corollary}
\label{pGrpdQuotRXSetIsom}
Now let's consider \(a, b \in \mathcal{G}_{ \scriptscriptstyle{0} }\), \(H \le \mathcal{G}^{a}\) and \(K \le \mathcal{G}^{b}\).
Then  \(H\) and $K$ induce subgroupoids \(\mathcal{H}\) and $\cK$ of \(\mathcal{G}\) such that \(\mathcal{H}_{ \scriptscriptstyle{0} }=\Set{a}\) ,  \(\mathcal{H}_{ \scriptscriptstyle{1} }=\mathcal{H}^{a}=H\) and  \(\mathcal{K}_{ \scriptscriptstyle{0} }=\Set{b}\) , \(\mathcal{K}_{ \scriptscriptstyle{1} }=\mathcal{K}^{b}=K\). Moreover,  $\cH$ and $\cK$ are conjugated subgroupoids if and only if $H$ and $K$ are conjugated isotropy subgroups.
\end{corollary}
\begin{proof}
It follows from Proposition~\(\ref{pConjSGrpdTrans}\).
\end{proof}

It is straightforward to see that conjugation induces an equivalence relation $\simc$ on the set $\SubG$  of all subgrouppoids  of $\cG$ with only one object. The equivalence class of a given element $\cH$ in $\SubG$ will be denoted by $[\cH]$. Notice that any subgroup of an  isotropy group of  $\cG$ can be considered as a subgroupoid with only one object (see Definition \ref{def:IsotropyConj}(1)) and, consequently, as an element of $\SubG$. The converse is, by definition, also true. We denote by $\rep(\SubG)$ a set of representative elements of $\SubG$ modulo the equivalence relation $\simc$. It is clear that this equivalence relation is extended to the whole set of all subgroupoids of $\cG$.

\begin{example}\label{exam:Nconj}
In contrast with the classical case of groups the conjugacy relation differs form the isomorphism relation. Here we give  examples of two subgroupoids which are  isomorphic but not conjugated, as well as two subgroupoids which are conjugated but not isomorphic.
\begin{itemize}
\item Let us show that there is a groupoid \(\Cat{G}\) with two subgroupoids \(\Cat{H}\) and \(\Cat{K}\) which are isomorphic but not conjugated. Namely, given a set \(J \neq \emptyset\), let's consider \(A, B \subseteq J\) such that \(A \neq \emptyset \neq B\) and \(\left\lvert{A}\right\rvert = \left\lvert{B}\right\rvert\).
Given an abelian group \(G\), the relation of conjugacy is the same of the relation of equality, thus if we consider two distinct and isomorphic subgroups \(H\) and \(K\) of \(G\) they are not conjugated.
In particular, given an abelian group \(U\), possible choices are \(G= U \times U\), \(H= U \times 1\) and \(K = 1 \times U\).
Now let us consider the induced groupoids  \(\Cat{G}=\trg{G}{J}\), \(\Cat{H}= \trg{H}{A}\) and \(\Cat{K}= \trg{K}{B}\) (see Example \ref{exam:induced} and Remark \ref{pCharacTrGrpd} for the notations).
Thanks to Remark \ref{pIsomTrGrpdGS}, we know that  the groupoids \(\Cat{H}\) and \(\Cat{K}\) are isomorphic.
By contradiction, let's assume that the subgroupoids \(\Cat{H}\) and \(\Cat{K}\) are conjugated.
Then, by Theorem \ref{thm:LR}$(iv)$, there are families \(\left(a_j\right)_{j \in B}\) and \(\left(g_j\right)_{j \in B}\) such that \(a_j \in A\), \(g_j \in \Cat{G}\left({a_j, j}\right)\) and \(g_j^{-1} \Cat{K}^{j} g_j = \Cat{H}^{a_j}\) for each \(j \in B\).
By definition, for each \(j \in B\) there is \(\eta_j \in G\) such that \(g_j= \left({j, \eta_j, a_j}\right)\) thus
\[ \begin{aligned}
\Set{a_j} \times H \times \Set{a_j} &=  \Cat{H}^{a_j} 
= g_j^{-1} \Cat{K}^{j} g_j
= \left({a_j, \eta_j^{-1}, j}\right) \left({\Set{j} \times K \times \Set{j} }\right) \left({j, \eta_j, a_j}\right) \\
&= \Set{a_j} \times \eta_j^{-1} K \eta_j \times \Set{a_j} .
\end{aligned}
\]
As a consequence we obtain \(H = \eta_j^{-1} K \eta_j\), which is a contradiction with the above choices made for $G$, $H$ and $K$.

\item Let us check that there is a groupoid \(\Cat{G}\) with two subgroupoids \(\Cat{B}\) and \(\Cat{A}\) which are conjugated but not isomorphic. To this end,  we consider two subsets \(A\) and \(B\) of a given set \(J\)
such that \(\emptyset \neq A \subseteq B \subseteq J\), and we construct the groupoids of pairs \(\Cat{G}= (J \times J, J)\), \(\Cat{B}= (B \times B, B)\) and \(\Cat{A}= (A \times A, A)\) (see Example \ref{exam:X} for the definition), where we consider $\cA$ and $\cB$ as subgroupoids of $\cG$.
Since \(A \neq \emptyset\) we can choose \(a \in A\)  and, for every \(b \in B\), we define the families \(\left(u_b\right)_{b \in B}\) and \(\left(g_b\right)_{b \in B}\) as follows:
\[ u_b = 
\begin{cases}
b,  \quad & b \in A \\
a, \quad & b \in B \setminus A
\end{cases}
\qquad \text{and} \qquad
g_b = 
\begin{cases}
(b,b),  \quad & b \in A \\
(b,a), \quad & b \in B \setminus A.
\end{cases}
\]
We have to check that for each \(b_1, b_2 \in \Cat{B}_{ \scriptscriptstyle{0} }\), we have \(g_{b_2}^{-1} \Cat{B}\left({b_1, b_2}\right) g_{b_1} = \Cat{A}\left({u_{b_1}, u_{b_2}}\right)\) but this condition is trivially satisfied in a groupoid of pairs.
Now for each \(\alpha \in A\) we have to check that there is \(b \in B\) such that \(\Cat{A}\left({u_b, \alpha}\right) \neq \emptyset\) but this is obvious: it is enough to choose \(b = \alpha\).
Lastly, the subgroupoids \(\Cat{A}\) and \(\Cat{B}\) are not isomorphic if \(|A| \lneq  |B|\).
\end{itemize}
\end{example}

\begin{remark}\label{rem:isotropy}
Given a groupoid \(\Cat{G}\), let's consider a subgroupoid \(\Cat{I}\) of $\cG$ with a single object \(a\), that is,  $\Io=\{a\}$.  Set $I=\Ia \leq \cG^{\Sscript{a}}$, we can construct a $\cG^{\Sscript{a}}$-equivariant injective map
$$
\cG^{\Sscript{a}}/I \longrightarrow (\cG/\cI)^{\Sscript{\Sf{R}}}, \qquad \Big(  Ip \longmapsto \cI[(a,p)] \Big).
$$
Moreover, if we assume that $\cG^{\Sscript{a}}$ and $\Go$ are finite sets, then the set of right cosets $(\cG/\cI)^{\Sscript{\Sf{R}}}$ must also be finite.
\end{remark}
\begin{proof}
We have
\[ \left({\Cat{G}/\Cat{I} }\right)^{\Sscript{\mathsf{R}}} = \Set{ \Cat{I}\left[{\left({a, p}\right) }\right]  |  p \in \Cat{G}_{ \scriptscriptstyle{1} }, \,  \mathsf{t}(p)= a }
\]
therefore, if we denote with \(\Cat{G}^{\Braket{a}}\) the connected component of \(\Cat{G}\) containing \(a\), using the characterization of transitive groupoids, we obtain
\[ \left\lvert{\left({\Cat{G}/\Cat{I} }\right)^{\Sscript{\mathsf{R}}}}\right\rvert \le \left\lvert{\Cat{G}^{\Braket{a} } }\right\rvert 
= \left\lvert{\left({\Cat{G}^{\Braket{a} } }\right)_{ \scriptscriptstyle{0} } }\right\rvert \times \left\lvert{\Cat{G}^{a} }\right\rvert \times \left\lvert{\left({\Cat{G}^{\Braket{a} } }\right)_{ \scriptscriptstyle{0} } }\right\rvert
\le  \left\lvert{ \left({\Cat{G} }\right)_{ \scriptscriptstyle{0} }}\right\rvert \times \left\lvert{\Cat{G}^{a} }\right\rvert  \times \left\lvert{\Cat{G}_{ \scriptscriptstyle{0} } }\right\rvert < \infty.
\]
\end{proof}

The following lemma will be used in subsequent sections.
\begin{lemma}\label{lema:Triangular}
Let $\cG$ be a groupoid and consider two elements $\cH$ and $\cK$  in $\rep\big(\SubG \big)$. Then we have the following properties:
\begin{enumerate}[(i)]
\item The set of $\cG$-equivariant maps $\GHom{(\cG/\cK))^{\Sscript{\Sf{R}}}}{(\cG/\cH))^{\Sscript{\Sf{R}}}} $ is a not empty set if and only if 
there exists  $g \in \cG(b,a)$  such that $ \Ka \subseteq g^{-1} \Ha g$, where $\Ha \leq \cG^{\Sscript{a}}$ and  $\Ka \leq \cG^{\Sscript{b}}$, for some $a, b \in \Go$.
\item Assume that all isotropy groups of  $\cG$  are  finite groups. Then the following implication 
$$
\GHom{(\cG/\cK))^{\Sscript{\Sf{R}}}}{(\cG/\cH))^{\Sscript{\Sf{R}}}}  \neq  \emptyset \; \Longrightarrow \;  \GHom{(\cG/\cH))^{\Sscript{\Sf{R}}}}{(\cG/\cK))^{\Sscript{\Sf{R}}}} \, =\, \emptyset
$$
holds true, whenever $\cH \neq \cK$.
\end{enumerate} 
\end{lemma}
\begin{proof}
$(i)$ Assume we have a $\cG$-equivariant map $F :  {(\cG/\cK))^{\Sscript{\Sf{R}}}} \to {(\cG/\cH))^{\Sscript{\Sf{R}}}} $ and set  $F\big(\cK[(b,\iota_{\Sscript{b}})]\big)\,=\,  \cH[(a,g)] \in (\cG/\cH))^{\Sscript{\Sf{R}}}$. Then, by definition, we know that $g \in \cG(b,a)$. Taking $k \in \Ka$, we compute
$$
\cH[(a,g)]= F\big(\cK[(b,\iota_{\Sscript{b}})]\big)=F\big(\cK[(b,k )]\big)= F\big(\cK[(b,\iota_{\Sscript{b}})]\, k\big)=F\big(\cK[(b,\iota_{\Sscript{b}})]\, \big) \,  k= \cH[(a,g)]\, k= \cH[(a,g k)].
$$
This means that there exists $h \in \Ha$ such that $hg=gk$. Therefore, we have $\Ka \subseteq g^{-1}\Ha g$. 
The other implication is proved in a similar way.

$(ii)$ Assume by contradiction that we also have that $\GHom{(\cG/\cH))^{\Sscript{\Sf{R}}}}{(\cG/\cK))^{\Sscript{\Sf{R}}}} \neq  \emptyset$. Applying the first part, we get that there exist $g_1 \in \cG(b,a)$ and $g_2 \in \cG(a,b)$ such that $ \Ka \subseteq g_1^{-1}\Ha g_1$ and $\Ha \subseteq g_2^{-1} \Ka g_2$. Thus $\Ha$ and $\Ka$ have the same cardinality as subsets of $\Ga$.  Since $\cG^{\Sscript{b}}$ is a finite group, we obtain $\Ka = g_1^{-1}\Ha g_1$. This means that $\cH$ and $\cK$ are conjugated, that is, they represent the same class in $\SubG/\simc$, which is a contradiction because, by hypothesis, $\cH \neq \cK$ as elements in $\rep\Big(\SubG\big)$.
\end{proof}

\subsection{Fixed points subsets of groupoid-sets and the table of marks of finite groupoids}\label{ssec:Fixed}
This subsection deals with  the fixed point subsets of groupoid-sets under subgroupoid actions and discusses their functorial properties. Moreover we give the analogue notion of what is known in the classical theory as \emph{the table of marks} attached to a given (finite) groupoid  \cite{Burnside:1911}. 

\begin{definition}\label{def:fixed}
Given a groupoid \(\mathcal{G}\), let \(\mathcal{H}\) be a subgroupoid of \(\mathcal{G}\) and let \(\left({X, \varsigma}\right)\) be a right \(\mathcal{G}\)-set.
We define the \emph{set of fixed points by \(\mathcal{H}\) in \(X\)} as
\[ 
X^{\Sscript{\cH}} = \Big\{ x \in X | \, \forall h \in \mathcal{H}_{ \scriptscriptstyle{1} } \text{ such that } \varsigma\left({x}\right)= \mathsf{t}\left({h}\right),\, \text{we have that } \, xh= x\Big\}.
\]
Notice that not any subgroupoid is allowed to stabilize the elements of a given right $\cG$-set $(X,\varsigma)$. More precisely, the set of fixed point $X^{\Sscript{\cH}}$ can be introduced only for those  subgroupoids $\cH$ satisfying the condition $\Ho \cap \varsigma(X) \neq \emptyset$ and possessing at most one object. If this is not the case, then it implicitly stands from the definition that we are setting $X^{\Sscript{\cH}} = \emptyset$.
\end{definition}

On the other hand, if $\cH$ and $\cH'$ are conjugated subgroupoids of $\cG$ with only one object (see Corollary \ref{pGrpdQuotRXSetIsom}), then $X^{\Sscript{\cH}}$ and $X^{\Sscript{\cH'}}$ are clearly in bijection.
In the following  result we collect the most useful properties of the functor of fixed points subsets under subgroupoid (with only one object) actions. 

\begin{proposition}\label{prop:Hom}
Let $\cH$ be a subgroupoid of $\cG$ with only one object. Then we have the following natural isomorphism 
$$
\GHom{(\cG/\cH)^{\Sscript{\Sf{R}}}}{X} \, \simeq \, X^{\Sscript{\cH}},
$$
of sets, for every right $\cG$-set $(X,\varsigma)$. In particular, if  \(F \colon (X, \varsigma) \longrightarrow (Y, \vartheta)\) is a  \(\mathcal{G}\)-equivariant map, then \(F\) induces a function
\[ 
{\tilde{F} }  \colon  {X^{\Sscript{\mathcal{H}}} } \longrightarrow {Y^{\Sscript{\mathcal{H}}} }, \qquad  \Big(  {x}  \longrightarrow {\tilde{F}\left({x}\right)= F\left({x}\right) } \Big)
\]
such that if \(F\) is an isomorphism of $\cG$-sets, then  \(\tilde{F}\) is a bijection. \\ Furthermore, 
given a disjoint union $X = \Uplus_{\Sscript{ i \, \in \, I}} X_i$ of right $\cG$-sets,  we have a natural bijection 
\begin{equation}\label{Eq:XHI}
X^{\Sscript{\cH}} \, \simeq\,  \underset{\Sscript{ i \, \in \, I}}{\Uplus} \,  X_i^{\Sscript{\cH}}.
\end{equation}
In particular, the functor $\GHom{(\cG/\cH)^{\Sscript{\Sf{R}}}}{-}: \rGsets \to \Sets$ commutes with arbitrary coproducts.  
\end{proposition}
\begin{proof}
The crucial natural isomorphism is given as follows:
$$
\xymatrix@R=0pt{  \GHom{(\cG/\cH)^{\Sscript{\Sf{R}}}}{X} \ar@{->}^-{}[rr] & &   X^{\Sscript{\cH}}  \\  f \ar@{|->}[rr] & & f \big( \cH[(a,\iota_{a})] \big) \\ \Big[ \cH[(a,g)] \mapsto xg \Big] & & x \ar@{|->}[ll]   }
$$
where $\Ho=\{a\}$ and the notation is the pertinent one. The rest of the proof is now a direct verification.
\end{proof}

\begin{corollary}\label{coro:cuadro}
Let $\cH$ and $\cH'$ be two subgroupoids of $\cG$ both with only one object.  Assume that we have a $\cG$-equivariant map $F: (\cG/\cH)^{\Sscript{\Sf{R}}} \to (\cG/ \cH')^{\Sscript{\Sf{R}}}$. Then, for any right $\cG$-set $(X,\varsigma)$, we have a commutative diagram:
\begin{equation}\label{Eq:cuadro}
\begin{gathered}
\xymatrix{  \GHom{(\cG/\cH')^{\Sscript{\Sf{R}}}}{X}  \ar@{->}^-{\cong}[rr]   \ar@{->}_-{\Sscript{\GHom{F}{X}}}[d]  && X^{\Sscript{\cH'}}   \ar@{-->}^-{\Sscript{X^F}}[d] \\  \GHom{(\cG/\cH)^{\Sscript{\Sf{R}}}}{X}   \ar@{->}^-{\cong}[rr] & & X^{\Sscript{\cH}} .}
\end{gathered}
\end{equation}
In particular, if $\cH$ and $\cH'$ are conjugated,  we have a bijection $X^{\Sscript{\cH}} \simeq X^{\Sscript{\cH'}}$.
\end{corollary}
\begin{proof}
Set $\Ho=\{a\}$, $\cH'_{\Sscript{0}}=\{a'\}$ and let $g \in \cG(a,a')$ be the arrow attached to the $\cG$-equivariant map $F$, that is, $g$ is determined by the equality $F\big(\cH[(a, \iota_{\Sscript{a}})]\big)=\cH'[(a',g)]$. Then the stated map $X^{\Sscript{F}}: X^{\Sscript{\cH'}} \to X^{\Sscript{\cH}}$ is a given by $x \mapsto xg$. The desired diagram commutativity is now clear from the the involved maps. The particular case is a direct consequence of the first claim.  
\end{proof}

\begin{remark}\label{rem:XH}
Let $\cG$ be a groupoid and consider as before the set $\SubG$ of all subgroupoids with only one object.  One can define a category whose objects set is $\SubG$  and,  given two objects $\cH, \cH' \in \SubG$, the set of arrows from $\cH'$ to $\cH$ is the set of all $\cG$-equivariant maps $\GHom{(\cG/\cH)^{\Sscript{\Sf{R}}}}{(\cG/\cH')^{\Sscript{\Sf{R}}}}$. This category is also denoted by $\SubG$. In this way, the set $\rep(\SubG)$ is then identified with the skeleton   of the category  $\SubG$. On the other hand,  for any right $\cG$-set $(X, \varsigma)$ we obtain as, in  Corollary \ref{coro:cuadro},  a  functor $\cH \to X^{\Sscript{\cH}} $, which is naturally isomorphic to the functor $\cH \to \GHom{(\cG/\cH)^{\Sscript{\Sf{R}}}}{X}$.
\end{remark}

Next we discuss  the cardinality of the fixed point subsets of right $\cG$-cosets by subgroupoids with a single object, that is, by elements of $\SubG$.   First, we give the definition of the notion of  finite groupoid, which we will deal with. 

\begin{definition}\label{def:finiteG}
Given a groupoid $\cG$, we say that $\cG$ is \emph{strongly finite} if its set of arrows $\Ga$ is finite. This obviously implies that $\Go$ and $\pi_{\Sscript{0}}(\cG)$ are finite sets.  We say that $\cG$ is \emph{locally strongly  finite} provided that the category $\SubG$ of Remark \ref{rem:XH} is \emph{skeletally finite}  and  each of the isotropy groups of $\cG$ is a finite set. Here skeletally finite means that $\SubG/\simc$ is a finite set.  Evidently, any strongly finite groupoid is  locally strongly finite.  On the other hand it could happens that a groupoid has each of its isotropy groups finite, but  $\rep(\SubG)$ is not. To see this, it suffices to look at the class of  not transitive groupoids with trivial isotropy groups and with an infinite  number of connected components. More precisely, one can take a groupoid of the form $\uplus_{i \, \in \,  I} \cX_i$, where $I$ is an infinite set and each of the groupoids $\cX_i$ is  one of those  expounded in  Example \ref{exam:X}. 
\end{definition}

Let $\cH$ be a subgroupoid of $\cG$. From now on we will denote by $\cG/\cH: = (\cG/\cH)^{\Sscript{\Sf{R}}}$ the set of right $\cG$-cosets.

\begin{proposition}\label{prop:Rep}
Let $\cG$ be a groupoid and consider  the quotient sets $\SubG/\simc$ and $\pi_{\Sscript{0}}(\cG)$. For each $a \in \pi_{\Sscript{0}}(\cG)$ we denote as before by $\cG^{\Sscript{\lar{a}}}$  the connected component  subgroupoid of $\cG$ containing $a$ (this is clearly a transitive groupoid).  We  consider in a canonical way $\SubGx{a}/\simc$ as a subset  of $\SubG/\simc$. Then: 
\begin{enumerate}[(i)]
\item We have a disjoint union 
$$
\SubG/\simc \,\, =\,\, \biguplus_{ a \, \in \,  \pi_{\Sscript{0}}(\cG)} \Big( \SubGx{a}/\simc \Big).
$$
\item If $\cG$ is locally strongly finite, then  $\pi_{\Sscript{0}}(\cG)$ is a finite set and so is each of the quotient sets $\SubGx{a}/\simc$.  
\item If $\cG$ is locally strongly finite, then the set of all $\cG$-equivariant maps $\GHom{\cG/\cH}{\cG/\cK}$ is finite, for every $\cH, \cK \in \SubG$.
\end{enumerate}
\end{proposition}
\begin{proof}
Parts $(i)$ and $(ii)$ are straightforward. Applying Lemma \ref{lema:Triangular}(i), one deduces $(iii)$, as each of the isotropy groups is assumed to be finite. 
\end{proof}

Given a locally strongly finite groupoid $\cG$, let us fix a set of representatives $\rep(\SubG)$ and a set of representatives of the quotient set $\pi_{\Sscript{0}}(\cG)$, whose elements we call $a_{\Sscript{1}}, \cdots , a_{\Sscript{n}} \in\Go$.
According to Proposition \ref{prop:Rep}  (i), we can write 
\begin{equation}\label{Eq:SubG}
\rep(\SubG) \,\, =\,\, \biguplus_{i=1}^{n} \rep(\SubGx{a_i}),
\end{equation}
where each of the $\cG^{\Sscript{\lar{a_i}}}$ is a transitive groupoid (i.e., the connected component containing $a_i$). Furthermore, once fixed the choice of $\rep(\SubG)$,  we can consider the following  family of positive integers:
$$
\Sf{m}_{\Sscript{(\cH,\, \cK)}} \,=\,  \left\lvert{ \GHom{\cG/\cH}{\cG/\cK}   }\right\rvert, \qquad \forall\,  \cH, \cK \in \rep\big(\SubG\big),
$$
and   by Proposition \ref{prop:Hom}, we know that these entries are
$$
\Sf{m}_{\Sscript{\left({\mathcal{H}, \mathcal{K} }\right)}} =  \left\lvert{ \left({{\mathcal{G}/\mathcal{K} }}\right)^{\Sscript{\cH}} }\right\rvert, \quad \forall \, \cH, \cK \in \rep(\SubG).
$$
This, in conjunction with Lemma \ref{lema:Triangular}, shows that the  natural numbers $\{\Sf{m}_{\Sscript{(\cH,\, \cK)}} \}_{\Sscript{\cH,\, \cK}\, \in \,  \rep\big(\SubG\big)}$ satisfy the following conditions:
\begin{equation}\label{Eq:Mhk}
\Sf{m}_{\Sscript{(\cH,\, \cK)}} \,  \Sf{m}_{\Sscript{(\cK,\, \cH)}}\,=\, 0 ,\; \forall\, \cH \neq \cK \quad \text{ and } \quad \Sf{m}_{\Sscript{(\cH,\, \cH)}}\,\neq\, 0,\;  \forall \, \cH, 
\end{equation}
where $\cH = \cK$ in $\rep\big(\SubG\big)$ means that $\cH$  and $\cK$ are conjugated (or isomorphic as objects in the category $\SubG$ of Remark \ref{rem:XH}). 
The table that we want to construct in the sequel, which will be formed by those coefficients (where \(\mathcal{H}\) denotes the row position  and \(\mathcal{K}\) denotes the column one), is what we can call, in analogy with the classical case \cite[\S 180]{Burnside:1911}, \emph{the table of marks of the groupoid} $\cG$.

\begin{proposition}[The table of marks of a finite groupoid]\label{prop:tableofmarks}
Let $\cG$ be a locally strongly finite groupoid. Then the fixed set of representatives $\rep(\SubG)$, can be endowed with a total order $\preceq$ satisfying the following property: for every $\cH ,\cK \in \rep(\SubG) $, we have 
$$
 \cH \preceq \cK \, \Longrightarrow \,  \begin{cases}  \Sf{m}_{\Sscript{(\cH,\, \cK)}} = 0 & \text{ if } \, \cH \neq \cK \\ \Sf{m}_{\Sscript{(\cH,\, \cK)}} \neq 0   &   \text{ if }\,  \cH = \cK. \end{cases} 
$$

In particular, under this choice of ordering, the table (or matrix) of marks of $\cG$ has the following form:
$$
\Big(  \Sf{m}_{\Sscript{(\cH,\, \cK)}} \Big)_{\cH,\, \cK \, \in \, \rep(\SubG)} \, \, = \, \, \left( 
\begin{array}{c|c|c}
M_1 & \boldsymbol{0}  & \boldsymbol{0} \\ \hline  \boldsymbol{0} & \ddots & \boldsymbol{0} \\ \hline  \boldsymbol{0} & \boldsymbol{0} & M_n
\end{array}
\right),
$$
where $n$ is the number of connected components of $\cG$ and each of the matrix $M_i$, $i=1,\cdots, n$ is a lower triangular matrix with each diagonal entry different from zero. 
\end{proposition}
\begin{proof}
To construct this total order on the finite set $\rep(\SubG)$, one proceeds as follows. 
If the handled groupoid $\cG$ has only one object, then we are in the classical situation  of a finite group and the total ordering is exactly given by comaring the cardinality of representatives subgroups of this group modulo the conjugation relation. The details are expounded in \cite[pages 236 and 237]{Burnside:1911}, and the result is one of the matrices $M_i$'s.  Regarding the case when $\cG$ is a transitive groupoid, one can employ, for instance, Theorem~\ref{tEquivCatTransGrpd} to reduce this case to the  particular one of finite groups and proceed as in the classical case. 

As for the general case, one uses equation \eqref{Eq:SubG} to decomposes $\rep(\SubG)$ into a finite disjoint union of finite sets $\{\rep(\SubGx{a_i})\}_{\Sscript{i=1,\cdots, n}}$, where $n$ is the number of connected components of $\cG$. In this way one can extended the total ordering of each piece $\rep(\SubGx{a_i})$ to the whole set $\rep(\SubG)$, since each of the $\cG^{\Sscript{\lar{a_i}}}$'s is a transitive groupoid  (following, for example, the order $1<2<\cdots<n$ between the pieces). The resulting matrix (or the table of marks) of $\cG$ will be a diagonal block-matrix whose blocks correspond to the matrix of $\cG^{(a_i)}$ and, such that, outside of these blocks, only zeros will appear. Given two distinct elements  $\cH, \cK \in \rep(\SubG)$ with  $\Ho=\{a\}$ and $\Ko=\{b\}$  such that $a$ and $b$ are not connected, it is necessary, by Corollary \ref{lema:Triangular}(i), to have  $\Sf{m}_{\Sscript{(\cH,\, \cK)}}=\Sf{m}_{\Sscript{(\cK,\, \cH)}}=0$.  Thus, the whole matrix will also be upper triangular with non zero entries in the diagonal as stated. 
\end{proof}

\section{Burnside Theorem for groupoid-sets: General and finite cases}\label{sec:BM}
Before introducing the ghost function  for (finite) groupoids, an analogue of Burnside Theorem for right groupoid-sets   will be accomplished in this section. The classical  situation of groups is described as follows. Take two right $G$-sets $X$ and $Y$ and assume that their fixed point subsets under any subgroup are in bijection, that is, $X^{\Sscript{H}} \simeq Y^{\Sscript{H}}$, for any subgroup $H$ of $G$.  Under this assumption, in general $X$ and $Y$ are not  necessarily isomorphic as right $G$-sets.  The main objective of the Burnside Theorem (see  \cite[Theorem I, page 238]{Burnside:1911} or, for instance, \cite[Theorem 2.4.5]{Bouc:2010}) is to seek further conditions under which  $X$ and $Y$  become isomorphic as right $G$-sets. 
From a categorical point of view, one can assume, in the previous situation, a stronger hypothesis, namely,    that the functors $H \to X^{\Sscript{H}}$ and $H \to Y^{\Sscript{H}}$ are naturally isomorphic (see Remark \ref{rem:XH} for the definition of these functors). Nevertheless, this is equivalent to say that the functors $\{e\} \to X^{\Sscript{\{e\}}}$ and $\{e\} \to Y^{\Sscript{\{e\}}}$ are naturally isomorphic (here we're taking   the full subcategory of the category of subgroups of $G$, with only one object  $e$  the neutral element of $G$) which, as we will see below, is equivalent to say that $X$ and $Y$ are isomorphic as right $G$-sets.  In this direction, it is not clear, at least to us, whether the condition  $X^{\Sscript{H}} \simeq Y^{\Sscript{H}}$, for every subgroup $H$ of $G$, implies that the  functors $H \to X^{\Sscript{H}}$ and $H \to Y^{\Sscript{H}}$ are naturally isomorphic (it seems that, without passing through the classical Burnside's theorem, this is not known even for the finite case, that is, when $G$, $X$ and $Y$ are finite sets). 
All this suggests that,  in the context of groupoid-sets, one should treat separately the case when the fixed point subsets functors are naturally isomorphic.

\subsection{The general case: Two $\cG$-sets with natural bijection between fixed points subsets}\label{ssec:Gcase}
Let us  first explain what is the meaning of the  natural bijection, between the fixed points subsets, that was mentioned above.

\begin{definition}\label{def:natural}
Let $(X, \varsigma)$ and $(Y, \vartheta)$ be two right $\cG$-sets. We say that \emph{$(X, \varsigma)$ and $(Y, \vartheta)$ have naturally the same fixed points subsets}, provided there is a natural bijection $X^{\Sscript{\cH}} \simeq Y^{\Sscript{\cH}}$, for every subgroupoid $\cH$ of $\cG$ with only one object. This means  that we have a commutative diagram
\begin{equation}\label{Eq:natural}
\begin{gathered}
\xymatrix@C=25pt{  X^{\Sscript{\cH'}}   \ar@{->}_-{\simeq}[d] \ar@{->}^-{\Sscript{X^F}}[rr] &   & X^{\Sscript{\cH}}    \ar@{->}^-{\simeq}[d]   \\   Y^{\Sscript{\cH'}}  \ar@{->}^-{\Sscript{Y^F}}[rr]    &   &  Y^{\Sscript{\cH}}    }
\end{gathered}
\end{equation}
for any $\cG$-equivariant map $F: \cG/\cH \to \cG/\cH'$ between cosets of subgroupoids with only one object, where $X^{\Sscript{F}}$ and $Y^{\Sscript{F}}$  are the maps given as in the proof of Corollary \ref{coro:cuadro}.  
\end{definition}

\begin{remark}\label{rem:Def}
In the case of groups, if we assume that two right $G$-sets have naturally the same fixed points subsets as in Definition  \ref{def:natural}, then this, in particular, implies that $X^{\Sscript{\{e\}}} \simeq Y^{\Sscript{\{e\}}}$ in a natural way ($e$ is the neutral element of $G$). Thus, for any $g \in G$, the right translation map $x \mapsto xg$ from $G$ to $G$ gives arise to a $G$-equivariant map $F : G/\{e \} \to G/\{e\}$ which, by the commutativity of diagram \eqref{Eq:natural}, shows that $X$ and $Y$ are isomorphic as right $G$-sets. Thus, in the group context, two right $G$-sets are isomorphic if and only if they have naturally the same fixed points subsets. The case of groupoids is a bit more elaborate, as we will see in the sequel. 
\end{remark}

Using the previous definition we can show the following result.

\begin{proposition}\label{prop:2to1}
 Let $\cG$ be a groupoid and let's consider  two right $\cG$-sets $(X, \varsigma)$ and $(Y,\vartheta)$.  Then the following statements are equivalent.
 \begin{enumerate}[(i)]
\item  $(X, \varsigma)$ and $(Y,\vartheta)$ have naturally  the same fixed points subsets under the action of each one object subgroupoid   (Definition \ref{def:natural}); 
\item $(X, \varsigma)$ and $(Y,\vartheta)$ are isomorphic as $\cG$-sets.
\end{enumerate}
\end{proposition}
\begin{proof}
$(ii) \Rightarrow (i)$. It is clearly deduced from Proposition \ref{prop:Hom},  Corollary \ref{coro:cuadro} and Remark \ref{rem:XH}. 

$(i) \Rightarrow (ii)$.  Given such an $\cH$ we have, by Proposition \ref{prop:Hom}, the following a commutative diagram  
$$
\xymatrix{  \GHom{\cG/\cH}{X} \ar@{->}^-{\cong}[rr]  \ar@{->}_-{\phi_{\Sscript{\cG/\cH}}}[d] & &   X^{\Sscript{\cH}} \ar@{->}^-{\simeq}[d] \\ \GHom{\cG/\cH}{Y} \ar@{->}^-{\cong}[rr] & &   Y^{\Sscript{\cH}} . }
$$
Let us check that $\phi_{-}$ establishes a natural transformation (isomorphism indeed) over the class of right $\cG$-sets which are right cosets by one object subgroupoids. Thus, given another subgroupoid with only one object $\cH'$ together with a $\cG$-equivariant maps  $F: \cG/\cH \to \cG/\cH'$  
$$
\xymatrix@C=35pt{ & X^{\Sscript{\cH'}} \ar@{->}^-{\simeq}[ld]  \ar@{->}^-{\simeq}[dd] \ar@{->}^-{\Sscript{X^F}}[rrr] &   &  & X^{\Sscript{\cH}}  \ar@{->}^-{\simeq}[rd]  \ar@{->}^-{\simeq}[dd] &  \\   \GHom{\cG/\cH'}{X}   \ar@{->}_-{\phi_{\Sscript{\cG/\cH'}}}[dd] \ar@{->}^-{\Sscript{\GHom{F}{X}}}[rrrrr] & & &  & &   \GHom{\cG/\cH}{X} \ar@{->}^-{\phi_{\Sscript{\cG/\cH}}}[dd]  \\  & Y^{\Sscript{\cH'}}  \ar@{->}^-{\Sscript{Y^F}}[rrr]  \ar@{->}^-{\simeq}[ld]  &   &  & Y^{\Sscript{\cH}}  \ar@{->}^-{\simeq}[rd] &  \\  \GHom{\cG/\cH'}{Y} \ar@{->}^-{\Sscript{\GHom{F}{Y}}}[rrrrr] & &  &  & &  \GHom{\cG/\cH}{Y}  }
$$
we need to show that the front rectangle is commutative. However, this follows immediately from Corollary \ref{coro:cuadro}, since we already know by assumptions that the rear square commutes, and the desired natural isomorphism $\phi_{-}$ is derived.

Now, let us consider an arbitrary $\cG$-set $(Z, \zeta)$. We know from \cite[Corollary 3.11]{Kaoutit/Spinosa:2018} that 
$$ 
Z \, \cong \, \biguplus_{z\, \in \,   \rep_{\scriptscriptstyle{\mathcal{G}}}(Z) } \cG/ \Stab{z}{G}
$$
and, for each subgroupoid $\cH$ of $\cG$ with a single object,   we have a commutative diagram
$$
\xymatrix{ \GHom{Z}{X} \ar@{-->}^-{\Sscript{\phi_{Z}}}[rr]   \ar@{->}_-{\simeq}[d] & & \GHom{Z}{Y} \ar@{->}^-{\simeq}[d] \\  \underset{z \, \in \, \rep_{\Sscript{\cG}}(Z) }{\prod} \GHom{\cG/\Stab{z}{G}}{X} \ar@{->}_-{\cong}^-{ \Sscript{\underset{z \, \in \, \rep_{\Sscript{\cG}}(Z) }{\prod} \phi_{\cG/\Stab{z}{G}}}}[rr]  & &  \underset{z \, \in \, \rep_{\Sscript{\cG}}(Z) }{\prod} \GHom{\cG/\Stab{z}{G}}{Y} . }
$$
This leads to a natural isomorphism $\GHom{Z}{X} \simeq \GHom{Z}{Y}$ for each $\cG$-set $(Z, \zeta)$.
As a consequence we obtain $(X,\varsigma) \cong (Y,\vartheta)$ as right $\cG$-sets, as claimed. 
\end{proof}

\begin{remark}\label{rem:contraejemplo}
Combining Propositions \ref{prop:2to1} and \ref{prop:Hom}, we have that  two $\cG$-sets are isomorphic if and only if their fixed points sets are in a natural bijection, in the sense of Definition \ref{def:natural}. It could happens that two $\cG$-sets have bijective fixed points subsets but not in a natural way, that is, there is no choice of a family of bijections which turns the diagrams \eqref{Eq:natural} commutative (up to our knowledge, this is not even known for the case of groups). Since we do  have neither a counterexample nor  a complete proof for the fact that these diagrams are always commutative, once a bijection is given between the fixed points subsets, it is wise to consider the proof of the case when diagrams \eqref{Eq:natural} do not commute. 
Of course, in this case, the proof of Proposition \ref{prop:2to1} does not work and the converse of the previous equivalence fails. Finiteness conditions should be imposed, in order to provide the proof of the converse implication. This  seems to explain  the notable difficulty of the classical Burnside theory. 
\end{remark}

\subsection{The finite case: Two finite $\cG$-sets with bijective fixed points subsets}\label{ssec:Pcase}
Next, we will try to find sufficient conditions under which two finite $\cG$-sets, whose fixed points subsets have the same cardinality, should be isomorphic;  this will be the Burnside Theorem we are looking for. 

Given groupoid $\cG$, recall that $\SubG$ denotes  its set of  subgroupoids with only one object  and $\simc$ is the equivalence relation on this set given by conjugation.

\begin{lemma}\label{lema:IMp}
Let $(X, \varsigma)$ be a right $\cG$-set and let's consider $x, x' \in X$. Then if $x$ and $x'$ belong to the same orbit,  $\Stab{x}{G}$ and $\Stab{x'}{G}$ are conjugated subgroupoids of $\cG$. Furthermore, the canonical map $X \to \SubG$ sending $x \mapsto \Stab{x}{G} \leq \cG^{\Sscript{\varsigma(x)}}$, which factors through the quotient sets $X/\cG \to \SubG/\simc$, leads to a well defined map
$$
\wp_{\Sscript{X}}: \rep_{\Sscript{\cG}}(X) \longrightarrow \SubG/\simc, \quad \Big( x \longmapsto [\Stab{x}{G}] \Big).
$$
\end{lemma}
\begin{proof}
Straightforward.
\end{proof}

Using this lemma, one can construct the following map with values in the natural numbers: given a right $\cG$-set $(X, \varsigma)$ (with a countable  underlying set $X$), we define
\begin{equation}\label{Eq:awp}
\Sf{a}_{\Sscript{X}}: \SubG \longrightarrow \mathbb{N}, \qquad \Big(  \cH \longmapsto | \wp_{\Sscript{X}}^{-1}\big(  [\cH]\big)|\Big).
\end{equation}
Of course we get that  $\Sf{a}_{\Sscript{X}}(\cH)=0$ if no representative element $x$ in $\rep_{\Sscript{\cG}}(X)$ has its orbit $\Orbit{x}{G}$ isomorphic to the coset $\cG/\cH$, or equivalently, if its stabilizer $\Stab{x}{G}$ is not conjugated with $\cH$. 

For any set $I$ and $Z$ any right $\cG$-set  we denote by $Z^{(I)}$ the disjoint union of $I$ copies of $Z$, that is, the coproduct,  in the category of right $\cG$-sets, of $Z$ with itself $I$-times. If  $I$ has a finite cardinal,  say $n \in \mathbb{N}$,  then we denote  this coproduct by $n\,  Z$, with the convention $0\,  Z=(\emptyset, \emptyset)$.

We know (see for instance \cite[Corollary 3.11]{Kaoutit/Spinosa:2018}) that the category of right $\cG$-sets has a cogenerator object given by the right $\cG$-set  $ \underset{\Sscript{\cH \, \in \, \SubG}}{\Uplus} \, \cG/\cH$ (the disjoint union of all the cosets of the form  $\cG/\cH$ where $\cH \in \SubG$). 
So given a right $\cG$-set $(X,\varsigma)$ with a countable underlying set (or it set of representatives modulo the $\cG$-action is countable),  then   we have a monomorphism of right $\cG$-sets
$$
\xymatrix{ \jmath: X\,  \ar@{^{(}->}^-{}[rr] &&  \underset{\Sscript{\cH \, \in \, \SubG}}{\Uplus} \, (\cG/\cH)^{(I_{\Sscript{\cH,\, X}})},   }
$$ 
whose image can be written as follows. First, we have the following isomorphism of $\cG$-sets:
 $$
 \underset{\Sscript{\cH \, \in \, \SubG}}{\Uplus} \, (\cG/\cH)^{(I_{\Sscript{\cH,\, X}})} \, \, \cong \,  \, \underset{\Sscript{\cK \, \in \, \rep\big(\SubG\big)}}{\Uplus} \, \, 	\big(\cG/\cK\big)^{(J_{\Sscript{\cK,\, X}})},
 $$
 where the cardinality of each of the sets  $J_{\Sscript{\cK,\, X}}$'s is of  the form $|J_{\Sscript{\cK,\, X}} |= |I_{\Sscript{\cK,\, X}} |\,  |[\cK]|$, where $|[\cK]|$ is the cardinal of the equivalence class represented by $\cK$ in the quotient set $ \SubG/\simc$.

Given an element $\cK \in \rep\big( \SubG\big)$, define the following natural numbers: 
$$
\Sf{n}_ {\Sscript{\cK}}(X) \,=\,  \begin{cases}
|J_{\Sscript{\cK,\, X}} |, \quad \text{ if } \,  \, \jmath(X) \cap (\cG/\cK)^{(J_{\Sscript{\cK,\, X}})} \neq \emptyset \\ 0, \quad \text{otherwise}.
 \end{cases}
 $$

 \begin{lemma}\label{lema:AN}
 Keep the above notations. Then, for every element $\cK \in \rep\big( \SubG\big)$,  we have  
 $$
 \Sf{a}_{\Sscript{X}}(\cK) \, \leq \, \Sf{n}_ {\Sscript{\cK}}(X)\quad \text{ and } \quad   \Big( \Sf{a}_{\Sscript{X}}(\cK) =0 \, \Leftrightarrow \, \Sf{n}_ {\Sscript{\cK}}(X)=0\Big).
 $$
 Furthermore, we have isomorphisms of right $\cG$-sets:
 \begin{equation}\label{Eq:ISO}
 X \, \cong \,  \jmath(X)\, \cong \, \underset{{\cK \, \in \, \rep\big(\SubG\big)}}{\Uplus}  \Sf{a}_X(\cK) \, \cG/\cK,
 \end{equation}
 where the map $\Sf{a}_{\Sscript{X}}$ is the one of equation \eqref{Eq:awp}. 
 \end{lemma}
 \begin{proof}
 It is immediate.
 \end{proof}

Observe that if the underlying set $X$ of $(X, \varsigma)$ is finite, then there are finitely many elements $\cK \in \rep\big( \SubG\big)$ with the property $\Sf{a}_{\Sscript{X}}(\cK) \neq 0$. Thus, in the finite $\cG$-sets case, the support sets $\{\cK \in \rep\big(\SubG\big)|\, \Sf{a}_{\Sscript{X}}(\cK) \neq 0\}$  have to be finite as well.

The subsequent theorem is the main result of this section. 

\begin{theorem}[Burnside Theorem]\label{tThmBurnside}
Let $\cG$ be a locally strongly finite groupoid (Definition \ref{def:finiteG}). Consider  two  finite right \(\mathcal{G}\)-sets \(\left({X, \varsigma}\right)\) and \(\left({Y, \vartheta}\right)\). Then the following statements are equivalent.
\begin{enumerate}
\item The right \(\Cat{G}\)-sets \(\left({X, \varsigma}\right)\) and \(\left({Y, \vartheta}\right)\) are isomorphic.
\item For each subgroupoid \(\mathcal{H}\) of \(\mathcal{G}\) with a single object, we have that 	
\[
\left\lvert{X^{\Sscript{\cH} }}\right\rvert= \left\lvert{Y^{\Sscript{\cH}}}\right\rvert.
\]
\end{enumerate}
In particular, this applies to any strongly finite groupoid. 
\end{theorem}
\begin{proof}
\((1) \Rightarrow (2)\). Follows from Propositions \ref{prop:Hom} or \ref{prop:2to1}.

\((2) \Rightarrow (1)\). Using the isomorphisms given in equation \eqref{Eq:ISO}, we know that 
\[ 
X \, \cong \, \biguplus_{ \mathcal{K}  \, \in \,  \rep\big(\SubG\big)  } \Sf{a}_{\Sscript{X}}(\cK) \,  \cG/\cK
\qquad \text{and} \qquad
Y \, \cong \, \biguplus_{ \mathcal{K}  \, \in \,  \rep\big(\SubG\big)  } \Sf{a}_{\Sscript{Y}}(\cK) \, \cG/\cK .
\]
By hypothesis it is assumed that   \(\left\lvert{X^{\Sscript{\cH}}}\right\rvert= \left\lvert{Y^{\Sscript{\cH}}}\right\rvert\) for each subgroupoid \(\mathcal{H}\) of \(\mathcal{G}\) with a single object. Applying the bijections of equation \eqref{Eq:XHI} to the previous isomorphisms, we get the following  equalities
\begin{multline*}
\sum_{ \mathcal{K} \, \in  \, \rep \left({\mathcal{S}_{\mathcal{G} } }\right) } \Sf{a}_{\Sscript{X}}(\cK)  \left\lvert{ \left({{\mathcal{G}/\cK }}\right)^{\Sscript{\cH}} }\right\rvert = \left\lvert{ \biguplus_{\mathcal{K}  \, \in  \, \rep\left({\mathcal{S}_{\mathcal{G} } }\right) }  \Sf{a}_{\Sscript{X}}(\cK)  \left({{ \mathcal{G}/\mathcal{K} } }\right)^{\Sscript{\cH} } }\right\rvert  
= \left\lvert{ \left({\biguplus_{\mathcal{K}  \, \in  \, \rep\left({\mathcal{S}_{\mathcal{G} } }\right) } \Sf{a}_{\Sscript{X}}(\cK) \,  {\mathcal{G}/\mathcal{K} }  }\right)^{\Sscript{\cH} } }\right\rvert 
= \left\lvert{X^{\Sscript{\cH} } }\right\rvert  \\
= \left\lvert{Y^{\Sscript{\cH}} }\right\rvert  
=   \sum_{ \mathcal{K}  \, \in \,  \rep \left({\mathcal{S}_{\mathcal{G} } }\right) }  \Sf{a}_{\Sscript{Y}}(\cK)   \left\lvert{ \left({{\mathcal{G}/\cK }}\right)^{\Sscript{\cH}} }\right\rvert, 
\end{multline*}
for every subgroupoid $\cH \in \SubG$. 
Therefore, for each \(\mathcal{H} \in \rep\left({\mathcal{S}_{\mathcal{G}}}\right)\),  we have the equality
\begin{equation}\label{eBurnGrpdSys}
\sum_{ \mathcal{K}  \, \in \,  \rep \left({\mathcal{S}_{\mathcal{G} } }\right) } \left(  \Sf{a}_{\Sscript{X}}(\cK)- \Sf{a}_{\Sscript{Y}}(\cK) \right)   \left\lvert{ \left({{\mathcal{G}/\mathcal{K} }}\right)^{\Sscript{\cH}} }\right\rvert =0.
\end{equation} 
Now by Proposition \ref{prop:tableofmarks}  the entries  $\{\Sf{m}_{\Sscript{(\cH,\, \cK)}}\}$ form an  upper triangular square matrix, which is non-singular.  It follows  that, in the system \eqref{eBurnGrpdSys}, we  have $\Sf{a}_{\Sscript{X}}(\cH)= \Sf{a}_{\Sscript{Y}}(\cH)$ for each $\cH \in \rep(\SubG)$. Therefore $(X, \varsigma)$ and $(Y,\vartheta)$ are isomorphic as right $\cG$-sets. 
Lastly, the particular claim is clear and this finishes the proof.
\end{proof}

The Burnside Theorem implies a sort of cancellative property, with respect to the internal  operation $\Uplus$.

\begin{corollary}
\label{cCancPropRXGSets}
Given a locally strongly finite groupoid $\cG$,  let be \(\left({X, \varsigma}\right)\), \(\left({Y, \vartheta}\right)\) and \(\left({Z, \zeta}\right)\) finite right \(\Cat{G}\)-sets such that there is an isomorphism of right \(\Cat{G}\)-sets of the form 
\[ \left({X, \varsigma}\right) \Uplus \left({Z, \zeta}\right) \cong \left({Y, \vartheta}\right) \Uplus \left({Z, \zeta}\right).
\]
Then we have an isomorphism \(\left({X, \varsigma}\right) \cong \left({Y, \vartheta}\right)\) of $\cG$-sets.
\end{corollary}
\begin{proof}
Given $\cH \in \SubG$, using the bijection of equation \eqref{Eq:XHI}, we obtain
\[ 
\left\lvert{\left({X, \varsigma}\right)^{\Sscript{\cH} } }\right\rvert + \left\lvert{\left({Z, \zeta}\right)^{\Sscript{\cH}}}\right\rvert = \left\lvert{\left({Y, \vartheta}\right)^{\Sscript{\cH}}}\right\rvert + \left\lvert{\left({Z, \zeta}\right)^{\Sscript{\cH}}}\right\rvert 
\]
therefore  \(\left\lvert{\left({X, \varsigma}\right)^{\Sscript{\cH} } }\right\rvert = \left\lvert{\left({Y, \vartheta}\right)^{\Sscript{\cH}}}\right\rvert\).
As a consequence, thanks to Theorem \ref{tThmBurnside}, we get \(\left({X, \varsigma}\right) \cong \left({Y, \vartheta}\right)\) as $\cG$-sets.
\end{proof}

\section{Burnside functor for groupoids: coproducts and products}\label{sec:}

In this section we introduce the Burnside ring attached to a groupoid with finitely many objects, whose construction is based on  the skeleton of the category of the right $\cG$-sets with underlying finite sets.  For the convenience of an inexperienced audience we recall in the Appendices \ref{ssec:Rig} and \ref{ssec:GrothF}, with very elementary arguments,  the general notion of Grothendieck construction as well as that of the category of rigs. Both are crucial in performing the construction of the Burnside ring functor.  
The compatibility of this functor with coproducts and product is needed in order to establish the main result of this section, which asserts that the Burnside ring of a given (finite) groupoid is the product  of the  Burnside rings of its isotropy groups, where the product is taken over  the set of the connected components (see Theorem \ref{tBurnTransCompProd}).

We assume, in this section,   that all handled groupoids have a finite set of objects. This condition is in fact needed to have a unit for the Burnside ring we are planing to introduce, since we will make use of the skeletally small category of  finite groupoid-sets to perform this construction. We also assume that functors between groupoid-sets preserve objects with finite underlying sets, and transform an empty groupoid-set to an empty one, as the induction  functors do. 
Given a groupoid $\cG$, we denote by $\frGsets$ the full subcategory of right $\cG$-sets with finite underlying sets.

\subsection{Burnside rig functor and coproducts.}\label{ssec:Brig}

Given a groupoid $\cG$, let \(\left({X, \varsigma}\right)\) be a finite right \(\Cat{G}\)-set, that is, an object in $\frGsets$, and denote by \(\left[{\left({X, \varsigma}\right)}\right]\) its equivalence class modulo the isomorphism relation.
Consider $\LG{G}$,   the \emph{quotient set of all finite right \(\Cat{G}\)-sets modulo  the isomorphism of right \(\Cat{G}\)-sets equivalence relation}. This means that elements of $\LG{G}$ are classes $[(X, \varsigma)]$ represented by $\cG$-sets $(X,\varsigma)$ with finite underlying set $X$. We  endow the set $\lL(\cG)$  with an addition and a multiplication operations: for every \(\left({X, \varsigma}\right), \left({Y, \vartheta}\right) \in \rset{G}\), we define
\[ \left[{\left({X, \varsigma}\right) }\right] + \left[{\left({Y, \vartheta}\right) }\right] := \left[{\left({X , \varsigma}\right) \Uplus \left({Y, \vartheta}\right) }\right] 
=  \left[{\left({X \Uplus Y, \varsigma \Uplus \vartheta}\right) }\right] 
\]
and
\[ \left[{\left({X, \varsigma}\right) }\right] \cdot \left[{\left({Y, \vartheta}\right) }\right] := \left[{\left({X, \varsigma}\right) \fpro{\Cat{G}_{ \scriptscriptstyle{0} } }  \left({Y, \vartheta}\right) }\right]
= \left[{\left({X  \, {{}_{ \scriptscriptstyle{\varsigma} } {\times}_{ \scriptscriptstyle{\vartheta} }\,}  Y, \varsigma \vartheta}\right) }\right],
\]
(see subsection \ref{ssec:SMC} for the notations).
It is immediate to check that these operations are well defined and that, in this way, \(\lL\left({\Cat{G}}\right)\) becomes a rig with additive neutral element \(\left[{\left({\emptyset, \emptyset}\right)}\right]\) and multiplicative neutral element \(\left[{\left({\Cat{G}_{ \scriptscriptstyle{0} }, \Id_{\Cat{G}_{ \scriptscriptstyle{0} } } }\right) }\right] \) (see Definition \ref{def:rig} for more details). 

The rig construction is a functorial one, as one can prove using general monoidal category theory.
We give, in our case, an elementary proof.

\begin{lemma}
\label{lBurnRigFu}
Given two groupoids \(\Cat{G}\) and \(\Cat{H}\), let \(F \colon \Rset{\cG} \longrightarrow \Rset{\cH}\) be a strong  monoidal  functor with respect to both monoidal structures: \(\Uplus\) and the fibered product.
Let us define
$$
{\Sf{h}}:  \LG{G}  \longrightarrow \LG{H}, \qquad \Big(  \left[{\left({X, \varsigma}\right) }\right] \longmapsto \left[{F\left({X, \varsigma }\right) }\right]  \Big).
$$
Then \(\Sf{h}\) is a homomorphism of rigs.
\end{lemma}
\begin{proof}
Clearly $\Sf{h}$ is a well defined map, since any functor preserves isomorphisms. 
Now, for every \(\left({X, \varsigma}\right), \left({Y, \vartheta}\right) \in \Rset{G}\) we have the following  isomorphisms of right \(\Cat{H}\)-sets
\[ 
F \left({ X \Uplus Y, \varsigma \Uplus \vartheta}\right) = F\left({\left({X, \varsigma}\right) \Uplus \left({Y, \vartheta}\right) }\right) \,\cong \, F\Big(\left(X, \varsigma\right) \Big) \Uplus F\Big(\left(Y, \vartheta\right) \Big)  \, \cong\,  F\left({X, \varsigma}\right) + F\left({Y, \vartheta}\right), 
\]
and 
\[ 
F \Big({ X  \, {{}_{ \scriptscriptstyle{\varsigma} } {\times}_{ \scriptscriptstyle{\vartheta} }\,}  Y, \varsigma \vartheta}\Big) = F\left({\left({X, \varsigma}\right) \fpro{\Cat{G}_{ \scriptscriptstyle{0} } }  \left({Y, \vartheta}\right) }\right)  \, \cong \,   F\left({X, \varsigma}\right) \underset{\Go}{\times} F\left({Y, \vartheta}\right)
\, \cong \,  F\left({X, \varsigma}\right) \cdot F\left({Y, \vartheta}\right). 
\]
Passing to the isomorphism classes and applying $\Sf{h}$, leads to the the equalities
\[ 
\begin{aligned}
&\Sf{h} \left({\left[{\left({X, \varsigma}\right) }\right] + \left[{\left({Y, \vartheta}\right) }\right] }\right) = \Sf{h} \left({\left[{\left({X \Uplus Y, \varsigma \Uplus \vartheta}\right) }\right] }\right) 
= \left[{F\left({X \Uplus Y, \varsigma \Uplus \vartheta}\right) }\right] \\
&= \left[{F \left({X, \varsigma}\right) + F \left({Y, \vartheta}\right) }\right] 
= \left[{F\left({X, \varsigma}\right) }\right] + \left[{F\left({Y, \vartheta}\right) }\right] 
= \Sf{h}\left({\left[{\left({X, \varsigma}\right) }\right] }\right) + \Sf{h}\left({\left[{\left({Y, \vartheta}\right) }\right] }\right)     
\end{aligned}
\]
and
\[ \begin{aligned}
& \Sf{h} \left({\left[{\left({X, \varsigma}\right) }\right] \cdot \left[{\left({Y, \vartheta}\right) }\right] }\right) = \Sf{h} \left({\left[{\left({X  \, {{}_{ \scriptscriptstyle{\varsigma} } {\times}_{ \scriptscriptstyle{\vartheta} }\,}  Y, \varsigma \vartheta}\right) }\right] }\right) 
= \left[{F\left({X  \, {{}_{ \scriptscriptstyle{\varsigma} } {\times}_{ \scriptscriptstyle{\vartheta} }\,}  Y, \varsigma \vartheta}\right) }\right] \\
&= \left[{F \left({X, \varsigma}\right) \cdot F \left({Y, \vartheta}\right) }\right] 
= \left[{F\left({X, \varsigma}\right) }\right] \cdot \left[{F\left({Y, \vartheta}\right) }\right] 
= \Sf{h} \left({\left[{\left({X, \varsigma}\right) }\right] }\right) \cdot \Sf{h} \left({\left[{\left({Y, \vartheta}\right) }\right] }\right)     .
\end{aligned}
\]
On the other hand,  we have the isomorphisms of right \(\Cat{H}\)-sets \(F\left({\emptyset, \emptyset}\right) \cong \left({\emptyset, \emptyset}\right)\) and \(F\left({\Cat{G}_{ \scriptscriptstyle{0} }, \Id_{\Cat{G}_{ \scriptscriptstyle{0} } }  }\right) \cong \left({\Cat{H}_{ \scriptscriptstyle{0} }, \Id_{\Cat{H}_{ \scriptscriptstyle{0} } } }\right)\). We then obtain
\[ \Sf{h} \left({\left[{\left({\emptyset, \emptyset}\right) }\right] }\right) = \left[{F\left({\emptyset, \emptyset }\right) }\right] 
= \left[{\left({\emptyset, \emptyset}\right) }\right] \quad 
\text{ and } \quad 
 \Sf{h} \left({\left[{\left({\Cat{G}_{ \scriptscriptstyle{0} }, \Id_{\Cat{G}_{ \scriptscriptstyle{0} } }  }\right) }\right] }\right) = \left[{F\left({\Cat{G}_{ \scriptscriptstyle{0} }, \Id_{\Cat{G}_{ \scriptscriptstyle{0} } } }\right) }\right] 
= \left[{\left({\Cat{H}_{ \scriptscriptstyle{0} }, \Id_{\Cat{H}_{ \scriptscriptstyle{0} } } }\right) }\right].
\]
As a consequence we have proved that \(\Sf{h}\) is a homomorphism of rigs as desired.
\end{proof}

Now, let  \(\phi \colon \Cat{H} \longrightarrow \Cat{G}\) be a homomorphism of groupoids. By Proposition~\(\ref{pMorGrpdIndFuRSets}\) we can consider the induced functor \(\phi^\ast \colon \Rset{\cG} \longrightarrow \Rset{\cH}\) and, thanks to Lemma~\(\ref{lBurnRigFu}\), from this functor we obtain a homomorphism of rigs from \(\lL\left({\Cat{G}}\right)\) to \(\lL\left(\Cat{H}\right)\), induced by \(\phi^\ast\), which we denote by \(\lL(\phi)\).
More precisely, we have 
\[ \begin{aligned}{\lL(\phi)}  \colon & {\lL\left({\Cat{G}}\right)} \longrightarrow { \lL\left({\Cat{H} }\right) } \\ & {\left[{\left({X,\varsigma}\right) }\right] }  \longrightarrow {\left[{\phi^\ast\left({X, \varsigma}\right) }\right]  .}\end{aligned} 
\]

\begin{propositiondef}\label{propdef:}
The correspondence  \(\lL\) defines a contravariant functor from the category of groupoids \(\Grou\) to the category of rigs \(\mathbf{Rig}\) which we call, inspired by \cite[page~\(381\)]{SchaNegEulCharDim}, the \emph{Burnside rig functor}.
\end{propositiondef}
\begin{proof}
Let \(\psi \colon \Cat{K} \longrightarrow \Cat{H}\) and \(\phi \colon \Cat{H} \longrightarrow \Cat{G}\) be morphisms of groupoid.
Using Proposition \ref{pIndFuIsomCompUn}, for each \(\left[{\left({X, \varsigma}\right) }\right] \in \lL\left({\Cat{G}}\right)\), we obtain
\[ \begin{aligned}
\lL\left({\psi}\right)L\left({\phi}\right)\Big( \left[{\left({X, \varsigma}\right) }\right] \Big) &= \lL\left({\psi}\right) \Big( \left[{\left({\phi^\ast\left({X,\varsigma}\right) }\right) }\right] 
= \left[{\psi^\ast \phi^\ast\left({X, \varsigma}\right) }\right] 
= \left[{\left({\phi \psi}\right)^\ast\left({X, \varsigma}\right) }\right]  \\
&= \lL\left({ \phi \psi }\right) \Big( \left[{\left({X, \varsigma}\right) }\right] \Big).
\end{aligned}
\]
Thus the following diagram is commutative
\[ \xymatrix{
\lL\left({\Cat{G} }\right)  \ar[d]_{\lL\left({\phi}\right) }  \ar[rr]^{\lL\left({ \phi \psi }\right) } &  & \lL\left({\Cat{K} }\right)   \\
\lL\left({\Cat{H} }\right).   \ar[rru]_{\lL\left({\psi}\right)}   && }
\]
Moreover, for each groupoid \(\Cat{G}\) we calculate,  thanks again  to Proposition \ref{pIndFuIsomCompUn},
\[ 
\lL\left({\Cat{G} }\right) \Big( \left[{\left({X, \varsigma}\right) }\right] \Big) = \left[{\left({\Id_{\Cat{G} }}\right)^\ast\left({X, \varsigma}\right) }\right] 
= \left[{\left({X, \varsigma}\right)}\right] 
= \Id_{\lL\left({\Cat{G} }\right) } \Big( \left[{\left({X, \varsigma}\right) }\right] \Big)
\]
thus \(\lL\left({\Cat{G} }\right) = \Id_{\lL\left({\Cat{G} }\right) }\). This shows that $\lL$ is a well defined functor as desired. 
\end{proof}

We finish this subsection by discussing the compatibility of the Burnside rig functor with coproduct.

\begin{proposition}
\label{pBuRigFunFinCopro}
The Burnside rig functor \(\lL\) sends coproduct to product.
In particular, given a family of groupoids  \(\left(\Cat{G}_j \right)_{j \, \in \,  I}\), let \(\left(i_j \colon \Cat{G}_j  \longrightarrow  \Cat{G}\right)_{j \, \in \,  I }\) be their coproduct in \(\Grou\).
Then
\[ \left(\lL\left({i_j}\right) \colon \lL\left({\Cat{G}}\right)  \longrightarrow \lL\left({\Cat{G}_j}\right) \right)_{j \in I} 
\]
is the product of the family \(\left(\lL\left({ \Cat{G}_j}\right)\right)_{j \, \in \,  I}\) in the category \(\mathbf{Rig}\).
\end{proposition}
\begin{proof}
Let \(\left(f_j \colon A \longrightarrow \lL\left({\Cat{G}_j}\right)\right)_{j \,  \in \,  I}\) be a  family of homomorphisms of rigs.
We have to prove that there is a unique homomorphism \(f \colon A \longrightarrow \lL\left({\Cat{G}}\right)\) of rigs such that the following diagram commutes for every \(j \in I\):
\begin{equation}\label{dBurnRigCopr}
\begin{gathered}
\xymatrix{  A \ar@{->}^-{f}[rr]   \ar@{->}_-{f_j}[drr]  && \lL(\cG)  \ar@{->}^-{\lL(i_j)}[d] \\      & & \lL(\cG_j). }
\end{gathered}
\end{equation}
Let \(a \in A\): for every \(j \in I\) there is \(\left({X_j, \varsigma_j}\right) \in \rset\cG_{j}\) such that \(f_j(a)=\left[{\left({X_j, \varsigma_j}\right) }\right] \in \lL\left({\Cat{G}_j}\right)\). Henceforth,  we define, thanks to the proof of Proposition \ref{pSetsCoprGrpd},
\[ 
f(a)= \left[{\biguplus_{j \, \in \, I} \widehat{\left({X_j, \varsigma_j}\right) } }\right] ,
\]
obtaining, in this way,  a function \(f \colon A \longrightarrow \lL\left({\Cat{G}}\right)\).
 Furthermore, for every \(l \in I\) we have
\begin{equation}
\label{eBurnRigCopr}
\lL\left({i_l}\right)f(a)= \left[{ \left({i_l}\right)^\ast \left( \biguplus_{j \, \in \, I} \widehat{\left({X_j, \varsigma_j}\right) } \right) }\right] 
=\left[{  \biguplus_{j \, \in \, I}  \left({i_l}\right)^\ast \left( \widehat{\left({X_j, \varsigma_j}\right) } \right) }\right]  .
\end{equation} 
For every \(l, j \in J\) such that \(l \neq j\) we clearly have
\[ \left({i_l}\right)^\ast \left({ \widehat{\left({X_j, \varsigma_j}\right)  } }\right) = \left({X_j  \, {{}_{ \scriptscriptstyle{\widehat{\varsigma_j} } } {\times}_{ \scriptscriptstyle{\left({i_l}\right)_{ \scriptscriptstyle{0} } } }\,} \left({\Cat{G}_l}\right)_{ \scriptscriptstyle{0} } , \pr_{ \scriptscriptstyle{2} }  }\right)  
= \left({\emptyset , \emptyset}\right)
\]
because \(\left({\Cat{G}_j}\right)_{ \scriptscriptstyle{0} } \cap \left({\Cat{G}_l}\right)_{ \scriptscriptstyle{0} } = \emptyset\).
Instead, for every \(j \in I\) we have the following isomorphism of right \(\Cat{G}_j\)-sets:
\[ 
\begin{aligned}{\pr_{ \scriptscriptstyle{1} } }  \colon & {\left({i_j}\right)^\ast \left({ \widehat{\left({X_j, \varsigma_j}\right)  } }\right) =\left({X_j  \, {{}_{ \scriptscriptstyle{\widehat{\varsigma_j} } } {\times}_{ \scriptscriptstyle{\left({i_j}\right)_{ \scriptscriptstyle{0} } } }\,} \left({\Cat{G}_j}\right)_{ \scriptscriptstyle{0} } , \pr_{ \scriptscriptstyle{2} } }\right) } \longrightarrow { \left({X_j, \varsigma_j}\right) } \\ & {(x,c)}  \longrightarrow {x.}
\end{aligned} 
\]
Continuing from formula~\(\eqref{eBurnRigCopr}\) we obtain that
\[ 
\lL\left({i_l}\right)f(a) = \left[{\left({X_l, \varsigma_l}\right) }\right] 
= f_l(a), \; \text{ for every }\; l \in I, \text{ and } a \in A.
\]
This shows that the diagram~\(\eqref{dBurnRigCopr}\) commutes.

The fact that $f$ is a morphism of rigs, that is, $f$ is compatible with addition and the multiplication, is  proved by direct computations using in part Lemma \ref{lRightGSetCoprod}.   Lastly, if  \(\gamma \colon A \longrightarrow \lL\left({\Cat{G}}\right)\) is another homomorphism of rigs which turns commutative diagrams \eqref{dBurnRigCopr}, then for a 
given \(a \in A\), let be \(\left({X, \varsigma }\right) \in \rset{\cG}\) such that \(\gamma(a) = \left[{\left({X, \varsigma}\right)}\right]\).
Setting, for every \(j \in I\), \(X_j = \varsigma_j^{-1}\left({\Cat{G}_{l\, 0}}\right)\) and \(\varsigma_l = \left.{\varsigma}\right|_{\varsigma^{-1}\left({\Cat{G}_l}\right)} \colon X_l \longrightarrow \left(\Cat{G}_l\right)_{ \scriptscriptstyle{0} }\), and restricting appropriately the action, we have
\[ 
\gamma(a)= \left[{\biguplus_{j \, \in \, I} \widehat{\left({X_j, \varsigma_j}\right) } }\right]. 
\]
Thus, using properties of \(\left({i_j}\right)^\ast\) already proved in this proof, we get
\[ 
f_i(a)= \lL\left({i_j}\right) \left({\gamma(a)}\right) 
= \left[{\left({i_j}\right)^\ast \left({X, \varsigma}\right) }\right]
= \left[{\left({X_j, \varsigma_j}\right) }\right] 
\]
Therefore, by definition of \(f\), we obtain that  $f(a) =\gamma(a)$, for every $a \in A$, and this shows that $f$ is unique and finishes the proof. 
\end{proof}

\subsection{Classical Burnside ring functor and product decomposition.}\label{ssec:Bring}
Now we introduce, using the Burnside rig functor,  the classical Burnside ring functor, and give our main result dealing with the decomposition of the Burnside ring of a given groupoid as a product of the classical Burnside rings of the isotropy groups, which  somehow justifies the terminology.
\begin{definition}\label{def:BF}
We define the \emph{classical Burnside ring functor} \(\mathscr{B}\) as the composition of the Burnside rig functor \(\lL\) with the Grothendieck functor \(\gG\), that is, \(\bB= \gG \lL\).
\end{definition}

The situation is explained in the following commutative diagrams of functors:
\[ \xymatrix{
\Grou \ar[d]_{\lL}  \ar[rr]^{\bB} &  & \mathbf{CRing}  \\
\mathbf{Rig}  \ar[rru]_{\gG} &&    }
\]
where $\bf{CRing}$ denotes the category of commutative rings.
Of course, since \(\lL\) is contravariant functor and \(\gG\) is a covariant one, \(\bB\) is a contravariant functor.

\begin{theorem}
\label{tEquivCatsIsomBurnRings}
Let be \(\mathcal{G}\) and \(\mathcal{A}\) be groupoids such that there is a symmetric strong monoidal  equivalence of categories
\[ 
\Rset{\cG} \simeq \Rset{\cA}
\]
with respect to both \(\Uplus\) and the fibered product.
Then there is an isomorphism of commutative rings
\[ 
\bB\left({\mathcal{G} }\right) \cong \bB\left({\mathcal{A} }\right).
\]
In particular two weakly equivalent groupoids have isomorphic classical Burnside rings.  
\end{theorem}
\begin{proof}
Let us denote by
\[ 
F \colon \Rset{\cG} \longrightarrow \Rset{\cA} 
\qquad \text{and} \qquad
G \colon \Rset{\cA} \longrightarrow \Rset{\cG}
\]
the strong monoidal functors which give the stated equivalence. 
Thanks to Proposition~\(\ref{pGrotFuIsoToIso}\), it is enough to prove that there is an isomorphism of rigs \(\lL\left({\Cat{G}}\right) \cong \lL\left({\Cat{A}}\right)\). By  Lemma~\(\ref{lBurnRigFu}\),  we have the following 
\[ 
\begin{aligned}
{\Sf{f}}  \colon & {\lL\left({\mathcal{G} }\right) } \longrightarrow {\lL\left({\mathcal{A} }\right) } \\ & {\left[{\left({X, \varsigma}\right)}\right] }  \longrightarrow {\left[{F(\left({X, \varsigma}\right))}\right] }
\end{aligned}   \qquad  \qquad \begin{aligned}
{\Sf{g}}  \colon & {\lL\left({\mathcal{A} }\right) } \longrightarrow {\lL\left({\mathcal{G} }\right) } \\ & {\left[{\left({Y, \vartheta}\right)}\right] }  \longrightarrow {\left[{G(\left({Y, \vartheta}\right))}\right] }
\end{aligned} 
\]
well defined homomorphism of rigs. It is left to reader to check that $\Sf{f}$ and $\Sf{g}$ are mutually inverse. 
\end{proof}

\begin{remark}
Observe that, for every finite right \(\Cat{G}\)-sets \(\left({X, \varsigma}\right)\), \(\left({Y, \vartheta}\right)\), \(\left({Z, \zeta}\right)\) and \(\left({W, \omega}\right)\), we have that
\[ 
\Big[ \left[{\left({X, \varsigma}\right) }\right], \left[{\left({Y, \vartheta}\right) }\right]  \Big] = \Big[ \left[{\left({Z, \zeta}\right) }\right], \left[{\left({W, \omega}\right) }\right]  \Big] 
\]
if and only if there is a finite right \(\Cat{G}\)-set \((U, u)\) such that
\[  \left[{\left({X, \varsigma}\right) }\right] + \left[{\left({W, \omega}\right) }\right] + \left[{\left({U, u}\right) }\right] = \left[{\left({Z, \zeta}\right) }\right] + \left[{\left({Y, \vartheta}\right) }\right] + \left[{\left({U, u}\right) }\right]
\]
as elements in $\bB(\cG)$, where the notation is the one adopted in  Appendix \ref{ssec:GrothF}. 
If the groupoid \(\Cat{G}\) is strongly finite, thanks to Corollary~\(\ref{cCancPropRXGSets}\), this is equivalent to say that \(\left[{\left({X, \varsigma}\right) }\right] + \left[{\left({W, \omega}\right) }\right]  = \left[{\left({Z, \zeta}\right) }\right] + \left[{\left({Y, \vartheta}\right) }\right]\).
\end{remark}

\begin{corollary}\label{cBurnFuCoprToProd}
The Burnside ring functor \(\bB\) sends coproduct to product.
In particular, given a family of groupoids  \(\left(\Cat{G}_j \right)_{j \, \in \, I}\), let \(\left(i_j \colon \Cat{G}_j  \longrightarrow  \Cat{G}\right)_{j \, \in \, I }\) be their coproduct in \(\Grou\).
Then
\[ \left( \bB\left({i_j}\right) \colon \bB\left({\Cat{G}}\right)  \longrightarrow \bB\left({\Cat{G}_j}\right) \right)_{j \, \in \,  I} 
\]
is the product of the family \(\left(\bB\left({ \Cat{G}_j}\right)\right)_{j \, \in \, I}\) in \(\mathbf{CRing}\).
\end{corollary}
\begin{proof}
Immediate from Proposition~\(\ref{pBuRigFunFinCopro}\) and Proposition~\(\ref{pGrothFuPresProd.}\).
\end{proof}

Our main result of  this section is the following:
\begin{theorem}\label{tBurnTransCompProd}
Given a groupoid \(\Cat{G}\), fix a set of representative objects $\rep(\Go)$ representing  the set of connected components $\pi_{\Sscript{0}}(\cG)$. For each \(a \, \in \, \rep(\Go)  \), let  \(\Cat{G}^{\lar{a}}\) be the connected component of \(\Cat{G}\) containing $a$, which we consider as a groupoid.
Then we have the following isomorphism of rings:
\[ 
\bB\left({\mathcal{G}}\right) \cong \prod_{a \, \in \, \rep(\Go) } \bB\left({ \Cat{G}^{\lar{a}}  }\right).
\]
\end{theorem}
\begin{proof}
Immediate from Corollary~\(\ref{cBurnFuCoprToProd}\), since we already know that $\{ \cG^{\lar{a}} \to \cG \}_{a\, \in \, \rep(\Go)}$ is a coproduct in the category of groupoids.
\end{proof}

Each connected component of a given groupoid is clearly a transitive groupoid, and the Burnside ring of transitive groupoids is given as follows. First notice that any group when considered as groupoid with only one object, its classical Burnside ring, as introduced in \cite{Solomon:1967} see also \cite{TomDieck:1979}, coincides with the ring hereby introduced (see Proposition \ref{pGrotFuIsoToIso} for the proof).

\begin{proposition}\label{pTransGrpdBurnRing}
Given a transitive groupoid \(\mathcal{G}\), let \(a \in \mathcal{G}_{ \scriptscriptstyle{0} }\) and \(G= \mathcal{G}^{ \scriptscriptstyle{a} }\).
Let \(\mathcal{A}\) be the subgroupoid of \(\mathcal{G}\) such that \(\mathcal{A}_{ \scriptscriptstyle{0} }= \Set{a}\) and \(\mathcal{A}_{ \scriptscriptstyle{1} }= \mathcal{G}^{ \scriptscriptstyle{a} }\).
Then we have a chain of isomorphism of rings 
$$
\bB(\cG) \, \cong \, \bB(\cA) \, \cong \, \bB(G)
$$
where $\bB(G)$ is the classical Burnside ring of the group $G$ introduced in \cite{Solomon:1967} see also \cite{TomDieck:1979}.
\end{proposition}
\begin{proof}
It is  immediately obtained by combining  Theorems \ref{tEquivCatTransGrpd} and  \ref{tEquivCatsIsomBurnRings}.
\end{proof}
The following corollary can be deduced directly from the non canonical equivalence of categories between a given groupoid and a disjoint union of its isotropy group types. It also shows that the Burnside functor, as defined in Definition \ref{def:BF}, does not distinguishes the arrows of  a given groupoid.  
Keep the notations of Theorem \ref{tBurnTransCompProd} and Proposition \ref{pTransGrpdBurnRing}:
\begin{corollary}\label{coro:main}
Given a groupoid \(\mathcal{G}\), we have the following isomorphism of rings:
\[ \bB\left({\mathcal{G}}\right) \cong \prod_{a \, \in \, \rep(\Go)} \bB\left({\mathcal{G}^{ \scriptscriptstyle{a} }  }\right),
\]
where the right hand side term is the product of commutative rings. 
\end{corollary}
\begin{proof}
It follows from Proposition \ref{pTransGrpdBurnRing} and Theorem \ref{tBurnTransCompProd}.
\end{proof}

\begin{remark}
\label{rBurnFree}
It was proved in \cite{Solomon:1967} that the Burnside ring of a group \(G\) is isomorphic to a ring that is a free abelian group over the set of conjugacy classes \(\mathcal{S}_{G} /\simc\).
Therefore, thanks to Corollary~\(\ref{coro:main}\), the Burnside ring of a groupoid is a free abelian group.
\end{remark}

\begin{examples}\label{exam:Bexample}
We expound examples of the Burnside ring of certain groupoids.
\begin{enumerate}[(1)]
\item It clear that if $G=\{\star\}$ is a trivial group, then $\bB(G)$ is  the ring of integers $\mathbb{Z}$. Therefore, the Burnside ring of any groupoid whose isotropy groups are trivial is the product ring  $\mathbb{Z}^{I}$   for some set $I$.    This is the case for instance of all the relation equivalence groupoids expounded in Example \ref{exam:X}. 
\item Let \(G\) be a cyclic group of order a prime number \(p \ge 2\).
Thanks to Remark~\(\ref{rBurnFree}\), we have the isomorphisms of abelian groups \(\bB(G) \cong \mathbb{Z} v \oplus \mathbb{Z} w\), where \(v= \left[{G/G}\right] =\left[{1}\right]\) and \(w =\left[{G/1}\right]= \left[{G}\right]\).
Now we have to study the multiplicative structure of \(\bB(G)\).
It is immediate to see that \(v^2= v\) and \(vw = wv = w\).
Since \(\left\lvert{G \times G}\right\rvert = p^2\), we deduce that either \(w^2= \left[{G \times G}\right] = p^2 v\) or \(w^2 = p w\).
Considering that \(G \times G\) can be decomposed into the \(p\) orbits \(\Set{ \left({a^{i+j}, a^j}\right) | {j}\in \Set{{0},\dots,{p-1}}  }\) for \({i}\in \Set{{0},\dots,{p-1}}\), we obtain \(w^2= pw\).
Now it is easy to deduce that we have the following isomorphism of rings
\[ \bB(G) \cong \frac{\mathbb{Z}[X]}{\Braket{X^2 - pX} }.
\]
\item Given not empty sets \(S_1\), \(S_2\) and \(S_3\), we denote with \(G_1\) the trivial group, with \(G_2\) a cyclic group or order a prime \(p \ge 2\) and with \(G_3\) the alternating group \(A_5\).
We consider the groupoid \(\Cat{G}\) with the following three connected component: \(\trg{S_1}{G_1}\), \(\trg{S_2}{G_2}\) and \(\trg{S_3}{G_3}\).
If follows from Corollary~\(\ref{coro:main}\), the two previous examples and \cite[page 10]{TomDieck:1979}, that we have the following isomorphism of rings
\[ \bB\left(\Cat{G}\right) \cong \bB\left(G_1\right) \times \bB\left(G_2\right) \times \bB\left(G_3\right) \cong \mathbb{Z} \times \frac{\mathbb{Z}[X]}{\Braket{X^2 - pX} } \times R ,
\]
where \(R\) is the Burnside ring of the group \(A_5\) described in \cite[page 10]{TomDieck:1979}.
\end{enumerate}
\end{examples}

\begin{remark}\label{rem:finitecase}
It was proved in \cite{DressSolvGrps} that a group \(G\) is solvable if and only if the prime ideal spectrum of \(\bB(G)\) is connected.
Since it is a known fact that the prime ideal spectrum of a direct product of commutative rings is the disjoint union of their spectrums, we deduce, thanks to Corollary~\(\ref{coro:main}\), that the prime ideal spectrum of the Burnside ring of a groupoid \(\Cat{G}\) is connected if and only if \(\Cat{G}\) is transitive and it has a solvable isotropy group type.
\end{remark}

\begin{remark}
\label{rIsomProdRig}
Now let \(\left(G_j\right)_{j \in J}\) be the connected components of the groupoid \(\Cat{G}\).
Let be \(A = \prod_{l \in J} \lL\left({\Cat{G}_l}\right)\) and \(R= \lL\left({\Cat{G} }\right)\): the families
\[ \left(\lL\left({i_j}\right) \colon  \lL\left({\Cat{G}}\right) \longrightarrow \lL\left({\Cat{G}_j}\right) \right)_{j \in J} 
\qquad \text{and} \qquad
\left(\pi_j \colon \prod_{l \in J} \lL\left({\Cat{G}_l}\right) \longrightarrow \lL\left({\Cat{G}_j}\right) \right)_{j \in J} 
\]
are products in the category \(\mathbf{Rig}\) therefore there are homomorphism of rigs \(f \colon A \longrightarrow R\) and \(h \colon R \longrightarrow A\) such that the following diagrams commute for every \(j \in J\):
\[ \xymatrix{
A \ar[rrd]_{\pi_j}  \ar[rr]^{f} &  & R \ar[d]^{L\left({i_j}\right) }  \\
&& \lL\left({\Cat{G}_j}\right)      }
\qquad \text{and} \qquad
\xymatrix{
R \ar[rrd]_{L\left({i_j}\right) }  \ar[rr]^{h} &  & A \ar[d]^{\pi_j}  \\
&& \lL\left({\Cat{G}_j}\right) .     }
\]
Using the universal property of the product of rings and the notations of the proof of Proposition~\(\ref{pSetsCoprGrpd}\) we obtain that the following homomorphism of rigs
\[ \begin{aligned}{f}  \colon & {\prod_{j \in J} \lL\left({\Cat{G}_j}\right) } \longrightarrow {\lL\left({\Cat{G} }\right) } \\ & {\left(\left[{\left({X_j, \varsigma_j }\right) }\right] \right)_{j \in J} }  \longrightarrow {\left[{ \biguplus_{j \in J} \widehat{\left({X_j, \varsigma_j}\right) } }\right] }\end{aligned} 
\]
and
\[
\begin{aligned}{h}  \colon & {\lL\left({\Cat{G} }\right) } \longrightarrow {\prod_{j \in J} \lL\left({\Cat{G}_j}\right) } \\ & {\left[{\left({X, \varsigma}\right) }\right] }  \longrightarrow { \left(\left[{ \left(  X  \, {{}_{ \scriptscriptstyle{\varsigma} } {\times}_{ \scriptscriptstyle{\Id_{\left({\Cat{G}_j}\right)_0} } }\,} \left({\Cat{G}_j}\right)_0 , \pr_{ \scriptscriptstyle{2} } \right) }\right] \right)_{j \in J } = \left(\left[{ \left(  \varsigma^{-1}\left({\left({\Cat{G}_j}\right)_0}\right) , \left.{\varsigma}\right|_{\varsigma^{-1}\left({\left({\Cat{G}_j}\right)_0}\right)} \right) }\right] \right)_{j \in J }  }\end{aligned}
\]
are isomorphism such that \(h = f^{-1}\).
It is now obvious that \(\gG(f)\) and \(\gG(h)\) are isomorphism of rings between \(\bB(\Cat{G})\) and \(\prod_{j \in J} \bB \left({\Cat{G}_j}\right)\).
\end{remark}

\section{The Burnside algebra of a groupoid and the ghost map}\label{sec:Idem}

In this section we will continue to assume that \(\Cat{G}\) is a groupoid with a finite set of object \(\Cat{G}_{ \scriptscriptstyle{0} }\).
We define the \emph{Burnside algebra of \(\Cat{G}\) over \(\mathbb{Q}\)} as \(\mathbb{Q} \bB\left({\Cat{G}}\right) = \mathbb{Q} \otimes_{\mathbb{Z}} \bB\left({\Cat{G} }\right)\) and, given a group \(G\), its Burnside algebra over \(\mathbb{Q}\) is defined as \(\mathbb{Q} \bB\left({G}\right) = \mathbb{Q} \otimes_{\mathbb{Z}} \bB\left({G }\right)\) (see \cite[Page 31]{Bouc:2010}).
Thanks to Corollary~\(\ref{coro:main}\) we have
\[ \bB\left({\Cat{G}}\right) \cong \prod_{a \, \in \, \rep\left({\Cat{G}_{ \scriptscriptstyle{0} } }\right) } \bB\left({\Cat{G}^{a} }\right) 
\]
and, tensoring with \(\mathbb{Q}\), we obtain, since over a finite set the direct product and the direct sum of \(\mathbb{Z}\)-modules coincide,
\[ \mathbb{Q} \bB\left({\Cat{G} }\right) = \mathbb{Q} \otimes_{\mathbb{Z}} \bB\left({\Cat{G} }\right) 
\cong \mathbb{Q} \otimes_{\mathbb{Z}} \prod_{ a \, \in \, \rep\left({\Cat{G}_{ \scriptscriptstyle{0} } }\right) } \bB\left({\Cat{G}^{a} }\right) 
\cong \prod_{ a \, \in \, \rep\left({\Cat{G}_{ \scriptscriptstyle{0} } }\right) } \left({ \mathbb{Q} \otimes_{\mathbb{Z}}  \bB\left({\Cat{G}^{a} }\right)  }\right) 
= \prod_{a \, \in \, \rep\left({\Cat{G}_{ \scriptscriptstyle{0} } }\right) }  \mathbb{Q}  \bB\left({\Cat{G}^{a} }\right)   .
\]
Therefore also the Burnside algebra \(\mathbb{Q} \bB\left({\Cat{G}}\right)\) is a split semi simple commutative \(\mathbb{Q}\)-algebra, exactly like the Burnside algebra of a group.
As a consequence the idempotents of \(\mathbb{Q} \bB\left({\Cat{G}}\right)\) are in a bijective correspondence with the set of elements \(\left(x_a\right)_{a \, \in \, \rep\left({\Cat{G}_{ \scriptscriptstyle{0} } }\right)}\), where \(x_a\) is an idempotent of \(\mathbb{Q} \bB\left({\Cat{G}^{a} }\right)\) for each \(a \in rep\left({\Cat{G}}\right)\).
We recall that the idempotents of the Burnside algebra \(\mathbb{Q} \bB\left({G}\right)\) of a group \(G\) were completely characterized in \cite[Theorem~\(2.5.2\)]{Bouc:2010}.

\begin{examples}
\begin{enumerate}[(1)]
\item It has been stated in Example \ref{exam:Bexample} that the Burnside ring of the trivial group is \(\mathbb{Z}\) therefore, of course, its Burnside algebra is \(\mathbb{Q}\) whose only idempotents are \(0\) and \(1\).
This implies that the Burnside algebra of a groupoid \(\Cat{G}\) with all isotropy group types trivial and a finite set of objects is \(\prod_{a \, \in \, \rep\left({\Cat{G}_{ \scriptscriptstyle{0} } }\right) }  \mathbb{Q}\). Therefore, this can be applied to any of the groupoids given in Example \ref{exam:X}.
\item Let \(G\) be a cyclic group of order a prime \(p \ge 2\).
Thanks to Example \ref{exam:Bexample} we know that \(\bB(G) \cong \mathbb{Z} v \oplus \mathbb{Z} w\), where 
$$
v= \left[{G/G}\right] =\left[{1}\right], \, w =\left[{G/1}\right]= \left[{G}\right], \, v^2= v, \;  vw = wv = w \text{ and  }  w^2 = pw.
$$
We will use \cite[Theorem~\(2.5.2\)]{Bouc:2010}: since the only subgroups of \(G\) are only \(G\) itself and \(1\), we have that \(\mu(1,1)= \mu(G,G)=1\) and \(\mu(1,G)=-1\) where \(\mu\) is the Moebius function on the poset of subgroups of \(G\).
Applying the quoted theorem and computing, we obtain
\[ e_1^G = \frac{1}{p} \mu(1,1) \left[{\frac{G}{1} }\right] = \frac{1}{p}w
\]
and
\[ e_G^G = \frac{1}{p} \left({\mu(1,G) \left[{\frac{G}{1} }\right] + p \mu(G,G) \left[{\frac{G}{G} }\right] }\right) 
= \frac{-1}{p} w + v,
\]
the two primitive idempotents of \(\mathbb{Q}\bB\left({G}\right)\).
Notice that, in the case of a cyclic group of order \(p\), we can, by abuse of notation, avoid distinguishing a subgroup of \(G\) from its conjugacy class.
Trying to rewrite \(e_1^G\) and \(e_G^G\) with the notations used in this paper we obtain
\[ e_1^G = \frac{1}{p} \otimes_{\mathbb{Z}} \left[{\left[{G}\right], 0}\right]
\]
and
\[ e_G^G = \frac{-1}{p} \otimes_{\mathbb{Z}} \left[{\left[{G}\right], 0}\right] + 1 \otimes_{\mathbb{Z}}  \left[{\left[{1}\right],0}\right] .
\]
\item Now let's consider a groupoid \(\Cat{G}\) with two connected components such that \(\Cat{G}_{ \scriptscriptstyle{0} }=\Set{a, b}\), \(\Cat{G}^{a}\) is a trivial group and \(\Cat{G}^{b}\) is the cyclic group of order \(p\).
We consider the subgroupoids \(\Cat{A}\) and \(\Cat{B}\) such that \(\Cat{A}_{ \scriptscriptstyle{0} }= \Set{a}\), \(\Cat{A}_{ \scriptscriptstyle{1} } = \Set{\iota_{ \scriptscriptstyle{a} }}\), \(\Cat{B}_{ \scriptscriptstyle{0} }= \Set{b}\) and \(\Cat{B}_{ \scriptscriptstyle{1} } = \Set{\iota_{ \scriptscriptstyle{b} }}\).
We denote with \(\varsigma \colon \Cat{A} \longrightarrow \Cat{G}_{ \scriptscriptstyle{0} }\) a structure map with image \(\Set{a}\) and with \(\vartheta \colon \Cat{G}/\Cat{B} \longrightarrow \Cat{G}_{ \scriptscriptstyle{0} }\) and \(\gamma \colon \Cat{G}/\Cat{B}^{(b)} \longrightarrow \Cat{G}_{ \scriptscriptstyle{0} }\) two structures maps with image \(\Set{b}\).
Thanks to Remark~\(\ref{rIsomProdRig}\) we deduce that the Burnside algebra \(\mathbb{Q} \bB\left({\Cat{G}}\right)\) has the following four primitive idempotents:
\[ \begin{aligned}
e_1 &= 1 \otimes_{\mathbb{Z}} \left[{0,0}\right]= 0, \\
e_2 &= 1 \otimes_{\mathbb{Z}} \left[{\left[{(\Cat{A}_{ \scriptscriptstyle{0} }, \varsigma)}\right], 0}\right] ,\\
e_3 &= \frac{1}{p} \otimes_{\mathbb{Z}} \left[{\left[{\left({\frac{\Cat{G} }{\Cat{B} }, \vartheta}\right) }\right], 0}\right] ,\\
e_4 &= \frac{-1}{p} \otimes_{\mathbb{Z}} \left[{\left[{\left({\frac{\Cat{G} }{\Cat{B} }, \vartheta}\right) }\right], 0}\right] + 1 \otimes_{\mathbb{Z}} \left[{\left[{\left({\frac{\Cat{G} }{\Cat{G}^{(b)}  }, \gamma}\right) }\right], 0}\right] .
\end{aligned}
\]
\end{enumerate}
\end{examples}

In the subsequent, we are going to construct the Ghost map of the groupoid \(\Cat{G}\) and prove that is injective.
Let \(\Cat{H} \in rep\left({\Cat{S}_{\Cat{G} } }\right)\) be a subgroupoid of \(\Cat{G}\) with only one object \(a\).
We want to prove that the function
\[ \begin{aligned}{\phi_{\Cat{H} } }  \colon & {\lL \left({\Cat{G} }\right) } \longrightarrow {\mathbb{N}} \\ & {\left[{ \left( X, \varsigma \right) }\right] }  \longrightarrow {\left\lvert{X^{\Sscript{\Cat{H} }} }\right\rvert }\end{aligned} 
\]
is a homomorphism of rigs.
We have \(\phi_{\Cat{H} }\left({ \left[{\emptyset}\right]  }\right)=0\) and \(\phi_{\Cat{H} }\left({\left[{\Cat{G}_{ \scriptscriptstyle{0} }}\right]  }\right)= \left\lvert{\Cat{G}_{ \scriptscriptstyle{0} }^{\Sscript{\Cat{H} }} }\right\rvert= \left\lvert{\Set{a} }\right\rvert=1\).
Given finite right \(\Cat{G}\)-sets \(\left({X,\varsigma}\right)\) and \(\left({Y, \vartheta}\right)\), we calculate
\[ \begin{aligned}
\phi_{\Cat{H} }\left({\left[{X, \varsigma}\right]+ \left[{Y, \vartheta}\right] }\right) 
&= \phi_{\Cat{H} }\left({\left[{X \Uplus Y, \varsigma \Uplus \vartheta}\right] }\right) 
= \left\lvert{\left({X \Uplus Y, \varsigma \Uplus \vartheta}\right)^{\Sscript{\Cat{H} }} }\right\rvert  \\
&= \left\lvert{\left({X, \varsigma}\right)^{\Cat{H} } }\right\rvert  +  \left\lvert{\left({Y, \vartheta}\right)^{\Cat{H} } }\right\rvert 
= \phi_{\Cat{H} }\left({\left[{X, \varsigma}\right] }\right)  + \phi_{\Cat{H} }\left({ \left[{Y, \vartheta}\right] }\right) 
\end{aligned}
\]
and
\[ \begin{aligned}
\phi_{\Cat{H} }\left({\left[{X, \varsigma}\right] \left[{Y, \vartheta}\right] }\right) 
&= \phi_{\Cat{H} }\left({\left[{X \fpro{\Cat{G}_{ \scriptscriptstyle{0} } }  Y, \varsigma  \vartheta}\right] }\right) 
= \left\lvert{\left({X \fpro{\Cat{G}_{ \scriptscriptstyle{0} } }  Y, \varsigma  \vartheta}\right)^{\Sscript{\Cat{H}} } }\right\rvert  
= \left\lvert{\left({ \left({\varsigma \vartheta}\right)^{-1} (a) }\right)^{\Sscript{\Cat{H}} } }\right\rvert \\
&=\left\lvert{\left({ \varsigma^{-1} (a) }\right)^{\Sscript{\Cat{H}} } \times \left( \vartheta^{-1} (a) \right) ^{\Sscript{\Cat{H}} } }\right\rvert 
=\left\lvert{\left({ \varsigma^{-1} (a) }\right)^{\Sscript{\Cat{H}} } }\right\rvert  \left\lvert{ \left( \vartheta^{-1} (a) \right) ^{\Sscript{\Cat{H} }} }\right\rvert \\ 
&= \left\lvert{\left({X, \varsigma}\right)^{\Sscript{\Cat{H}} } }\right\rvert    \left\lvert{\left({Y, \vartheta}\right)^{\Sscript{\Cat{H} }} }\right\rvert 
= \phi_{\Cat{H} }\left({\left[{X, \varsigma}\right] }\right)   \phi_{\Cat{H} }\left({ \left[{Y, \vartheta}\right] }\right) .
\end{aligned}
\]
Applying the Grothendieck functor \(\gro\) we obtain the homomorphism of rings
\[ 
\begin{aligned}{\gro\left({\phi_{\Cat{H} } }\right) }  \colon & {\bB\left({\Cat{G} }\right)= \gro\left({\lL\left({\Cat{G} }\right) }\right) } \longrightarrow {\gro\left({\mathbb{N}}\right) = \mathbb{Z}} \\ & { \left[ \left[{X}\right] , \left[{Y}\right] \right] }  \longrightarrow {\left\lvert{X^{\Sscript{\Cat{H} }} }\right\rvert- \left\lvert{Y^{\Sscript{\Cat{H}} } }\right\rvert .}\end{aligned} 
\]
Using the universal property of the direct product, we are able to define the following homomorphism of rings, which is called the \emph{Ghost map of the groupoid \(\Cat{G}\)}:
\begin{equation}\label{Eq:ghost}
\begin{gathered}
\begin{aligned}{\fk{g}  }  \colon & {\bB\left({\Cat{G} }\right) } \longrightarrow {\prod_{\Cat{H} \, \in  \, \rep \left({\mathcal{S}_{\mathcal{G} } }\right) }  \mathbb{Z}} \\ & { \left[ \left[{X}\right], \left[{Y}\right] \right] }  \longrightarrow {\left(\left\lvert{X^{\Sscript{\Cat{H}} } }\right\rvert- \left\lvert{Y^{\Sscript{\Cat{H}} } }\right\rvert\right)_{\Cat{H} \, \in  \, \rep \left({\mathcal{S}_{\mathcal{G} } }\right) }  .}\end{aligned} 
\end{gathered}
\end{equation}
Now let's suppose that there are \(\left[ \left[{X}\right], \left[{Y}\right] \right] ,  \left[\left[{A}\right], \left[{B}\right]  \right] \in B\left({\Cat{G} }\right)\) such that \(\fk{g}\left({\left[ \left[{X}\right], \left[{Y}\right] \right] }\right)= \fk{g}\left({\left[ \left[{A}\right],\left[{B}\right] \right] }\right)\).
Then that for each \(\Cat{H} \in rep\left({\Cat{S}_{\Cat{G} } }\right)\) we obtain
\[ \left\lvert{\left({X \Uplus B}\right)^{\Sscript{\Cat{H}} } }\right\rvert = \left\lvert{X^{\Sscript{\Cat{H} }} }\right\rvert + \left\lvert{B^{\Sscript{\Cat{H} }} }\right\rvert 
= \left\lvert{A^{\Sscript{\Cat{H} }} }\right\rvert + \left\lvert{Y^{\Cat{H} } }\right\rvert 
= \left\lvert{\left({A \Uplus Y}\right)^{\Sscript{\Cat{H} }} }\right\rvert .
\]
which, thanks to the Burnside Theorem, implies \(X \Uplus B \cong A \Uplus Y\).
As a consequence we have \(\left[{X}\right] + \left[{B}\right] = \left[{A}\right]+ \left[{Y}\right]\), therefore \(\left[ \left[{X}\right], \left[{Y}\right] \right] = \left[ \left[{A}\right], \left[{B}\right] \right]\).  Summing up:
\begin{proposition}\label{prop:fkG}
The Ghost map 
$$
\fk{g}: \bB(\cG) \longrightarrow  {\prod_{\Cat{H} \, \in  \, \rep \left({\mathcal{S}_{\mathcal{G} } }\right) }  \mathbb{Z}}
$$ 
explicitly given by equation \eqref{Eq:ghost}, is actually a monomorphism of rings. 
\end{proposition}

%%%%%%%%%%%%%%%%Appendix%%%%%%%%%%%%%%

\appendix

\section{Laplaza  categories and the Grothendieck functor}\label{asec:laplaza}

\subsection{Laplaza categories and their functors}\label{ssec:Laplaza}

When there are two monoidal structures, \(\oplus\) and \(\boxplus\) on a category one of them, \(\boxplus\), can distribute over the other, that is, given objects \(A\), \(B\) and \(C\), there are natural isomorphisms
\[ A \boxplus (B \diamond C) \cong A \boxplus B \diamond A \boxplus C
\qquad \text{and} \qquad
(A \diamond B) \boxplus C \cong A \boxplus C \diamond B \boxplus C.
\]
This situation was foreshadowed in \cite{LaplazaCoCatAsCoDist} and studied in \cite[page~\(29\)]{LaplazaCohCatDistrib}, where a complete set of coherency conditions is provided.
We will call such a category, for lack of a better name, a \textit{Laplaza category} and we will denote the category of small Laplaza categories by \(\mathbf{LPZCat}\).
An example of a small Laplaza category is the category \(\rset{\cG}\) of the finite right \(\mathcal{G}\)-sets over a groupoid \(\Cat{G}\) where \(\oplus\) is the coproduct, i.e., the disjoint union \(\Uplus\), and \(\boxplus\) is the fibered product $\underset{\Sscript{\Go}}{\times}$.

\begin{definition}\label{defApp:Laplaza}
Let \(\left({\Cat{C}_1, \diamond_1, \boxplus_1}\right)\) and \(\left({\Cat{C}_2, \diamond_2, \boxplus_2}\right)\) be Laplaza categories.
A \emph{Laplaza functor} 
$$
F \colon \left({\Cat{C}_1, \diamond_1, \boxplus_1}\right) \longrightarrow \left({\Cat{C}_2, \diamond_2, \boxplus_2}\right)
$$
is simultaneously both a strong monoidal functor \(F \colon \left({\Cat{C}_1, \diamond_1}\right) \longrightarrow \left({\Cat{C}_2, \diamond_2}\right)\) and a strong monoidal functor \(F \colon \left({\Cat{C}_1, \boxplus_1}\right) \longrightarrow \left({\Cat{C}_2, \boxplus_2}\right)\).
\end{definition}

\begin{definition}
Let \(\left({\Cat{C}_1, \diamond_1, \boxplus_1}\right)\) and \(\left({\Cat{C}_2, \diamond_2, \boxplus_2}\right)\) be Laplaza categories.
A Laplaza functor 
$$
F \colon \left({\Cat{C}_1, \diamond_1, \boxplus_1}\right) \longrightarrow \left({\Cat{C}_2, \diamond_2, \boxplus_2}\right) 
$$
is said to be an \emph{isomorphism of Laplaza categories} if it is an isomorphism of categories and the inverse functor \(F^{-1}\) is also a Laplaza functor.
\end{definition}

\begin{definition}
Given two Laplaza categories \(\left({\Cat{C}_1, \diamond_1, \boxplus_1}\right)\) and \(\left({\Cat{C}_2, \diamond_2, \boxplus_2}\right)\) (the units are omitted for brevity), let
\[ F, \, G \colon \left({\Cat{C}_1, \diamond_1, \boxplus_1}\right) \longrightarrow \left({\Cat{C}_2, \diamond_2, \boxplus_2}\right)
\]
be two strong monoidal functors.
A \textbf{Laplaza natural transformation}\index{Laplaza!natural transformation}\index{Natural transformation!Laplaza} (respectively, a \textbf{Laplaza natural isomorphism})\index{Laplaza!natural isomorphism}\index{Natural isomorphism!Laplaza}
\[ \varphi \colon  F  \longrightarrow G \colon \left({\Cat{C}_1, \diamond_1, \boxplus_1}\right) \longrightarrow \left({\Cat{C}_2, \diamond_2, \boxplus_2}\right)
\]
is simultaneously both a monoidal natural transformation (respectively, a monoidal natural isomorphism)
\[ \varphi \colon  F  \longrightarrow G \colon \left({\Cat{C}_1, \diamond_1}\right) \longrightarrow \left({\Cat{C}_2, \diamond_2}\right)
\]
and a monoidal natural transformation (respectively, a monoidal natural isomorphism)
\[ \varphi \colon  F  \longrightarrow G \colon \left({\Cat{C}_1,  \boxplus_1}\right) \longrightarrow \left({\Cat{C}_2, \boxplus_2}\right).
\]
\end{definition}

\begin{definition}
Given \(\left({\Cat{C}_1, \diamond_1, \boxplus_1}\right)\) and \(\left({\Cat{C}_2, \diamond_2, \boxplus_2}\right)\) two Laplaza categories, we say that a \textbf{Laplaza adjunction}\index{Laplaza!adjunction}\index{Adjunction!Laplaza} is a couple of Laplaza functors
\[ F \colon \left({\Cat{C}_1, \diamond_1, \boxplus_1}\right) \longrightarrow \left({\Cat{C}_2, \diamond_2, \boxplus_2}\right)
\qquad \text{and} \qquad
G \colon \left({\Cat{C}_2, \diamond_2, \boxplus_2}\right) \longrightarrow \left({\Cat{C}_1, \diamond_1, \boxplus_1}\right)
\]
such that there are Laplaza transformations
\[ \eta \colon \Id_{\Cat{C}_1} \longrightarrow GF
\qquad \text{and} \qquad
\varepsilon \colon FG \longrightarrow \Id_{\Cat{C}_2}
\]
sucth that \(\varepsilon F \circ F \eta = \Id_{F}\) and \(G\varepsilon \circ \eta G = G\).
\end{definition}

\begin{definition}
Given \(\left({\Cat{C}_1, \diamond_1, \boxplus_1}\right)\) and \(\left({\Cat{C}_2, \diamond_2, \boxplus_2}\right)\) two Laplaza categories, let
\[ F \colon \left({\Cat{C}_1, \diamond_1, \boxplus_1}\right) \longrightarrow \left({\Cat{C}_2, \diamond_2, \boxplus_2}\right)
\]
be a Laplaza functor.
We say that \(F\) realizes a \textbf{Laplaza equivalence of categories}\index{Laplaza!equivalence of categories}\index{Equivalence of categories!Laplaza} if there is a Laplaza functor
\[ G \colon \left({\Cat{C}_2, \diamond_2, \boxplus_2}\right) \longrightarrow \left({\Cat{C}_1, \diamond_1, \boxplus_1}\right)
\]
and Laplaza natural isomorphisms
\[ \eta \colon \Id_{\Cat{C}_1} \longrightarrow GF
\qquad \text{and} \qquad
\varepsilon \colon FG \longrightarrow \Id_{\Cat{C}_2}.
\]
\end{definition}

\begin{corollary}\label{cAdjEquivInvLaplaza}
Let be \(\left({\Cat{D}, \otimes, I}\right)\) and \(\left({\Cat{C}, \boxplus, J}\right)\) be Laplaza categories and let \(F \colon \Cat{C} \longrightarrow \Cat{D}\) be a Laplaza functor which is an ordinary equivalence of categories.
Then there is a Laplaza functor \(G \colon \Cat{D} \longrightarrow \Cat{C}\) such that there are Laplaza isomorphisms 
\[ \eta \colon \Id_{\Cat{C} } \longrightarrow GF
\qquad \text{and} \qquad
\varepsilon \colon FG \longrightarrow \Id_{\Cat{D}}
\]
and such that \((F,G)\) realises an adjunction with unit \(\eta\) and counit \(\varepsilon\).
In particular, the inverse of a Laplaza functor is a Laplaza functor.
\end{corollary}

\subsection{The Category \textbf{Rig}}\label{ssec:Rig} 
One of the essential notion to introduce a Burnside ring is that of \emph{rig}.
\begin{definition}\label{def:rig}
Let \(S\) be a set with two associative and commutative internal operations  \(\cdot\) and \(+\).
We call \(S\) a \emph{rig} if the following conditions are satisfied:
\begin{enumerate}
\item \(+\) has a neutral element \(0\);
\item \(\cdot\) has a neutral element \(1\);
\item \(\cdot\) distributes over \(+\) on the right and on the left that is, for each \(a, b, c \in S\),
\[ a(b+c)=ab+ac
\qquad \text{and} \qquad
(a+b)c = ac+ ab;
\]
\item \(S\) respect the absorption/annihilation laws that is, for each \(a \in S\) we have \(a \cdot 0 = 0 = 0 \cdot a\).
\end{enumerate}
A homomorphism of rigs \(f \colon S \longrightarrow T\) is a function which is a homomorphism of monoids both as \(f \colon (S, +) \longrightarrow (T, +)\) and as \(f \colon (S, \cdot) \longrightarrow (T, \cdot)\).
The category of rigs will be denoted by \(\mathbf{Rig}\).
\end{definition}

With this definition we choose to follow the definitions given by \cite[page~\(7\)]{GlaLitSemiringsInfSci}, and \cite[page~\(379\)]{SchaNegEulCharDim}.
The reader should know, however, that what we called a rig is called a semiring by other authors (\cite[page~\(1\)]{GolSemRingsAppl}).
Nevertheless, in analogy with the word semigroup which describes a monoid without a neutral element, we think that the word semiring should be reserved to a ring which lacks both the negative elements (i.e., the inverses wrt to the addition) and the additive neutral element.

\begin{proposition}
\label{pProdRigs}
Given a family of rigs \(\left(S_i\right)_{i \in I}\), set \(S= \prod_{i \in I} S_i\) and let \(\pi_i \colon S \longrightarrow S_i\) be the canonical projection.
Then \(\left(\pi_i \colon S \longrightarrow S_i\right)_{ i \in I}\) is the product of the family \(\left(S_i\right)_{ i \in I}\) in the category \(\mathbf{Rig}\).
\end{proposition}
\begin{proof}
Given a rig \(A\), let \(\left(f_j \colon A \longrightarrow S_j\right)_{j \in I}\) be a family of rigs.
We have to prove that there is only one morphism \(f \colon A \longrightarrow S\) such that the following diagram commutes for every \(j \in I\):
\[ \xymatrix{
A   \ar[rr]^{f} \ar[rrd]_{f_j}  &  & S \ar[d]^{\pi_j} \\
&& S_j .    }
\]
For each \(a \in A\) we define \(f(a)=\left(f_j(a)\right)_{j \in I}\): obviously, \(f\) is a homomorphism of rigs because so is \(f_j\) for every \(j \in I\).
Regarding the commutativity of the diagram, for every \(j \in I\) and for every \(a \in A\) we have:
\[ \pi_j\left(f(a)\right) = \pi_j \left({\left(f_i(a)\right)_{i \in I} }\right) 
= f_j(a).
\]
Now let \(g \colon A \longrightarrow S\) be another morphism of rigs such that, for every \(j \in I\), the following diagram commutes:
\[ \xymatrix{
A   \ar[rr]^{g} \ar[rrd]_{f_j}  &  & S \ar[d]^{\pi_j} \\
&& S_j .    }
\]
Then for every \(j \in I\) and \(a \in A\) we have
\[ \pi_j\left({g(a)}\right) = f_j(a)= \pi_j\left({f(a)}\right) 
\]
thus \(f(a)=g(a)\) and \(f=g\).
\end{proof}

\subsection{The Grothendieck functor}\label{ssec:GrothF}
We will denote by \(\gG\) the \emph{Grothendieck functor} which sends a rig \(S\) to the ring \(\gG(S)\) constructed as follows.
We define a equivalence relation \(\sim\) on \(S \times S\) such that for every \((a, b), (c,d) \in S \times S\), \((a,b) \sim (c,d)\) if and only if there is \(e \in S\) such that \(a + d + e = c + b + e\).
The equivalence class of the couple \((a,b) \in S \times S\) will be denoted with \(\left[{(a,b)}\right]\), or simply by \([a,b]\) to make the notations more clear, and the quotient set of \(S \times S\) with \(\gG(S)\).
We will define an addition and a multiplication on \(\gG(S)\) as follows: for every \([a,b], [c,d] \in \gG(S)\), 
\[ [a,b]+ [c,d] = [a+c, b+d]
\qquad \text{and} \qquad
[a,b] \cdot [c,d] = [ ac + bd, ad + bc ].
\]
In this way \(\gG(S)\) becomes a commutative rings with \([0,0]\) as neutral element with respect to \(+\) and \([1,0]\) as neutral element with respect to \(\cdot\).

Given a rig \(S\), the ring \(\gG(S)\) has the following universal property.

\begin{proposition}\label{pUnivPropGrothFu}
Given a rig \(S\), for any ring \(H\) and for any homomorphism of rigs \(\psi \colon S \longrightarrow H\), there is a unique homomorphism of rings \(\theta \colon \gG(S) \longrightarrow H\) such that \(\psi = \theta \varphi\), that is, such that the following diagram is commutative:
\[ \xymatrix{
S \ar[d]_{\varphi}  \ar[rr]^{\psi} &  & H  \\
\gG(S).  \ar@{.>}[rru]_{\theta} & &   }
\]
\end{proposition}

Using the universal property of Proposition~\(\ref{pUnivPropGrothFu}\), given a homomorphism of rigs \(f \colon S \longrightarrow T\) we can define
\[ \begin{aligned}{\gG(f)}  \colon & {\gG(S)} \longrightarrow {\gG(T)} \\ & {[a,b]}  \longrightarrow {\left[{f(a), f(b)}\right]. }\end{aligned} 
\]
It is possible to prove that \(\gG(f)\) is a homomorphism of rings and that, with these definitions, \(\gG\) becomes a covariant functor from the category of rigs \(\mathbf{Rig}\) to the category of commutative rings \(\mathbf{CRing}\).

\begin{proposition}
\label{pGrotFuIsoToIso}
Given an isomorphism of rigs \(f \colon S \longrightarrow T\) we obtain an isomorphism of rings
\[ 
\gG(S) \colon \gG(S) \longrightarrow \gG(T).
\]
\end{proposition}
\begin{proof}
Immediate.
\end{proof}

\begin{proposition}\label{pGrothFuPresProd.}
The Grothendieck functor \(\gG\) preserves all products.
In particular, given a family of rigs  \(\left(S_j \right)_{j \in I}\), let be \(\left(\pi_j \colon S \longrightarrow S_j\right)_{j \in I }\) their product in \(\mathbf{Rig}\).
Then
\[ \left(\gG\left({\pi_j}\right) \colon \gG(S) \longrightarrow \gG\left({S_j}\right) \right)_{j \in I} 
\]
is the product of the family \(\left(\gG\left({S_j}\right)\right)_{j \in I}\) in \(\mathbf{CRing}\).
\end{proposition}
\begin{proof}
Given a ring \(A\), let \(\left(A \longrightarrow G\left({S_j}\right) \right)_{j \, \in \,  I}\) be a family of morphisms in \(\mathbf{CRing}\).
We have to prove that there is a unique homomorphism of rings \(f \colon A \longrightarrow \gG(S)\) such that for every \(j \in I\) the following diagram commutes:
\[ \xymatrix{
A   \ar[rr]^{f} \ar[rrd]_{f_j}  &  & \gG(S) \ar[d]^{\gG\left({\pi_j}\right) } \\
&& \gG\left({S_j}\right) .    }
\]
Thanks to Proposition~\(\ref{pProdRigs}\), we will assume that \(S = \prod_{j \in I} S_j\) and that \(\pi_j \colon S \longrightarrow S_j\) is the canonical projection for every \(j \in I\) (the categorical product is unique up to isomorphism in every category so there is no loss of generality in this choice).
Let \(a \in A\): for every \(j \in I\) there are \(x_j, y_j \in S_j\) such that \(f_j(a)= \left[{x_j, y_j}\right]\) thus we can define
\[ f(a) = \left[{\left(x_j\right)_{j \in I}, \left(y_j\right)_{ j \in J} }\right] .
\]
We have to prove that this is a good definition.
For every \(j \in I\) let be \(z_j, w_j \in S\) such that \(\left[{x_j, y_j}\right] = \left[{z_j, w_j}\right]\): then there is \(e_j \in S\) such that \(x_j + w_j + e_j = z_j + y_j + e_j\).
As a consequence we have
\[ \left(x_j\right)_{j \in I} + \left(w_j\right)_{j \in I} + \left(e_j\right)_{j \in I} = \left(z_j\right)_{j \in I} + \left(y_j\right)_{j \in I} + \left(e_j\right)_{j \in I} 
\]
thus
\[ \left[{\left(x_j\right)_{j \in I}, \left(y_j\right)_{j \in I} }\right] = \left[{\left(z_j\right)_{j \in I}, \left(w_j\right)_{j \in I} }\right] 
\]
and \(f\) is well defined.

Now we have to prove that \(f\) is a homomorphism of rings.
Given \(a,b \in A\), for every \(j \in I\) let be \(a_j, \alpha_j, b_j, \beta_j \in S_j\) such that \(f_j(a)= \left[{a_j, \alpha_j}\right]\) and \(f_j(b)= \left[{b_j, \beta_j}\right]\).
We have
\[ f_j(a+b) = f_j(a) + f_j(b) 
= \left[{a_j, \alpha_j}\right] + \left[{b_j , \beta_j}\right] 
= \left[{a_j + b_j, \alpha_j + \beta_j}\right] 
\]
and
\[ f_j(ab) = f_j(a)  f_j(b) 
= \left[{a_j, \alpha_j}\right] \left[{b_j , \beta_j}\right] 
= \left[{a_j  b_j + \alpha_j \beta_j, a_j \beta_j + \alpha_j b_j}\right] 
\]
thus
\[ \begin{aligned}
f(a) + f(b) &= \left[{\left(a_j\right)_{j \in I}, \left(\alpha_j\right)_{j \in I} }\right] + \left[{\left(b_j\right)_{j \in I}, \left(\beta_j\right)_{j \in I} }\right]  \\
&= \left[{\left(a_j\right)_{j \in I} + \left(b_j\right)_{j \in I}, \left(\alpha_j\right)_{j \in I} + \left(\beta_j\right)_{j \in I} }\right] \\
&= \left[{\left(a_j + b_j\right)_{j \in I}, \left(\alpha_j + \beta_j\right)_{j \in I} }\right] \\
&= f(a+b)
\end{aligned}
\]
and
\[ \begin{aligned}
f(a) f(b) &= \left[{\left(a_j\right)_{j \in I}, \left(\alpha_j\right)_{j \in I} }\right]  \left[{\left(b_j\right)_{j \in I}, \left(\beta_j\right)_{j \in I} }\right]  \\
&= \left[{\left(a_j\right)_{j \in I}  \left(b_j\right)_{j \in I} + \left(\alpha_j\right)_{j \in I} \left(\beta_j\right)_{j \in I} , \left(a_j\right)_{j \in I}  \left(\beta_j\right)_{j \in I} + \left(\alpha_j\right)_{j \in I} \left(b_j\right)_{j \in I}}\right] \\
&= \left[{\left(a_j  b_j + \alpha_j \beta_j \right)_{j \in I}, \left(a_j \beta_j + \alpha_j b_j\right)_{j \in I} }\right] \\
&= f(ab).
\end{aligned}
\]
Moreover, for each \(j \in I\) we have \(f_j(0)=[0,0]\) and \(f_j(1)=[1,0]\) thus
\[ f(0)=\left[{\left(0_j\right)_{j \in I}, \left(0_j\right)_{j \in I}}\right]
\qquad \text{and} \qquad
f(1)=\left[{\left(1_j\right)_{j \in I}, \left(0_j\right)_{j \in I}}\right]
\]
therefore we have proved that \(f\) is a homomorphism of rings.
Regarding the commutativity of the diagrams, for every \(j \in I\) and every \(a \in A\) let be \(x_j, y_j \in S_j\) such that \(f_j(a)=\left[{x_j, y_j}\right]\).
Then \(f(a)=\left[{\left(x_i\right)_{i \in I}, \left(y_i\right)_{i \in I} }\right]\) thus
\[ G\left({\pi_j}\right)\left({f(a)}\right) = \left[{ \pi_j\left(x_i\right)_{i \in I}, \pi_j\left(y_i\right)_{i \in I} }\right]
= \left[{x_j, y_j}\right] = f_j
\]
therefore \(G\left({\pi_j}\right) f = f_j\) and the commutativity of the diagrams is proved.

Now let be \(g \colon A \longrightarrow G(S)\) another homomorphism of rings such that, for every \(j \in I\), the following diagram commutes:
 \[ \xymatrix{
A   \ar[rr]^{g} \ar[rrd]_{f_j}  &  & G(S) \ar[d]^{G\left({\pi_j}\right) } \\
&& G\left({S_j}\right) .    }
\]
For every \(j \in I\) and for every \(a \in A\) we have
\[ G\left({\pi_j}\right)\left({g(a)}\right) = f_j(a)
= G\left({\pi_j}\right)\left({f(a)}\right) 
\]
Let be \(x_j, y_j \in S\) such that \(f_j(a) = \left[{x_j, y_j}\right]\) and let be \(\left(e_i\right)_{i \in I}, \left(f_i\right)_{i \in I} \in S\) such that \(g(a)=\left[{\left(e_i\right)_{i \in I}, \left(f_i\right)_{i \in i}}\right]\).
We calculate:
\[ \begin{aligned}
 \left[{e_j, f_j}\right] &= \left[{\pi_j\left[{\left(e_i\right)_{i \in i} }\right], \pi_j\left({\left(f_i\right)_{i \in I} }\right) }\right] 
= G\left({\pi_j}\right)\left({g(a)}\right) 
= G\left({\pi_j}\right)\left({f(a)}\right)  \\
&= G\left({\pi_j}\right) \left( \left[{\left(x_i\right)_{i \in I}, \left(y_i\right)_{i \in I} }\right]  \right)
= \left[{\pi_j\left({\left(x_i\right)_{i \in I} }\right), \pi_j\left({\left(y_i\right)_{i \in I} }\right) }\right] 
= \left[{x_j, y_j}\right]
\end{aligned}
\]
thus we obtain that there is \(\varepsilon_j \in S_j\) such that
\[ e_j + y_j + \varepsilon_j = x_j + f_j + \varepsilon_j
\]
therefore
\[ \left(e_j\right)_{j \in I} + \left(y_j\right)_{j \in I} + \left(\varepsilon_j\right)_{j \in I} = \left(x_j\right)_{j \in I} + \left(f_j\right)_{j \in I} + \left(\varepsilon_j\right)_{j \in I} 
\]
and \(g(a)= \left[{\left(x_i\right)_{i \in I}, \left(y_i\right)_{i \in I} }\right] = f(a)\).
We have now proved that \(f=g\).
\end{proof}

\bigskip
\noindent{}
\textbf{Acknowledgement.} 
The second author wants to thank the departement of algebra of the University of Granada, Campus of Ceuta, for the marvellous experience, the friendly and welcoming ambience, as well as for the partial support of the grant MTM2016-77033-P, which enabled him to undertake the long journey there.

\end{document}